\documentclass{ipi}
\usepackage{amsmath}
\usepackage{url}
  \usepackage{paralist}
  \usepackage{graphics} %% add this and next lines if pictures should be in esp format
  \usepackage{epsfig} %For pictures: screened artwork should be set up with an 85 or 100 line screen
\usepackage{graphicx}  \usepackage{epstopdf}%This is to transfer .eps figure to .pdf figure; please compile your paper using PDFLeTex or PDFTeXify.
 \usepackage[colorlinks=true]{hyperref}
\hypersetup{urlcolor=blue, citecolor=red}

\def\currentvolume{14}
 \def\currentissue{5}
  \def\currentyear{2020}
   
    \def\ppages{X--XX}
     \def\DOI{10.3934/ipi.2020043}

\newtheorem{theorem}{Theorem}[section]

\newtheorem{lemma}[theorem]{Lemma}
\newtheorem{proposition}{Proposition}

\theoremstyle{definition}

\newtheorem{remark}{Remark}

\usepackage{mathptmx,amsmath,amssymb,amsfonts,amsthm}
\usepackage{graphicx,color,xcolor,pictex,epstopdf,url,algorithm,algorithmic,caption,endnotes}
\let\oldref\ref
\renewcommand{\ref}[1]{(\oldref{#1})}  % stupid kludge for laTeX
\renewcommand{\eqref}[1]{(\oldref{#1})} %BB03dec18
\input colordvi
\font\tenrm=cmr10
\font\teni=cmmi10 \skewchar\teni='177
\font\tensy=cmsy10 \skewchar\tensy='60
\font\tenex=cmex10
\font\tenit=cmti10
\font\tensl=cmsl10
\font\tenbf=cmbx10
\font\tentt=cmtt10
\font\ninerm=cmr9
\font\ninei=cmmi9 \skewchar\ninei='177
\font\ninesy=cmsy9 \skewchar\ninesy='60
\font\nineit=cmti9
\font\ninesl=cmsl9
\font\ninebf=cmbx9
\font\ninett=cmtt9
\font\eightrm=cmr8
\font\eighti=cmmi8 \skewchar\eighti='177
\font\eightsy=cmsy8 \skewchar\eightsy='60
\font\eightit=cmti8
\font\eightsl=cmsl8
\font\eightbf=cmbx8
\font\eighttt=cmtt8
\font\sevenrm=cmr7
\font\seveni=cmmi7 \skewchar\seveni='177
\font\sevensy=cmsy7 \skewchar\sevensy='60
\font\sevenbf=cmbx7
\font\sevenit=cmmi7
\font\sevensl=cmmi7
\font\seventt=cmr7
\font\sixrm=cmr6
\font\sixi=cmmi6 \skewchar\sixi='177
\font\sixsy=cmsy6 \skewchar\sixsy='60
\font\sixbf=cmbx6
\font\fiverm=cmr5
\font\fivei=cmmi5 \skewchar\fivei='177
\font\fivesy=cmsy5 \skewchar\fivesy='60
\font\fivebf=cmbx5
\def\tenpoint{\def\rm{\fam0\tenrm}%
        \textfont0=\tenrm \scriptfont0=\sevenrm \scriptscriptfont0=\fiverm
        \textfont1=\teni \scriptfont1=\seveni \scriptscriptfont1=\fivei
        \textfont2=\tensy \scriptfont2=\sevensy \scriptscriptfont2=\fivesy
        \textfont3=\tenex \scriptfont3=\tenex \scriptscriptfont3=\tenex
        \def\it{\fam\itfam\tenit}%
        \textfont\itfam=\tenit
        \def\sl{\fam\slfam\tensl}%
        \textfont\slfam=\tensl
        \def\bf{\fam\bffam\tenbf}%
        \textfont\bffam=\tenbf \scriptfont\bffam=\sevenbf
                \scriptscriptfont\bffam=\fivebf
        \def\tt{\fam\ttfam\tentt}%
        \textfont\ttfam=\tentt
        \normalbaselineskip=12pt%
        \let\sc=\eightrm        % Small caps
        \setbox\strutbox=\hbox{\vrule height8.5pt depth3.5pt width0pt}%
        \normalbaselines\rm}
\def\ninepoint{\def\rm{\fam0\ninerm}%
        \textfont0=\ninerm \scriptfont0=\sixrm \scriptscriptfont0=\fiverm
        \textfont1=\ninei \scriptfont1=\sixi \scriptscriptfont1=\fivei
        \textfont2=\ninesy \scriptfont2=\sixsy \scriptscriptfont2=\fivesy
        \textfont3=\tenex \scriptfont3=\tenex \scriptscriptfont3=\tenex
        \def\it{\fam\itfam\nineit}%
        \textfont\itfam=\nineit
        \def\sl{\fam\slfam\ninesl}%
        \textfont\slfam=\ninesl
        \def\bf{\fam\bffam\ninebf}%
        \textfont\bffam=\ninebf \scriptfont\bffam=\sixbf
                \scriptscriptfont\bffam=\fivebf
        \def\tt{\fam\ttfam\ninett}%
        \textfont\ttfam=\ninett
        \normalbaselineskip=11pt%
        \let\sc=\sevenrm        % Small caps
        \setbox\strutbox=\hbox{\vrule height8pt depth3pt width0pt}%
        \normalbaselines\rm}
\def\eightpoint{\def\rm{\fam0\eightrm}%
        \textfont0=\eightrm \scriptfont0=\sixrm \scriptscriptfont0=\fiverm
        \textfont1=\eighti \scriptfont1=\sixi \scriptscriptfont1=\fivei
        \textfont2=\eightsy \scriptfont2=\sixsy \scriptscriptfont2=\fivesy
        \textfont3=\tenex \scriptfont3=\tenex \scriptscriptfont3=\tenex
        \def\it{\fam\itfam\eightit}%
        \textfont\itfam=\eightit
        \def\sl{\fam\slfam\eightsl}%
        \textfont\slfam=\eightsl
        \def\bf{\fam\bffam\eightbf}%
        \textfont\bffam=\eightbf \scriptfont\bffam=\sixbf
                \scriptscriptfont\bffam=\fivebf
        \def\tt{\fam\ttfam\eighttt}%
        \textfont\ttfam=\eighttt
        \normalbaselineskip=9pt%
        \let\sc=\sixrm  % Small caps
        \setbox\strutbox=\hbox{\vrule height7pt depth2pt width0pt}%
        \normalbaselines\rm}
\def\sevenpoint{\def\rm{\fam0\sevenrm}%
        \textfont0=\sevenrm \scriptfont0=\fiverm \scriptscriptfont0=\fiverm
        \textfont1=\seveni \scriptfont1=\fivei \scriptscriptfont1=\fivei
        \textfont2=\sevensy \scriptfont2=\fivesy \scriptscriptfont2=\fivesy
        \textfont3=\tenex \scriptfont3=\tenex \scriptscriptfont3=\tenex
        \def\it{\fam\itfam\sevenit}%
        \textfont\itfam=\sevenit
        \def\sl{\fam\slfam\sevensl}%
        \textfont\slfam=\sevensl
        \def\bf{\fam\bffam\sevenbf}%
        \textfont\bffam=\sevenbf \scriptfont\bffam=\fivebf
                \scriptscriptfont\bffam=\fivebf
        \def\tt{\fam\ttfam\seventt}%
        \textfont\ttfam=\seventt
        \normalbaselineskip=8pt%
        \let\sc=\fiverm  % Small caps
        \setbox\strutbox=\hbox{\vrule height6pt depth2pt width0pt}%
        \normalbaselines\rm}

\newbox\boxaddrone \newbox\boxaddrtwo

\title[Simulatenous recovery of diffusion coefficient and reaction term]
      {On the simultaneous recovery of the conductivity
and the nonlinear reaction term in a parabolic equation}

\author[Barbara Kaltenbacher and William Rundell]{}


 \email{barbara.kaltenbacher@aau.at}
 \email{rundell@math.tamu.edu}

\thanks{The first author is supported by FWF grant P30054; the second author is supported by NFS grant {\sc dms}-1620138}

\thanks{$^*$ Corresponding author: William Rundell}

\begin{document}
\maketitle

% Enter the first author's name and address:
\centerline{\scshape Barbara Kaltenbacher}
\medskip
{\footnotesize
% please put the address of the first author
 \centerline{Department of Mathematics,
Alpen-Adria-Universit\"at Klagenfurt}
%   \centerline{Other lines}
   \centerline{9020 Klagenfurt, Austria}
} % Do not forget to end the {\footnotesize by the sign }

\medskip

\centerline{\scshape William Rundell$^*$ }
\medskip
{\footnotesize
 % please put the address of the second  and third author
 \centerline{Department of Mathematics,
Texas A\&M University}
%   \centerline{Other lines}
   \centerline{College Station, Texas 77843, USA}
}

\bigskip

% The name of the associate editor will be entered by an editorial staff
% "Communicated by the associate editor name" is not needed for special issue.
 \centerline{(Communicated by Fioralba Cakoni)}

%The abstract of your paper
\begin{abstract}
This paper considers the inverse problem of recovering both the
unknown, spatially-dependent conductivity $a(x)$ and the nonlinear reaction
term $f(u)$ in a reaction-diffusion  equation
from overposed data.
These measurements can  consist of:
the value of two different solution measurements taken at a later time $T$;
time-trace profiles from two solutions;
or both final time and time-trace measurements from a single forwards
solve  data run.
We prove both uniqueness results and the convergence of iteration
schemes designed to recover these coefficients.
The last section of the paper shows numerical reconstructions based on these
algorithms.
\end{abstract}

%The title of your section 1
\section{Introduction}\label{sect:introduction}

In their most basic form reaction-diffusion equations are nonlinear
parabolic equation of the type
$ u_t(x,t) + \mathbb{L}u = f(u)$ where $\,\mathbb{L}$ is an elliptic operator
defined on a domain $\Omega\subset \mathbb{R}^d$, $d\geq 1$
and $f(u)$ is a function only of the state or dependent variable $u$
which is defined on $\Omega\times [0,t]$ for some $T>0$.

Such equations arise in a wide variety of applications
and two examples of these models are in ecology where $u$
represents the population
density at a fixed point $x$ and time $t$ and $f(u)$ is frequently taken
to be quadratic in $u$ as in the Fisher model;
or in chemical reactions where $f$ is cubic in
the case of the combustion theory of Zeldovich and Frank-Kamenetskii,
\cite{Grindrod:1996, Murray:2002,Kuttler:17}.
In these models $\mathbb{L}$ is of second order but higher order operators
also feature in significant applications, for example in the Cahn-Hilliard
equation~\cite{CahnHillard:1958,ElliottSungmu:1986}.

The physical interpretation of these equations is the time rate
of change of  $u$ is a sum of two components:
the diffusion term given by $\mathbb{L}u$ and the driving term $f(u)$.
Often these are in competition, the former tending to disperse the value
of $u$ and the latter, if $f>0$, to increase it.
If $f$ is negative or positive but sublinear, then on a bounded domain
the solution will usually decay exponentially (but depending on the boundary
conditions/values and the eigenvalues of the combined operator $\mathbb{L}$
being positive).
At the other extreme if the growth rate of $f$ is sufficiently large
then global existence in time of the solution can fail.
This is a complex but well-documented situation, see for example,
\cite{AronsonWeinberger:1978,Levine:1990} and the references within.
More general references for semilinear parabolic equations
are~\cite{Friedman:1964,Friedman:1969}.

In most applications one assumes these terms in the equation are known,
that is, the coefficients in the elliptic operator and the exact form
of $f(u)$.  Indeed, frequently, $f(u)$ is taken to be a polynomial
or possibly a rational function of low degree so that only a few parameters
have to be determined by a least squares fit to the data.
Even if the diffusion constant has also to be determined the method remains
fundamentally the same.
In this paper we make no such assumptions: both our diffusion coefficient and
the nonlinear term can be arbitrary functions within a given smoothness class.
We shall take $\mathbb{L}u = -\nabla\cdot\bigl(a(x)\nabla u\bigr)$ where
the diffusion coefficient $a(x)$ is unknown.

With this background setting our basic model equation is thus
\begin{equation}\label{eqn:basic_pde_parabolic}
u_t(x,t) -\nabla\cdot(a(x)\nabla u(x,t)) = f(u) + r(x,t,u)
\end{equation}
where $r(x,t,u)$ is a known forcing function.
Boundary conditions for \eqref{eqn:basic_pde_parabolic}
will be of the impedance form
\begin{equation}\label{eqn:bdry_cond}
a\frac{\partial u}{\partial\vec{n}} + \gamma u = {b},\qquad
x\in\partial\Omega,\quad t>0,
\end{equation}
with possibly space dependent but known coefficient $\gamma$ (where $\gamma=\infty$ indicates the case of Dirichlet boundary conditions) as well as given space and time dependent right hand side $b$.
We impose the initial condition
\begin{equation}\label{eqn:init_cond}
u(x,0) = u_0(x),\qquad x\in\Omega.
\end{equation}
Now of course the recovery of the coefficient $a$ and the term $f(u)$ requires
over-posed data and we shall consider several different forms and combinations.

The two basic forms are final time and boundary measurements of $u$.
The former consists of the value $g(x) = u(x,T)$ at some later time $T$
and corresponds to such conditions as a thermal or concentration  map of $u$
within the region $\Omega$ or census data taken at a fixed time.
The latter consists of a time trace $h(t) = u(\tilde x,t)$ for some fixed
point $\tilde x$ on $\partial\Omega$.

Combinations include measurements of one of these for two different
sets of initial / boundary conditions corresponding to solutions $u(x,t$ and
$v(x,t)$ or a measurement of both of these for a single data run.
As we will see, some of these combinations are superior to others and the
difference can be substantial.

The main technique we will use follows that of previous work by the
authors and others:  we project the differential equation onto the
boundaries containing the overposed data and obtain nonlinear equations
for our unknowns $a(x)$ and $f(u)$ which we then solve by iterative methods.
This idea for recovering the nonlinearity $f(u)$ in equations such as
\eqref{eqn:basic_pde_parabolic}
was first used in \cite{PilantRundell:1986,PilantRundell:1987}
under time-trace data.
References for recovering coefficients in parabolic equations in general,
but specifically including final-time data,
are~\cite{Isakov:1991,Isakov:2006}.
A recent such problem for the subdiffusion case based on fractional derivatives
is found in~\cite{ZhangZhi:2017}.

More recently, the authors have considered several inverse problems
based on the above equation.
In \cite{KaltenbacherRundell:2019b} the equation taken was
$u_t - \triangle u = q(x)g(u)$ where $g(u)$ was assumed known and $q(x)$
which determines the unknown spatial strength of $g$, had to be recovered
from final time data.
In \cite{KaltenbacherRundell:2019c} the equation taken was
$u_t - \triangle u = f(u)$ and recovery of the unknown $f(u)$ was again from
final time data.
The case of $\mathbb{L}u = -\nabla\cdot\bigl(a(x)\nabla u\bigr) + q(x)u$
where both  $a(x)$ and $q(x)$ are unknown and have to be determined
from a pair of runs each providing final-time data was considered in
\cite{KaltenbacherRundell:2020a}.
In this work the equation taken was linear, so that $f(u)$ did not appear.

Our main tool is based on the technique of projecting the differential
operator onto that segment of the boundary where overposed data lies
and thereby obtain an operator equation suitable for fixed point iteration
and the attendant methods for showing convergence of such schemes.
Other possibilities exist and have been widely used such as
Newton-type schemes.
For the problems at hand these techniques are certainly viable candidates
but require, often considerably, more complex analysis,
in particular, convergence conditions that are likely
not satisfied by highly nonlinear problems.
In a departure from most of the recent work in this area we will
use a Schauder space setting for the parabolic equation.
This offers several advantages including the fact that regularity estimates
for parabolic equations set in Schauder space are independent of spatial
dimension.

In each of the above papers the diffusion process was extended to include
the subdiffusion case where the time-derivative was of fractional type.
The very rough picture that evolves from this modification is no change
to the uniqueness questions of recovery of terms but possibly
substantial changes in the degree of ill-conditioning, especially when
either only very short or when very long time values are present.
Certainly, the analysis is more complex when fractional derivatives are
included.
In the present work we will consider only the classical parabolic situation
as we don't see sufficient novelty of outcome in the central ideas
to merit the analytical complexities that a comprehensive inclusion would
entail although we will sometimes comment on particular instances.

We make a simple, but essential, overview comment on inverse problems
seeking to recover a term that involves the dependent variable $u$.
One can only recover such a term over a range of values that are contained
in the data measurement range and, further, during the evolution process
the range of the solution $u$ cannot have taken values outside of this range.
This makes such problems quite distinct from determining coefficients
where the required range is known from the geometry of the domain $\Omega$.
From a practical perspective this means that the experimental set up
must take this point into consideration and an arbitrary mix of initial
and boundary conditions and values is unlikely to satisfy this range
condition.

There is some way around this issue by representing the unknown function,
here $f(u)$, in a chosen basis set.
If $f$ is known to have a simple form then a few well-chosen elements
may suffice, but otherwise the required extrapolation will lead to
a severely ill-posed problem.
We are not considering this case as we seek a methodology that includes
quite complex situations.
In addition, the iteration schemes we propose are best formulated
in a pointwise form and this would be precluded without a range condition.

In the current work we seek to determine a pair of unknown coefficients.
One of these, $a(x)$ is only spatially dependent but strongly influences the
resulting solution profile.
The reaction term $f(u)$ also might have a strong influence on the
solution profile but
it can also be such that it is small in magnitude and its influence on the
solution is dominated by the diffusion process governed by $a(x)$.
This case is actually more difficult than when the two terms have roughly
equal effect on the solution.
In either situation the range condition must still be in place.
In section \oldref{sect:recons} we will see a situation where violation of
this can occur in a rather subtle way leading to incorrect reconstructions.

The plan of the paper is to introduce the iterative schemes that will
be used to recover both $a$ and $f$ in section~\oldref{sec:conv}
and present conditions under which we are able to show a convergence analysis.
This section is broken down into three cases depending on the data type being
measured: two runs with different initial/boundary data and each
measuring a final time $g(x) = u(x,T)$ for some fixed $t=T$;
two runs each measuring a time trace of the solution $u$ at a fixed
point $\tilde x\in\partial\Omega$; and a single run with measurements of
both a time trace and a final time.
We will also provide a convergence analysis of each of these schemes.
The following section~\oldref{sect:recons}  will show actual reconstructions
using these iterative schemes and demonstrate both their capabilities and their
limitations.

Before proceeding to the above agenda we should make some comments
on physical scaling as this will be very relevant to our reconstructions in
section~\oldref{sect:recons}.
${\mathbb L}$ has been written with a single coefficient $a(x)$ but
more general possibilities would include:
$\;{\mathbb L}u := -\nabla\cdot(a\nabla u) + d(x)\nabla_i u + q(x) u$.
In this setting $d$ would correspond to a drift term and $q$ a potential
or as a linear forcing term with a space-dependent modifier of its
magnitude.
Equation~\eqref{eqn:basic_pde_parabolic} then appears with several
terms all of which have specific physical attributes and in consequence
scale factors depending on the situation.
By this we mean that depending on the context the range of value
taken by the various terms, including our assumed unknowns $a(x)$ and $f(u)$
may have enormous differences.
Examples for the diffusion coefficient range from around
$10^{-6}$ to $10^{-5}$ metres$^2$/sec for molecules in a gas to about
$10^{-10}$ metres$^2$/sec for molecules in a liquid such as water,
\cite{Britanicca:2020},
and in this case the effect of drift represented by $d(x)$
 can overshadow the diffusion effect
(of course the domain size of $\Omega$ also plays a role).
Because of this we must be careful in interpreting the results of
numerical experiments which have taken generic order unity values
as the default.

\def\ul#1{\underline{#1}}
\def\lamtil{\tilde{\lambda}}
\def\phitil{\tilde{\varphi}}
\def\matrix#1#2{\begin{array}{cc}{#1}\\{#2}\end{array}}
\def\fhatpluspp{(\hat{f}^+)''}
\def\fhatplusp{(\hat{f}^+)'}
\section{Convergence analysis}\label{sec:conv}
We consider the inverse problem of recovering the spatially varying diffusion $a(x)$ and the nonlinearity $f(u)$ in the context of three different observation settings:
\begin{itemize}
\item
observation of two states (driven by different initial and boundary conditions as well as forcing terms) at the same final time $T$
\item
observation of one or two states at final time and at a (boundary or interior) point $\tilde{x}$ over time
\item
observation of two different states at $\tilde{x}$ over time
\end{itemize}
We discuss each of these cases in a separate subsection below.

In the first two cases we can follow the approach of projecting the PDE onto the
observation manifold; in the third case this is only partly possible, since
time trace measurements do not map directly into space dependence (of $a$).
While the methods themselves clearly extend to higher space dimensions in the
first two cases (in the third case they cannot, by an elementary dimension count)
we provide an analysis in one space dimension $\Omega=(0,L)$ only.
Possibilities and limitations with respect to to higher space dimensions will
be indicated in some remarks.

Since well-definedness of the methods to be discussed relies on existence, uniqueness and regularity of solutions to the semilinear parabolic initial boundary value problem \eqref{eqn:basic_pde_parabolic}, \eqref{eqn:bdry_cond}, \eqref{eqn:init_cond}, we will provide a statement from the literature \cite{Friedman:1964} on well-posedness of this forward problem.
For our purposes it suffices to consider the case of homogeneous boundary conditions ($b=0$ in \eqref{eqn:bdry_cond}). As a matter of fact, it would be enough to look at the spatially 1-d case for which we do the analysis; however this would not change so much in the assumptions, since the Schauder space results in \cite{Friedman:1964} are independent of the space dimensions.
In case $\gamma=\infty$ (Dirichlet boundary value problem) we can directly make use of \cite[Theorem 9, Chap. 7 Sec. 4, page 205]{Friedman:1964} for existence and of \cite[Theorem 6,  Chap. 7 Sec. 4, page 205]{Friedman:1964} for uniqueness.
In case of Neumann or impedance conditions $\gamma<\infty$, these can be extended by means of the results from \cite[Theorem 9, Chap. 5 Sec. 3, page 144 ff.]{Friedman:1964}, noting that these corresond to the second boundary value problem in \cite{Friedman:1964} with $\beta=\frac{\gamma}{a}$, $g=\frac{b}{a}$.
\begin{theorem}\label{th:wellposedforward}
Let $\Omega$ be a $C^{2+s}$ domain for some $s>0$, let $f$ and $r$ be H\"older continuous on bounded sets and $f$ satisfy the growth conditions
$uf(u)\leq A_1 u^2+A_2$ %(4.10) page 203 in Friedman
$|f(u)|\leq A(u)$ for some positive monitonically increasing function $A$ %(4.17) page 203 in Friedman
and let $a$, $\nabla a$ be H\"older continuous (note that the differential operator in \cite{Friedman:1964} is in nondivergence form) with $a$ being positive and bounded away from zero on $\overline{\Omega}$.
Moreover, let $u_0\in C^{2+s}(\overline{\Omega})$ satisfy the compatibility conditions
$u_0=0$ and $\nabla\cdot(a(x)\nabla u_0(x)) = f(u_0(x)) + r(x,0,u_0)$ for $x\in\partial\Omega$ if $\gamma=\infty$,
or vanish in a neighborhood of $\partial \Omega$ in case $\gamma<\infty$.\\
Then there exists a classical solution $u$ with $u_t,\,u_{x_i},\,u_{x_ix_j}$ H\"older continuous of \eqref{eqn:basic_pde_parabolic}, \eqref{eqn:bdry_cond}, \eqref{eqn:init_cond} with $b=0$.\\
If additionally $f$ is Lipschitz continuous, then this solution is unique.
\end{theorem}
Note that the results taken from \cite{Friedman:1964} deal with solutions on bounded time domains, which allows to impose less restrictive growth conditions than those needed on all of $[0,\infty)$ as a time interval.
Since we will consider functions $f$ that are nonconstant on a bounded domain only, Lipschitz continuity of $f$ is sufficient for the smoothness and growth conditions in Theorem \oldref{th:wellposedforward}.

\subsection{Final time observations of two states}\label{subsec:axfu-twofinal}

Identify $a(x)$, $f(u)$ in
\begin{eqnarray}
&&u_t-(a u_x)_x=f(u)+r^u \quad t\in(0,T)\,, \quad u(0)=u_0\label{eqn:u-f}\\
&&v_t-(a v_x)_x=f(v)+r^v \quad t\in(0,T)\,,  \quad v(0)=v_0\label{eqn:v-f}
\end{eqnarray}
with homogeneous impedance boundary conditions
\begin{equation}\label{eqn:bc}
a\partial_{\vec{n}} u+\gamma^u u =0\,, \quad a\partial_{\vec{n}} v+\gamma^v v =0 \quad\mbox{ on }\partial\Omega
\end{equation}
from final time data
\begin{equation}\label{eqn:finaltimedata}
g^u(x)=u(x,T)\,, \quad g^v=v(x,T)\,, \quad x\in\Omega,
\end{equation}
where $\Omega=(0,L)$.

In order for this data to contain enough information to enable identification
of both $a$ and $f$, it is essential that the two solutions $u$ and $v$ are
sufficiently different, which is mainly achieved by the choice of the two
different driving terms $r^u$ and $r^v$, while the boundary conditions and
later on also the initial data will just be homogeneous.

Projecting the PDEs \eqref{eqn:u-f}, \eqref{eqn:v-f} onto the measurement
manifold $\Omega\times\{T\}$ and denoting by $u(x,t;a,f)$, $v(x,t;a,f)$ their
solutions with coefficients $a$ and $f$, we obtain equivalence of the inverse
problem \eqref{eqn:u-f}--\eqref{eqn:finaltimedata} to a fixed point equation
with the fixed point operator
$\mathbb{T}$ defined by $(a^+,f^+)=\mathbb{T}(a,f)$ with
\begin{equation}\label{eqn:fp_fiti-fiti}
\begin{aligned}
&(a^+ g^u_x)_x(x)+f^+(g^u(x))=D_tu(x,T;a,f)-r^u(x,T) \quad x\in\Omega\\
&(a^+ g^v_x)_x(x)+f^+(g^v(x))=D_tv(x,T;a,f)-r^v(x,T) \quad x\in\Omega\,.
\end{aligned}
\end{equation}
Thus, the iterates are implicitly defined by
\[
\mathbb{M}(a^+,f^+)=\mathbb{F}(a,f) +\mathrm{b}\,,
\]
with the linear operator $\mathbb{M}$, the nonlinear operator $\mathbb{F}$, and the inhomogeneity $\mathrm{b}$ defined by
\begin{equation}\label{eqn:MFb}
\begin{aligned}
\mathbb{M}(a^+\!,f^+)=
\left(\begin{array}{c}\!(a^+ g^u_x)_x+f^+(g^u)\\(a^+ g^v_x)_x+f^+(g^v)\end{array}\right)\,, \quad
\\
\mathbb{F}(a,f)=
\left(\begin{array}{c}u_t(\cdot,T;a,f)\\ v_t(\cdot,T;a,f)\end{array}\right)\,, \quad
\mathrm{b}=
-\left(\begin{array}{c}\!r^u(\cdot,T)\!\\ \!r^v(\cdot,T)\end{array}\right)\,.
\end{aligned}\end{equation}

The question of self mapping and contractivity of $\mathbb{T}$ in a neighbourhood of an exact solution $(a_{ex},f_{ex})$ leads us to consider the differences $\hat{a}^+=a^+-a_{ex}$, $\hat{f}^+=f^+-f_{ex}$, $\hat{a}=a-a_{ex}$, $\hat{u}=u(\cdot,\cdot;a,f)-u_{ex}$, $\hat{v}=v(\cdot,\cdot;a,f)-v_{ex}$,  along with the identity
\[
\mathbb{M}(\hat{a}^+,\hat{f}^+)=\mathbb{F}(a,f)-\mathbb{F}(a_{ex},f_{ex})
\]
(note that the inhomogeneities cancel out)
so that
\begin{eqnarray*}
&&(\hat{a}^+ g^u_x)_x(x)+\hat{f}^+(g^u(x))=D_t\hat{u}(x,T) \quad x\in\Omega \\
&&(\hat{a}^+ g^v_x)_x(x)+\hat{f}^+(g^v(x))=D_t\hat{v}(x,T) \quad x\in\Omega \,.
\end{eqnarray*}
As in \cite{KaltenbacherRundell:2020a}, the convergence estimates consist of two steps:
\begin{enumerate}
\item[(a)] prove bounded invertibility of $\mathbb{M}$, i.e., estimate appropriate norms of $a^+$, $f^+$
\begin{eqnarray}
&&(\hat{a}^+ g^u_x)_x(x)+\hat{f}^+(g^u(x))=\check{u}(x,T) \quad x\in\Omega\label{eqn:diffu}\\
&&(\hat{a}^+ g^v_x)_x(x)+\hat{f}^+(g^v(x))=\check{v}(x,T) \quad x\in\Omega\label{eqn:diffv}
\end{eqnarray}
by appropriate norms of $\check{u}$, $\check{v}$.
\item[(b)] prove smallness of $\mathbb{F}(a,f)-\mathbb{F}(a_{ex},f_{ex})$, i.e., estimate $\check{u}$, $\check{v}$ in terms of the above chosen norms of $\hat{a},\hat{f}$, using the fact that $\check{u}$, $\check{v}$ satisfy certain
{\sc pde}s.
\end{enumerate}

\paragraph{Step (a), bounding $\mathbb{M}^{-1}$}:
Unlike that of \cite{KaltenbacherRundell:2020a},
the elimination strategies are limited.
In this reference  $\hat{q}^+$ was eliminated in order to estimate
$\hat{a}^+$ and then $\hat{a}^+$ was eliminated for estimating $\hat{q}^+$.
In the current situation we cannot eliminate $\hat{f}^+$ any more;
but we can in fact eliminate $\hat{a}^+$ and express $\hat{f}^+$ in terms of $\check{u}$ and $\check{v}$.
This step is considerably more complicated than the $\hat{a}^+$
elimination in \cite{KaltenbacherRundell:2020a}
and is carried out by the following steps.

We eliminate $\hat{a}^+$ by integrating \eqref{eqn:diffu}, \eqref{eqn:diffv} with respect to $x$ and multiplying with $g^v_x(x)$ and $-g^u_x(x)$, respectively,
\[
\begin{aligned}
&g^v_x(x) \int_x^L \hat{f}^+(g^u(\xi))\, d\xi - g^u_x(x) \int_x^L \hat{f}^+(g^v(\xi))\, d\xi \\
&= g^v_x(x) \int_x^L \check{u}(\xi,T)\, d\xi - g^u_x(x) \int_x^L \check{v}(\xi,T)\, d\xi
\end{aligned}
\]
Here we have avoided boundary terms involving $\hat{a}$ by assuming that either $a(L)$ is known (so that $\hat{a}(L)=0$) or homogeneous Neumann boundary conditions are imposed on the right hand boundary (so that $g^u_x(L)=g^v_x(L)=0$).
In order to arrive at a first kind Volterra integral equation for $\hat{f}^+$, we assume that $g^u$ is strictly monotone (w.l.o.g. increasing) with
\begin{equation}\label{eqn:mu}
g^u_x(x)\geq\mu>0  \quad x\in\Omega
\end{equation}
and divide by $g^u_x(x)$ to get
\[
\begin{aligned}
&\tfrac{g^v_x}{g^u_x}(x) \int_x^L \hat{f}^+(g^u(\xi))\, d\xi - \int_x^L \hat{f}^+(g^v(\xi))\, d\xi
= \tfrac{g^v_x}{g^u_x}(x) \int_x^L \check{u}(\xi,T)\, d\xi - \int_x^L \check{v}(\xi,T)\, d\xi
\end{aligned}
\]
and after differentiation
\begin{equation}\label{eqn:fhatplusVolterra}
\begin{aligned}
& \hat{f}^+(g^v(x))
-\tfrac{g^v_x}{g^u_x}(x) \hat{f}^+(g^u(x))
+(\tfrac{g^v_x}{g^u_x})_x(x)\int_x^L \hat{f}^+(g^u(\xi))\, d\xi
\\
&= \check{v}(x,T)
- \tfrac{g^v_x}{g^u_x}(x) \check{u}(x,T)
+(\tfrac{g^v_x}{g^u_x})_x(x)\int_x^L \check{u}(\xi,T)\, d\xi \,.
\end{aligned}
\end{equation}
Further differentiations of \eqref{eqn:fhatplusVolterra} yield the identities
\begin{equation}\label{eqn:fhatplusprime}
\begin{aligned}
& (\fhatplusp (g^v(x)) - \fhatplusp (g^u(x))) g^v_x(x)
-2(\tfrac{g^v_x}{g^u_x})_x(x) \hat{f}^+(g^u(x))
+(\tfrac{g^v_x}{g^u_x})_{xx}(x)\int_x^L \hat{f}^+(g^u(\xi))\, d\xi
\\
&= \check{v}_x(x,T)
- \tfrac{g^v_x}{g^u_x}(x) \check{u}_x(x,T)
- 2(\tfrac{g^v_x}{g^u_x})_x(x) \check{u}(x,T)
+(\tfrac{g^v_x}{g^u_x})_{xx}(x)\int_x^L \check{u}(\xi,T)\, d\xi \,.
\end{aligned}
\end{equation}
and
\begin{equation}\label{eqn:fhatplusprimeprime}
\begin{aligned}
(&\fhatpluspp(g^v(x))g^v_x(x) - \fhatpluspp(g^u(x))g^u_x(x)) g^v_x(x)\\
&\quad=-(\fhatplusp (g^v(x)) - \fhatplusp (g^u(x))) g^v_{xx}(x)
+2(\tfrac{g^v_x}{g^u_x})_x(x)g^u_x(x)\, \fhatplusp (g^u(x))\\
&\qquad+3(\tfrac{g^v_x}{g^u_x})_{xx}(x) \hat{f}^+(g^u(x))
-(\tfrac{g^v_x}{g^u_x})_{xxx}(x)\int_x^L \hat{f}^+(g^u(\xi))\, d\xi
\\
&\qquad+ \check{v}_{xx}(x,T)
- \tfrac{g^v_x}{g^u_x}(x) \check{u}_{xx}(x,T)
- 3(\tfrac{g^v_x}{g^u_x})_x(x) \check{u}_x(x,T)\\
&\qquad- 3(\tfrac{g^v_x}{g^u_x})_{xx}(x) \check{u}(x,T)
+(\tfrac{g^v_x}{g^u_x})_{xxx}(x)\int_x^L \check{u}(\xi,T)\, d\xi \\
&\quad =:\ \Phi(\fhatplusp ,{\hat{f}^+})(x)+ b(x)\,.
\end{aligned}
\end{equation}

To analyse this integral equation we assume that the slopes of $g^u$ and $g^v$ are sufficiently different in the sense that
\begin{equation}\label{eqn:kappa}
\left|\frac{g^v_x(x)}{g^u_x(x)}\right|\leq \kappa<1 \quad x\in\Omega
\end{equation}
and at the same time still ensuring that
\begin{equation}\label{eqn:delta}
g^v_x(x)\geq\delta  \quad x\in\Omega.
\end{equation}
Moreover, since the strategy is to control the full $C^2$ norm by bounding $\fhatpluspp$, we assume $f(g^u(0))$, $f(g^v(0))$, $f'(g^u(0))$, $f'(g^v(0))$ to be known so that
\begin{equation}\label{eqn:fhatplus0}
{\hat{f}^+}(g^u(0))={\hat{f}^+}(g^v(0))=\fhatplusp (g^u(0))=\fhatplusp (g^v(0))=0.
\end{equation}
As was found in
\cite{KaltenbacherRundell:2019c}, \cite{KaltenbacherRundell:2020b},
we have to impose range conditions
\begin{equation}\label{eqn:rangecondition_finaltime}
J=g^u(\Omega)\supseteq g^v(\Omega)\,, \quad
J\supseteq u_{ex}(\Omega\times(0,T))\,, \quad J\supseteq v_{ex}(\Omega\times(0,T))\,.
\end{equation}
Of course the roles of $g^u$ and $g^v$ could be reversed here. Note however, that the assumption of $g^u$ being the data with the larger range conforms with the condition \eqref{eqn:kappa} of $g^u$ being the function with the steeper slope.

\begin{lemma}\label{lem:fhat}
Let $\beta\in[0,1]$, $g^u,g^v\in C^4(\Omega)$, $\check{u}(T),\check{v}(T)\in C^{2,\beta}(\Omega)$, and assume that \eqref{eqn:mu}, \eqref{eqn:kappa}, \eqref{eqn:delta}, \eqref{eqn:fhatplus0} hold.
Then \eqref{eqn:fhatplusprimeprime} is uniquely solvable in
$C^{2,\beta}(J)$ and
\[
\|\fhatpluspp\|_{C^{0,\beta}(J)}
\leq C(g^u,g^v) \Bigl(\|\check{u}(T)\|_{C^{2,\beta}(\Omega)}+\|\check{v}(T)\|_{C^{2,\beta}(\Omega)}\Bigr)
\]
and therefore, by \eqref{eqn:fhatplus0},
\begin{equation}\label{eqn:estf}
\begin{aligned}
\|\hat{f}^+\|_{C^{2,\beta}(J)}\leq \tilde{C} \|\fhatpluspp\|_{C^{0,\beta}(J)}
&\leq \bar{C}(g^u,g^v) \Bigl(\|\check{u}(T)\|_{C^{2,\beta}(\Omega)}+\|\check{v}(T)\|_{C^{2,\beta}(\Omega)}\Bigr)\,,
\end{aligned}\end{equation}
where $\bar{C}(g^u,g^v)$ only depends on smoothness of $g^u,g^v$, as well as on $\kappa, \delta, \mu$.
\end{lemma}
\begin{proof}
%\subsection*{Appendix (proof of Lemma~\oldref{lem:fhat})}
By \eqref{eqn:kappa}, \eqref{eqn:delta} we get, after division by $-g^u_x(x) g^v_x(x)$,
\[
\hat{f}^{+\,\prime\prime}(g^u(x))-\hat{f}^{+\,\prime\prime}(g^v(x))\tfrac{g^v_x}{g^u_x}(x)
= -\frac{1}{g^v_x(x)g^u_x(x)} \Bigl(\Phi(\hat{f}^{+\,\prime},{\hat{f}^+})(x)+b(x)\Bigr)\,.
\]
Hence, taking the supremum over $x\in\Omega$ and using \eqref{eqn:rangecondition_finaltime}
we obtain, using \eqref{eqn:mu}, \eqref{eqn:kappa}, \eqref{eqn:delta},
\[
\|\hat{f}^{+\,\prime\prime}\|_{C(J)}\leq \kappa \|\hat{f}^{+\,\prime\prime}\|_{C(J)} + \frac{1}{\mu\delta}
\|\Phi(\hat{f}^{+\,\prime},{\hat{f}^+})+b\|_{C(\Omega)}
\]
that is,
\[
\|\hat{f}^{+\,\prime\prime}\|_{C(J)}\leq \frac{1}{\mu\delta(1-\kappa)}
\|\Phi(\hat{f}^{+\,\prime},{\hat{f}^+})\|_{C(\Omega)}\,.
\]
By \eqref{eqn:fhatplus0} we can write
\[
\begin{aligned}
\hat{f}^{+\,\prime}(g^u(x))&=\int_0^x \hat{f}^{+\,\prime\prime}(g^u(\xi))g^u_x(\xi)\, d\xi\,, \\
{\hat{f}^+}(g^u(x))
%=\int_0^x \hat{f}^{+\,\prime}(g^u(\xi))g^u_x(\xi)\, d\xi
&=\int_0^x \int_0^\xi \hat{f}^{+\,\prime\prime}(g^u(\sigma))g^u_x(\sigma)\, d\sigma g^u_x(\xi)\, d\xi \,,\\
\int_x^L {\hat{f}^+}(g^u(\xi)) \, d\xi
&= \int_x^L \int_0^\xi \int_0^\sigma \hat{f}^{+\,\prime\prime}(g^u(\tau))g^u_x(\tau)\, d\tau g^u_x(\sigma)\, d\sigma \, d\xi\,,
\end{aligned}
\]
hence
\[
\Phi(\hat{f}^{+\,\prime},{\hat{f}^+}) = K \hat{f}^{+\,\prime\prime}
\]
with a compact operator $K:C(J)\to C(\Omega)$.
On the other hand, with the isomorphism  $B:C(J)\to C(\Omega)$ defined by
\[
(Bj)(x):= (j(g^v(x))g^v_x(x)-j(g^u(x))g^u_x(x))g^v_x(x)
= -\Bigl(j(g^u(x))-j(g^v(x))\tfrac{g^v_x(x)}{g^u_x(x)}\Bigl)g^u_x(x) g^v_x(x)
\]
(and noting that $\,\|B^{-1}\|_{C(\Omega)\to C(J)}\leq \frac{1}{\mu\delta(1-\kappa)}$, see above),
we can write the problem of recovering $\hat{f}^{+\,\prime\prime}$ as a second kind Fredholm equation
\[
j - B^{-1}K j = B^{-1}b
\]
for $j=\hat{f}^{+\,\prime\prime}$ and apply the Fredholm alternative in $C(J)$. To this end, we have to prove that the kernel of $I-B^{-1}K$ is trivial. For this purpose we have to be able to conclude $j\equiv0$ from
$J\equiv0$ where
\[
\begin{aligned}
&J(x)=: j\bigl((g^v(x))g^v_x(x) -(g^u(x))g^u_x(x)\bigr) g^v_x(x)\\
&\quad+\int_0^x\! j(g^v(\xi))g^v_x(\xi)\, d\xi -
g^v_{xx}(x)\int_0^x\! j(g^u(\xi))g^u_x(\xi)\, d\xi
\\
&\quad-2(\tfrac{g^v_x}{g^u_x})_x(x) g^u_x(x) \int_0^x\!j(g^u(\xi))g^u_x(\xi)\, d\xi -3(\tfrac{g^v_x}{g^u_x})_{xx}(x)
\int_0^x \!\int_0^\xi\!j(g^u(\sigma))g^u_x(\sigma)\, d\sigma g^u_x(\xi)\, d\xi
\\
&\quad+(\tfrac{g^v_x}{g^u_x})_{xxx}(x)
\int_x^L\!\int_0^\xi\!\int_0^\sigma\!j(g^u(\tau))g^u_x(\tau)\, d\tau g^u_x(\sigma)\, d\sigma \, d\xi.
\end{aligned}
\]
We then have
\[
\begin{aligned}
J(x) =& j\bigl(g^v(x))g^v_x(x) - (g^u(x))g^u_x(x)\bigr) g^v_x(x)
+\int_{g^v(0)}^{g^v(x)}\!j(z)\, dz\\
&\quad- (g^v_{xx}(x)+2(\tfrac{g^v_x}{g^u_x})_x(x))
\int_{g^u(0)}^{g^u(x)}\!j(z)\, dz
 -3(\tfrac{g^v_x}{g^u_x})_{xx}(x)
\int_{g^u(0)}^{g^u(x)}\!\int_{g^u(0)}^{z} j(y)\, dy \, dz
\\
&\quad+(\tfrac{g^v_x}{g^u_x})_{xxx}(x)
\int_x^L
\int_{g^u(0)}^{g^u(\xi)}\!\int_{g^u(0)}^{z} j(y)\, dy \, dz \, d\xi \,.
\end{aligned}
\]
After division by $-(g^u_x\cdot g^v_x)(x)$ and the change of variables $\zeta=g^u(x)$, this is equivalent to
\[
\begin{aligned}
0=& j(\zeta)-\mathrm{a}(\zeta)\,j(\mathrm{b}(\zeta))
+\int_{\underline{g}}^\zeta \mathrm{k} (\zeta,z) j(z)\, dz\\
&+\mathrm{c}(\zeta)\int_{g^u(0)}^\zeta \int_{g^u(0)}^{z} j(y)\, dy \, dz
+\mathrm{d}(\zeta)\int_\zeta^{g^u(L)}\!\frac{1}{g^u_x((g^u)^{-1}(r)}\int_{g^u(0)}^\zeta \int_{g^u(0)}^{z} j(y)\, dy \, dz\, dr
\end{aligned}
\]
with
\[
\begin{aligned}
&\mathrm{a}(g^u(x))=(\tfrac{g^v_x}{g^u_x})(x), \qquad\mathrm{b}(g^u(x))=g^v(x), \\
&\mathrm{k}(\zeta,z)=
\frac{g^v_{xx}+2(\tfrac{g^v_x}{g^u_x})_x}{(g^u_x\cdot g^v_x)}((g^u)^{-1}(\zeta))
-\frac{1}{(g^u_x\cdot g^v_x)}((g^u)^{-1}(\zeta)) 1\!\mathrm{I}_{[g^v(0),\mathrm{b}(\zeta)]}(z),
\end{aligned}
\]
that is, to an integral equation of the form
\[
0= j(\zeta)-\mathrm{a}(\zeta)\,j(\mathrm{b}(\zeta))
+\int_{\underline{g}}^\zeta \widetilde{\mathrm{k}} (\zeta,z) j(z)\, dz
\]
for all $\zeta\in [g^u(0),g^u(L)]$.
Since this is a homogeneous second kind Volterra integral equation, using Gronwall's inequality and the fact that $0\leq\mathrm{a}\leq\kappa<1$, it follows that $j\equiv0$.

Thus we have
\[
\|\hat{f}^{+\,\prime\prime}\|_{C(J)}\leq C \|b\|_{C(\Omega)}\,.
\]
Moreover we can replace $C(J)$ by $C^\beta(J)$ since the operator $B$ defined above is an isomorphism also between $C^\beta(J)$ and $C^\beta(\Omega)$ as $Bj=b$ implies
\[
\begin{aligned}
&\frac{|j(g^u(x))-j(g^u(y))|}{|g^u(x)-g^u(y)|^\beta} \\
&= g^v_x(x))g^u_x(x)^{-1}\frac{|j(g^v(x))-j(g^v(y))|}{|g^u(x)-g^u(y)|^\beta}
+\frac{g^v_x(x))g^u_x(x)^{-1}-g^v_x(y))g^u_x(y)^{-1}}{|g^u(x)-g^u(y)|^\beta} j(g^v(y))\\
&\quad-\Bigl(\frac{(g^u_x(x))g^v_x(x)^2)^{-1}b(x)-(g^u_x(y))g^v_x(y)^2)^{-1}b(y)}{|g^u(x)-g^u(y)|^\beta}\Bigr)
\end{aligned}
\]
hence, by taking the supremum over $x\in\Omega$ and using the range condition as well as the fact that by \eqref{eqn:kappa}, $|g^u(x)-g^u(y)|^\beta\geq |g^v(x)-g^v(y)|^\beta$,
\[
\begin{aligned}
|j(g^u(x))&-j(g^u(y))|_{C^{0,\beta}(\Omega)}\\
&\leq \frac{1}{1-\kappa} \Bigl( C_1 \|j\|_{C(J)} + \|b\|_{C^{0,\beta}(\Omega)}\Bigr)
\leq \frac{1}{1-\kappa} \Bigl( \frac{C_1}{\mu\delta(1-\kappa)} \|b\|_{C(\Omega)} + \|b\|_{C^{0,\beta}(\Omega)}\Bigr)
\end{aligned}
\]
under the smoothness assumptions made on $g^u$, $g^v$.
\end{proof}
\begin{remark}
Here in case $\beta=0$, $C^{0,0}(\Omega)$ could as well be replaced by $L^\infty(\Omega)$, so we have the choice of spaces
\[
X_f= C^{2,\beta}(J) \mbox{ for some }\beta\in[0,1]\mbox{ or }X_f=W^{2,\infty}(J)
\]
available.
\end{remark}

Now we bound $\hat{a}^+$ in
\begin{equation}\label{eqn:diffu1}
(\hat{a}^+ g^u_x)_x(x)+\hat{f}^+(g^u(x))=\check{u}(x,T) \quad x\in\Omega
\end{equation}
(cf. \eqref{eqn:diffu}), in terms of $\check{u}$, $\check{v}$ and $\hat{f}^+$ where the latter is then bounded by means of \eqref{eqn:estf}.
Here we first of all try to choose the space for $\hat{a}^+$ analogously to \cite{KaltenbacherRundell:2020a} as
\[
X_a = \{d\in H^1(\Omega)\cap L^\infty(\Omega)\,:\, d(L)=0\} \,,
\]
where we have assumed to know $a$ at one of the boundary points (we take the right hand one again) for otherwise \eqref{eqn:diffu1} would not uniquely determine $\hat{a}^+$.
The estimate can simply be carried out as follows.
After integration with respect to space and dividing by $g^u_x$ we get
\begin{equation}\label{eqn:idahatplus}
\begin{aligned}
\hat{a}^+(x)&=\frac{1}{g^u_x(x)}\int_x^L (\check{u}(\xi,T)-\hat{f}^+(g^u(\xi)))\, d\xi\\
\hat{a}^+_x(x)&=
-\frac{g^u_{xx}(x)}{(g^u_x(x))^2}\int_x^L (\check{u}(\xi,T)-\hat{f}^+(g^u(\xi)))\, d\xi%\\
 -\frac{1}{g^u_x(x)}(\check{u}(x,T)-\hat{f}^+(g^u(x)))
\end{aligned}
\end{equation}
and therefore
\[
\begin{aligned}
\|\hat{a}^+\|_{L^\infty(\Omega)}&\leq\frac{1}{\mu}\bigl(\|\check{u}(\cdot,T)\|_{L^1(\Omega)}
+L\|\hat{f}^+\|_{L^\infty(J)}\bigr)\\
\|\hat{a}^+_x\|_{L^2(\Omega)}&\leq
\frac{\sqrt{L}\|g^u_{xx}\|_{L^2(\Omega)}}{\mu^2}\bigl(\|\check{u}(\cdot,T)\|_{L^1(\Omega)}
+L\|\hat{f}^+\|_{L^\infty(J)}\bigr)\\
&\quad +\frac{1}{\mu}\bigl(\|\check{u}(\cdot,T)|_{L^2(\Omega)}+\sqrt{L}\|\hat{f}^+\|_{L^\infty(J)}\bigr)
\end{aligned}
\]
This gives an obvious mismatch with the much higher norm of $\hat{f}$ (and consequently of $\check{u}$, $\check{v}$) in the estimate \eqref{eqn:estf} which from \cite{KaltenbacherRundell:2019c} and \cite{KaltenbacherRundell:2020b} we know to be needed, though.
Also, Lemma \ref{lem:fhat} would not work when applied to the first in place of the second derivative of $\hat{f}^+$, as both $\fhatplusp $ terms in the left hand side of \eqref{eqn:fhatplusprime} have the same factor.

Thus we next try a space for $a^+$ that is more aligned to estimate \eqref{eqn:estf}.
From \eqref{eqn:idahatplus} we get, using the fact that $C^{2,\beta}(\Omega)$ is a Banach algebra and $\|j(g)\|_{C^{2,\beta}(\Omega))}\leq \|j\|_{C^{2,\beta}(J))}\|g\|_{C^{2,1}(\Omega))}$,
\begin{equation}\label{eqn:esta_withf}
\begin{aligned}
&\|\hat{a}^+_x\|_{C^{2,\beta}(\Omega)}\\
&\leq
\Bigl(L\|\bigl(\tfrac{1}{g^u_x}\bigr)_x\|_{C^{2,\beta}(\Omega)}+ \|\tfrac{1}{g^u_x}\|_{C^{2,\beta}(\Omega)}\Bigr) \bigl(\|\check{u}(T)\|_{C^{2,\beta}(\Omega)}+\|g^u\|_{C^{2,1}(\Omega))} \|\hat{f}^+\|_{C^{2,\beta}(J))}\bigr)
\end{aligned}
\end{equation}
and thus, together with \eqref{eqn:estf},
\begin{equation}\label{eqn:esta}
\begin{aligned}
\|\hat{a}^+\|_{C^{3,\beta}(\Omega)}\leq \bar{C}(g^u,g^v) \Bigl(\|\check{u}(T)\|_{C^{2,\beta}(\Omega)}+\|\check{v}(T)\|_{C^{2,\beta}(\Omega)}\Bigr)\,,
\end{aligned}
\end{equation}
where again $\bar{C}(g^u,g^v)$ only depends on the smoothness of $g^u,g^v$, as well as on and $\kappa, \delta, \mu$.
This completes step (a).

\begin{proposition} \label{prop:M}
For $g^u\in C^{4,\beta}(\Omega)$, $g^v\in C^4(\Omega)$ satisfying \eqref{eqn:mu}, \eqref{eqn:kappa}, \eqref{eqn:delta}, \eqref{eqn:rangecondition_finaltime}, there exists a constant $\bar{C}(g^u, g^v)$ depending only on
$\|g^u\|_{C^{4,\beta}(\Omega)}$, $\|g^v\|_{C^4(\Omega)}$, $\mu$, $\kappa$, $\delta$ such that the operator $\mathbb{M}:X_a\times X_f\to C^{2,\beta}(\Omega)$ as defined in \eqref{eqn:MFb} with
\begin{equation}\label{eqn:XaXf}
\begin{aligned}
&X_a=\{d\in C^{3,\beta}(\Omega)\, : \, d(L)=0\}\,,\\
&X_f=\{j\in C^{2,\beta}(J)\, : \, j(g^u(0))=j(g^v(0))=j'(g^u(0))=j'(g^v(0))=0\}
\end{aligned}
\end{equation}
is bounded and invertible with
\[
\|\mathbb{M}^{-1}\|_{C^{2,\beta}(\Omega)\to X_a\times X_f}\leq \bar{C}(g^u, g^v)\,.
\]
\end{proposition}

\begin{remark}
This part of the convergence proof would directly carry over to the case of a fractional
time derivative $D_t^\alpha$ in place of $D_t$ in
\eqref{eqn:u-f}, \eqref{eqn:v-f}.
However, further bounding $\bar{u}=D^\alpha u$ using the fact that it satisfies
the {\sc pde}
$D_t^\alpha \bar{u}-(a \bar{u}_x)_x=D_t^\alpha f(u)+ D_t^\alpha r^u =D_t^\alpha (\frac{f(u)}{u} \cdot u )+ D_t^\alpha r^u$ becomes problematic due to the lack of a chain or product rule for the fractional time derivative.
\end{remark}

\begin{remark}
Extension to higher space dimensions does not appear to be possible, since the strategy of eliminating $\hat{a}^+$ relies on one dimensional integration.
\end{remark}

\begin{remark}
Besides some assumptions directly on the the searched for coefficients $a$ and $f$ (see \eqref{eqn:XaXf}), the proof of Proposition \eqref{prop:M} shows that we also need to impose some conditions on the data $g^u$, $g^v$. The monotonicity assumption \eqref{eqn:mu} allows for inversion of $g^u$ in order to recover values of $f$ from values of $f(g^u)$, avoiding potentially contradictory double or multiple assignments; likewise for \eqref{eqn:delta} and $g^v$.
Also, it is clear that the range of the data must cover all values that will actually appear as arguments in $f$; note that the range condition \eqref{eqn:rangecondition_finaltime} is only imposed on the exact states. The range of the data is therefore the natural maximal domain for any reasonable reconstruction of $f$ and evaluation of $f$ outside this domain must be avoided since it could lead to false results in a forward simulation.
Finally, \eqref{eqn:kappa} is a condition on sufficient deviation of slopes between $g^u$ and $g^v$, in order to allow for unique and stable recovery of $f^+$ according to Lemma \eqref{lem:fhat}. It is probably the most technical and most difficult to realize of these condition. Still note that it is compatible with \eqref{eqn:rangecondition_finaltime} in the sense that the steeper function $g^u$ is the one with larger range. Note that no such condition as \eqref{eqn:kappa} will be needed in the analysis of the final time - time trace observation setting of the next section.
\\
All these conditions can in principle be achieved by appropriate excitations $r^u$, $r^v$ and their validity can be checked during the reconstruction process. Maximum and comparison principles can provide a good intuition for this choice. However, the concrete design of excitations to enable / optimize reconstruction is certainly a topic on its own.
\end{remark}

\bigskip

\paragraph{Step (b), smallness of
$\mathbb{F}(a,f)-\mathbb{F}(a_{ex},f_{ex})$}:\\
The differences $\hat{u}$, $\hat{v}$ along with $\bar{u}=u_t$, $\bar{v}=v_t$ solve
\begin{eqnarray}
&&D_t\hat{u}-(a_{ex} \hat{u}_x)_x+q^u\,\hat{u} = (\hat{a} u_x)_x +\hat{f}(u) \quad t\in(0,T)\,, \quad \hat{u}(0)=0 \label{eqn:uhat}\\
&&D_t\hat{v}-(a_{ex} \hat{v}_x)_x+q^v\,\hat{v} = (\hat{a} v_x)_x +\hat{f}(v) \quad t\in(0,T)\,, \quad \hat{v}(0)=0 \label{eqn:vhat}
\end{eqnarray}
with
\[
\begin{aligned}
q^u&=-\frac{f_{ex}(u)-f_{ex}(u_{ex})}{u-u_{ex}}=-\int_0^1f_{ex}'(u_{ex}+\theta\hat{u})\, d\theta
\,, \quad
q^v&=-\frac{f_{ex}(v)-f_{ex}(v_{ex})}{v-v_{ex}}%=-\int_0^1f_{ex}'(v_{ex}+\theta\hat{v})\, d\theta
\end{aligned}
\]
and
\begin{eqnarray*}
&&\bar{u}_t-(a \bar{u}_x)_x-f'(u)\bar{u}=D_tr^u \quad t\in(0,T)\,, \quad
\bar{u}(0)=(a u_{0x})_x+f(u_0)+r^u(0)\\
&&\bar{v}_t-(a \bar{v}_x)_x-f'(v)\bar{v}=D_tr^v \quad t\in(0,T)\,,  \quad
\bar{v}(0)=(a v_{0x})_x+f(v_0)+r^v(0)\,.
\end{eqnarray*}
Moreover we can write $\check{u}=D_t\hat{u}$, $\check{v}=D_t\hat{v}$ as solutions to the {\sc pde}s
\begin{eqnarray}
&&\check{u}_t-(a_{ex} \check{u}_x)_x-f_{ex}'(u_{ex})\,\check{u} = (\hat{a} \bar{u}_x)_x
+\Bigl(\hat{f}'(u) + \int_0^1f_{ex}''(u_{ex}+\theta\hat{u})\, d\theta\hat{u}\Bigr) \bar{u}
\,, \label{equ:ucheck}\\
&&\check{v}_t-(a_{ex} \check{v}_x)_x-f_{ex}'(v_{ex})\,\check{v} = (\hat{a} \bar{v}_x)_x
+\Bigl(\hat{f}'(v) + \int_0^1f_{ex}''(v_{ex}+\theta\hat{v})\, d\theta\hat{v}\Bigr) \bar{v}
\label{equ:vcheck}
\end{eqnarray}
with initial conditions
\begin{equation}\label{eqn:initucheckvcheck}
\check{u}(0)= (\hat{a} u_{0x})_x +\hat{f}(u_0)\,, \qquad
\check{v}(0)= (\hat{a} v_{0x})_x +\hat{f}(v_0)\,,
\end{equation}
where we have used the identity $f'(u)D_tu-f_{ex}'(u_{ex})D_tu_{ex}$\\ $=f_{ex}'(u_{ex})D_t(u-u_{ex})
+ \Bigl((f-f_{ex})'(u)+(f_{ex}'(u)-f_{ex}'(u_{ex}))\Bigr)D_tu$.

We can estimate $\check{u}$, $\check{v}$ from \eqref{equ:ucheck}, \eqref{equ:vcheck}, \eqref{eqn:initucheckvcheck} in the same way as we estimate $z$ in the estimate preceding Theorem 3.4 of \cite{KaltenbacherRundell:2020b},
taking into account the additional term $(\hat{a} \bar{u}_x)_x$ in the right hand side and $(\hat{a} u_{0x})_x$ in the initial condition, which yields the following.

We split the solution of \eqref{equ:ucheck} $\check{u}=\check{u}^r+\check{u}^0$ into a part $\check{u}^r$ satisfying the inhomogeneous {\sc pde}
\[
\check{u}^r_t-(a_{ex} \check{u}^r_x)_x - {f_{ex}}'(u_{ex})\, \check{u}^r = (\hat{a} \bar{u}_x)_x
+\Bigl(\hat{f}'(u) + \int_0^1f_{ex}''(u_{ex}+\theta\hat{u})\, d\theta\hat{u}\Bigr) \bar{u}
\]
with homogeneous initial conditions $\check{u}^r(x,0)=0$
and a part $\check{u}^0$ satisfying the homogeneous {\sc pde}
$\check{u}^0_t-(a_{ex} \check{u}^0_x)_x - {f_{ex}}'(u_{ex}) \check{u}^0 = 0$ with inhomogeneous initial conditions
\begin{equation}\label{eqn:checku0_init}
\check{u}^0 (x,0)=(\hat{a} u_{0x})_x(x) +\hat{f}(u_0(x)) \quad  x\in\Omega\,.
\end{equation}

Using a classical estimate from the book by Friedman \cite[Theorem 6, page 65]{Friedman:1964}, and abbreviating $\Theta_t=(0,t)\times\Omega$, $\Theta=(0,T)\times\Omega$,
we estimate
\begin{equation}\label{eqn:estzr0}
\begin{aligned}
&\|\check{u}^r\|_{C([0,t];C^{2,\beta}(\Omega))}
\leq
\sum_{|m|\leq2}\|D^m_x \check{u}^r\|_{C^{0,\beta}(\Theta_t)}\\
&\leq K\Bigl(\|(\hat{a} \bar{u}_x)_x\|_{C^{0,\beta}(\Theta_t)}+
\|\int_0^1 {f_{ex}}''(u_{ex}+\theta\hat{u})\, d\theta \, \hat{u} + \hat{f}'(u)\|_{C^{0,\beta}(\Theta)} \|\bar{u}\|_{C^{0,\beta}(\Theta_t)}\Bigr) \\
&\leq K\Bigl(\|\hat{a}\|_{C^{2,\beta}(\Theta_t)}+
(\|{f_{ex}}''\|_{C^{0,\beta}(J)}\|\hat{u}\|_{C^{0,\beta}(\Theta_t)}+\|\hat{f}'\|_{C^{0,\beta}(J)})
(*)
\Bigr)
\|\bar{u}\|_{C^{2,\beta}(\Theta_t)}\,,
\end{aligned}
\end{equation}
where
\[
\begin{aligned}
&(*)=(1+\|u_{ex}\|_{C^{0,1}(\Theta_t)}+\|\hat{u}\|_{C^{0,1}(\Theta_t)})\\
&\|\bar{u}\|_{C^{2,\beta}(\Theta_t)}\leq(\|u_{ex,t}\|_{C^{2,\beta}(\Theta_t)}+\|\check{u}\|_{C^{2,\beta}(\Theta_t)})
\end{aligned}
\]
and
\[
\|\check{u}\|_{C^{2,\beta}(\Theta_t)}
\leq \|\check{u}^r\|_{C^{2,\beta}(\Theta_t)}+\|\check{u}^0\|_{C^{2,\beta}(\Theta_t)}
\leq \sum_{|m|\leq2}\|D^m_x \check{u}^r\|_{C^{0,\beta}(\Theta_t)}
+\|\check{u}^0\|_{C^{2,\beta}(\Theta_t)}\,.\\
\]
Note that we actually estimate $\|\check{u}(T)\|_{C^{2,\beta}(\Omega)}$ by $\|\check{u}\|_{C^{2,\beta}(\Omega\times(0,T))}:=\sum_{m\leq2}
\|D_x^k\check{u}\|_{C^{\beta}(\Omega\times(0,T))}$
with the definition of $\|\cdot\|_{C^{\beta}(\Omega\times(0,T))}$ from \cite{Friedman:1964} (note that $\|\cdot\|_{C^{2,\beta}(\Omega\times(0,T))}$ has a different meaning in \cite{Friedman:1964}).

Here, $\hat{u}$ solves \eqref{eqn:uhat}, so applying again \cite[Theorem 6, page 65]{Friedman:1964} we obtain
\[
\begin{aligned}
\|\hat{u}\|_{C^{2,\beta}(\Theta_t)}&\leq
\sum_{|m|\leq2}\|D^m_x \hat{u}\|_{C^{0,\beta}(\Theta_t)}
\leq K \|(\hat{a}u_x)_x+\hat{f}(u)\|_{C^{0,\beta}(\Theta_t)} \\
&\leq
K \Bigl(\|\hat{a}\|_{C^{2,\beta}(\Theta_t)}\|\hat{u}\|_{C^{2,\beta}(\Theta_t)}+\|\hat{f}\|_{C^{0,\beta}(J)} (1+\|u_{ex}\|_{C^{0,1}(\Theta_t)}+\|\hat{u}\|_{C^{0,1}(\Theta_t)})\Bigr)\\
&\leq
K \Bigl(\|\hat{a}\|_{C^{2,\beta}(\Theta_t)}+\|\hat{f}\|_{C^{0,\beta}(J)}\Bigr) (1+\|u_{ex}\|_{C^{0,1}(\Theta_t)}+\|\hat{u}\|_{C^{2,\beta}(\Theta_t)})\Bigr),
\end{aligned}
\]
thus, for $\|\hat{a}\|_{C^{2,\beta}(\Theta_t)}+\|\hat{f}\|_{C^{0,\beta_0}(J)}\leq\rho_0<\frac{1}{K}$,
\[
\|\hat{u}\|_{C^{0,1}(\Theta_t)}
\leq \|\hat{u}\|_{C^{2,\beta}(\Theta_t)}
\leq \frac{1}{1-\rho_0 K}
\Bigl(\|\hat{a}\|_{C^{2,\beta}(\Theta_t)}+\|\hat{f}\|_{C^{0,\beta}(J)}\Bigr)
(1+\|u_{ex}\|_{C^{0,1}(\Theta_t)})\,.
\]

Altogether, for $\|\hat{a}\|_{C^{2,\beta}(\Theta_t)}+\|\hat{f}'\|_{C^{0,\beta}(J)}\leq\rho_1$ with $\rho_1<\frac{1}{K}$, we end up with an estimate for $\check{u}^r$ of the form
\begin{equation}\label{eqn:estucheckr}
\begin{aligned}
\|\check{u}^r\|_{C([0,t];C^{2,\beta}(\Omega))}
\leq &C(K,\rho_0,\rho_1,\|{f_{ex}}''\|_{C^{0,\beta}(J)}) \Bigl(
\|\hat{a}\|_{C^{2,\beta}(\Omega)}
+\|\hat{f}\|_{C^{1,\beta}(J)}(1+\|u_{ex}\|_{C^{0,1}(\Theta_t)}\Bigr)\\
&\quad\times(\|u_{ex,t}\|_{C^{2,\beta}(\Theta_t)}+\|\check{u}^0\|_{C^{2,\beta}(\Theta_t)})\,.
\end{aligned}
\end{equation}

To estimate $\check{u}^0$, we do not use the same result from \cite[Theorem 6, page 65]{Friedman:1964}, since this would not give a contraction estimate with respect to $\hat{a}$, $\hat{f}$. Rather we attempt to employ dissipativity of the equation, however, this fails like in \cite{KaltenbacherRundell:2020b}, as we will now illustrate.
Let us first of all point out that the difficulty here lies in the fact that the coefficient $f_{ex}'(u_{ex})$ is time-dependent and thus the abstract ODE corresponding to \eqref{equ:ucheck} is non-autonomous.
Thus semigroup decay estimates would require the zero order coefficient to be
constant in time or at least time periodic,
see, e.g., \cite[Chapter 6]{Lunardi:1995}. To illustrate why indeed the general
form of $f_{ex}'(u_{ex})$ most probably prohibits decay of solutions, we follow
a perturbation approach, with a time constant potential
$q\approx -f_{ex}'(u_{ex})$, which we assume to be positive.
Using the series expansion of $\check{u}^0$ in terms of the eigenvalues
and -functions $(\lambda_n,\phi_n)_{n\in\mathbb{N}}$ of the elliptic operator
$\mathbb{L}$ defined by
$\mathbb{L}v=-(a_{ex}v_x)_x+q\,v$ as well as the induced Hilbert spaces
\[
\dot{H}^\sigma(\Omega)=\{v\in L^2(\Omega)\ : \ \sum_{n=1}^\infty \lambda_n^{\sigma/2} \langle v,\phi_n\rangle \phi_n \, \in L^2(\Omega)\}
\]
where $\langle\cdot,\cdot\rangle$ denotes the $L^2$ inner product on $\Omega$, with the norm
%\[
$\displaystyle{
\|v\|_{\dot{H}^\sigma(\Omega)}
= \Bigl(\sum_{n=1}^\infty \lambda_n^\sigma \langle v,\phi_n\rangle^2 \Bigr)^{\!\frac{1}{2}}
}$
%\]
that is equivalent to the $H^\sigma(\Omega)$ Sobolev norm, we can write
\begin{equation}\label{eqn:checku0}
\begin{aligned}
\check{u}^0(x,t)=&\sum_{n=1}^\infty \bigl(e^{-\lambda_n t} \langle (\hat{a} u_{0x})_x \!+\!\hat{f}(u_0),\phi_n\rangle
\!+\! \int_0^t e^{-\lambda_n (t-s)} \langle (f_{ex}'(u_{ex}(s))+q)\check{u}^0(s),\phi_n\rangle\, ds\bigr) \phi_n(x)\\
=:&\,\check{u}^{0,1}(x,t)+\check{u}^{0,2}(x,t)
\end{aligned}
\end{equation}
From Sobolev's embedding theorem with $\sigma > d/2+2+\beta = 5/2 + \beta$,
\begin{equation}\label{eqn:estucheck01C2beta}
\begin{aligned}
\|\check{u}^{0,1}(t)\|_{C^{2,\beta}(\Omega)}&\leq C_{\dot{H}^\sigma,C^{2,\beta}}^\Omega
\Bigl(\sum_{n=1}^\infty \lambda_n^\sigma e^{-2\lambda_n t} \langle (\hat{a} u_{0x})_x+\hat{f}(u_0),\phi_n\rangle^2\Bigr)^{1/2}\\
&\leq  C_{\dot{H}^\sigma,C^{2,\beta}}^\Omega
\sup_{\lambda\geq\lambda_1}\lambda^{\sigma/2-1} e^{-\lambda t}
\|(\hat{a} u_{0x})_x+\hat{f}(u_0)\|_{\dot{H}^2(\Omega)}\\
&\leq  C_{\dot{H}^\sigma,C^{2,\beta}}^\Omega
\Psi(t;\sigma,\lambda_1)
\Bigl(\|\hat{a}\|_{C^2(\Omega)} \|u_{0\,xx}\|_{\dot{H}^2(\Omega)} \\
&\qquad +\|\hat{a}\|_{C^3(\Omega)} \|u_{0\,x}\|_{\dot{H}^2(\Omega)}
+ C_\mathbb{L}\|\hat{f}\|_{C^2(J)} \|u_0\|_{\dot{H}^2(\Omega)}\Bigr)
\end{aligned}
\end{equation}
and
\begin{equation}\label{eqn:estucheck02C2beta}
\begin{aligned}
\|&\check{u}^{0,2}(t)\|_{C^{2,\beta}(\Omega)}\leq C_{\dot{H}^\sigma,C^{2,\beta}}^\Omega
\Bigl(\int_0^\tau\sum_{n=1}^\infty \lambda_n^\sigma e^{-2\lambda_n (t-s)}
\langle ({f_{ex}}'(u_{ex}(s))+q)\check{u}^0(s),\phi_n\rangle^2\, ds\Bigr)^{1/2}\\
&\leq  C_{\dot{H}^\sigma,C^{2,\beta}}^\Omega C(\Omega)
\left(\int_0^t\Psi(\tau-s;\sigma,\lambda_1)\|{f_{ex}}'(u_{ex}(s))+q\|_{H^2(\Omega)}^2\, ds \right)^{1/2}
\|\check{u}^0\|_{C([0,t];C^2(\Omega))}\,,
\end{aligned}
\end{equation}
with $C_{\mathbb{L}}$ such that
$\|\mathbb{L} j(v)\|_{L^2(\Omega)}\leq C_{\mathbb{L}} \|j\|_{C^2(\mathbb{R})} \|v\|_{\dot{H}^2(\Omega)}$ for all $j\in C^2(\mathbb{R})$, $v\in \dot{H}^2(\Omega)$ and
\begin{equation}\label{eqn:Psi}
\Psi(t;\sigma,\lambda_1) = \begin{cases}
(\sigma/2-1)^{\sigma/2-1} e^{1-\sigma/2}\, t^{1-\sigma/2}&\mbox{ for }t\leq \frac{\sigma-2}{2\lambda_1}\\
\lambda_1^{\sigma/2-1} e^{-\lambda_1 t}&\mbox{ for }t\geq \frac{\sigma-2}{2\lambda_1}
\,.\end{cases}
\end{equation}
The appearance of $\|\check{u}^0\|_{C([0,t];C^2(\Omega))}$ on the right hand side of \eqref{eqn:estucheck02C2beta} forces us to look at the supremum of $\|\check{u}^0(t)\|_{C^2(\Omega)}$ over $t\in[0,T]$, which, however, due to the singularity of $\Psi$ at $t=0$ cannot be estimated appropriately by \eqref{eqn:estucheck01C2beta}.
Thus the convolution term in \eqref{eqn:checku0} inhibits exponential decay as would be expected from the estimate of the first term in \eqref{eqn:checku0}.

Thus we have no means of establishing contractivity in the presence of the
initial term \eqref{eqn:checku0_init} and therefore will assume that
$u_0$ and $f(0)$ vanish.
Thus, putting this altogether
\begin{equation}\label{eqn:estucheck}
\begin{aligned}
&\|\check{u}(t)\|_{C^{2,\beta}(\Omega)}\leq
\|\check{u}^r\|_{C([0,t];C^{2,\beta}(\Omega))}\\
&\leq C
\|u_{ex,t}\|_{C^{2,\beta}(\Theta)}
(\|\hat{a}\|_{C^3(\Omega)}+(1+\|u_{ex}\|_{C^{0,1}(\Theta)})\|\hat{f}\|_{C^{1,\beta}(J)})
\,.
\end{aligned}
\end{equation}
provided $\|\hat{a}\|_{C^{2,\beta}(\Theta_t)}+\|\hat{f}\|_{C^{1,\beta}(J)}\leq\rho$ small enough.\\
The same estimate can be used for bounding $\|\check{v}\|_{C^{2,\beta}(\Omega\times(0,T))}$

This together with \eqref{eqn:esta}, \eqref{eqn:estf} yields contractivity of $\mathbb{T}$ on a ball of sufficiently small radius $\rho$.

\begin{theorem}
Under the assumptions of Proposition \ref{prop:M} and if additionally
$\|u_{ex,t}\|_{C^{0,\beta}(\Theta)}$ is sufficiently small and $u_0=0$, $f_{ex}(0)=0$, $g^u(0)=0$, there exists $\rho>0$ such that $\mathbb{T}$ is a self-mapping on $B_\rho^{X_a\times X_f}(a_{ex},f_{ex})$ and the convergence estimate
\[
\|\mathbb{T}(a,f)-\mathbb{T}(a_{ex},f_{ex})\|_{X_a\times X_f}\leq
q \|(a,f)-(a_{ex},f_{ex})\|_{X_a\times X_f}
\]
holds for some $q\in(0,1)$ and $X_a$, $X_f$ as in \eqref{eqn:XaXf}.
\end{theorem}

\begin{remark} In \cite[Theorem 3.5]{KaltenbacherRundell:2020b} we have
alternatively proven contractivity for monotone $f$.
However, the approach taken there does not seem to go through here
for the following reasons.
Additionally to the exponential decay of $\bar{u}$,
we would need to show exponential decay of $(\hat{a}\, \bar{u}_x)_x$.
Another obstacle is that the strategy of showing contractivity of $\mathbb{T}$ by exploiting the regularity gain
\[
\begin{aligned}
\|\mathbb{T}^2(a,f)-\mathbb{T}^2(a_{ex},f_{ex})\|_{C^{3,\beta}(\Omega)\times C^{2,\beta}(J)}
&\leq C(g^u,g^v)\|(\check{u},\check{v})\|_{C^{2,\beta}(\Omega)^2}\\
&\leq \tilde{C}\,C(g^u,g^v)\|\mathbb{T}(a,f)-\mathbb{T}(a_{ex},f_{ex})\|_{C^{1,\beta}(\Omega)\times C^{1,\beta}(J)}
\end{aligned}
\]
does not work here.
This is because we have no possibility to estimate the $C^{1,\beta}(J)$
norm of the $f$ part of
$\|\mathbb{T}(a,f)-\mathbb{T}(a_{ex},f_{ex})\|_{C^{1,\beta}(\Omega)\times C^{1,\beta}(J)}$ in terms of the $C^{1,\beta}(\Omega)$ norm of $\check{u},\check{v}$. (So far the Volterra integral equation approach from step (a) only works for differentiability order $\geq2$.)
However, the maximal parabolic regularity approach in \cite{KaltenbacherRundell:2020b} relies on the embedding $W^{2,p}(\Omega)\to C^{1,\beta}(\Omega)$ applied to $\check{u},\check{v}$ and therefore only works with the $C^{1,\beta}$ norm of $\hat{f}$.
\end{remark}

\subsection{Final time / time trace observations of one or two states}\label{subsec:axfu-twomixed}

Identify $a(x)$, $f(u)$ in \eqref{eqn:u-f}, \eqref{eqn:v-f}
with homogeneous impedance boundary conditions \eqref{eqn:bc},
from observations
\begin{equation}\label{eqn:mixeddata}
g(x)=u(x,T), \quad x\in\Omega, \qquad h(t)=v(\tilde{x},t), \quad t\in\Theta\subseteq[0,T]\,,
\end{equation}
for some $\tilde{x}\in\overline{\Omega}$, typically on the boundary. We here assume that $\tilde{x}$ is the right hand boundary point $L$; other cases can be treated analogously.
However, if $\tilde{x}$ is an interior point, then some term that we can eliminate by using the boundary condition would have to be additionally be taken into account in the analysis.
More precisely, we make use of the right hand boundary condition that implies
$a^+_x(L) v_x(L,t;a,f) =- \frac{a^+_x(L)}{a^+(L)} \gamma^v v(L,t;a,f)$ and replace $v(L,t;a,f)$ by the data $h(t)$ here.
Thus, again projecting onto the measurement manifolds $\Omega\times\{T\}$ and $\{L\}\times(0,T)$,
we define the fixed point operator $\mathbb{T}$ by $(a^+,f^+)=\mathbb{T}(a,f)$ with
\begin{equation}\label{eqn:fp_fiti-titr}
\begin{aligned}
&(a^+ g_x)_x(x)+f^+(g(x))=D_tu(x,T;a,f)-r^u(x,T) \quad x\in\Omega\\
&a^+(L) v_{xx}(L,t;a,f) - \frac{a^+_x(L)}{a^+(L)} \gamma^v h(t)+f^+(h(t))=h_t(t)-r^v(L,t) \quad t\in\Theta\,,
\end{aligned}
\end{equation}
which, when assuming that $a^L=a_{ex}(L)$ is known and $\gamma^v(L)=0$, simplifies to
\begin{eqnarray}
(a^+ g_x)_x(x)+f^+(g(x))&=&D_tu(x,T;a,f)-r^u(x,T) \quad x\in\Omega \label{eqn:faplus_mixed_u}\\
f^+(h(t))&=&h_t(t)-r^v(L,t)-a^L v_{xx}(L,t;a,f) \quad t\in\Theta\,. \label{eqn:faplus_mixed_v}
\end{eqnarray}
The recursion for the the differences $\hat{a}^+=a^+-a_{ex}$, $\hat{f}^+=f^+-f_{ex}$, $\hat{a}=a-a_{ex}$,
reads as
\begin{eqnarray}
(\hat{a}^+ g_x)_x(x)+\hat{f}^+(g(x))&=&D_t\hat{u}(x,T) \quad x\in\Omega\label{eqn:ahat_mixed}\\
\hat{f}^+(h(t))&=&-a^L \hat{v}_{xx}(L,t;a,f) \quad t\in\Theta\,.\label{eqn:fhat_mixed}
\end{eqnarray}
where $\hat{u}$, $\hat{v}$ solve \eqref{eqn:uhat}, \eqref{eqn:vhat}, that is
\[
\mathbb{M}(\hat{a}^+,\hat{f}^+)=\mathbb{F}(a,f)-\mathbb{F}(a_{ex},f_{ex})\,,
\]
with the operators $\mathbb{M}$ and $\mathbb{F}$ defined by
\[
\mathbb{M}(a^+,f^+)=
\left(\begin{array}{c}(a^+ g_x)_x+f^+(g^u)\\f^+(h)\end{array}\right)\,, \quad
\mathbb{F}(a,f)=
\left(\begin{array}{c}u_t(\cdot,T;a,f)\\ -a^Lv_{xx}(L,\cdot;a,f)\end{array}\right)\,.
\]

Inverting $\mathbb{M}$ is therefore much easier this time. We can first invert \eqref{eqn:faplus_mixed_v} for $f^+$ relying on the range condition
\begin{equation}\label{eqn:rangecondition_mixed}
J=h(\Theta)\,, \quad
J\supseteq u_{ex}(\Omega\times(0,T))\,, \quad J\supseteq v_{ex}(\Omega\times(0,T))
\end{equation}
(in place of \eqref{eqn:rangecondition_finaltime}) and then compute $a^+$ from \eqref{eqn:faplus_mixed_u} in exactly the same way as we have done before, cf. \eqref{eqn:esta_withf}. It will turn out that a lower regularity estimate of $a^+$ and $f^+$ (or more precisely, of $\hat{a}^+$ and $\hat{f}^+$, in order to prove contractivity) suffices, that is, we will use the function spaces
$X_a=C^{2,\beta}(\Omega)$, $X_f=C^{1,\beta}(J)$ (with further boundary conditions).

\paragraph{Step (a), bounding $\mathbb{M}^{-1}$}:
From \eqref{eqn:fhat_mixed}, we get, after differentiation,
\[
\|\fhatplusp (h)\, D_th\|_{C^{0,\beta}(\Theta)}\leq |a^L| \|D_t\hat{v}_{xx}\|_{C^{0,\beta}(\Omega\times(0,T))}
\]
thus, assuming strict monotonicity of $h$, that is,
\begin{equation}\label{eqn:mu_h}
|D_th(t)|\geq\delta
\end{equation}
we get, using $\hat{f}^+(p)=0$ for some $p\in J$,
\begin{equation}\label{eqn:estf_mixed}
\begin{aligned}
\|\hat{f}^+\|_{C^{1,\beta}(J)}\leq C(h) |a^L| \|D_t\hat{v}\|_{C^{2,\beta}(\Omega\times(0,T))}
\end{aligned}
\end{equation}
where $C(h)$ only depends in $\mu$ and $\|h\|_{C^{1,1}(\Theta)}$.
Likewise, we could attempt to estimate a higher order norm of $f^+$ by differentiating a second time
\[
\|\fhatpluspp(h)\, (D_th)^2\|_{C^{0,\beta}(\Theta)}\leq
\|\fhatplusp (h)\, D_t^2h\|_{C^{0,\beta}(\Theta)}
+|a^L| \|D_t^2\hat{v}_{xx}\|_{C^{0,\beta}(\Omega\times(0,T))}\,.
\]
However, this would involve too high derivatives of $\hat{v}$.

From \eqref{eqn:ahat_mixed}, that is \eqref{eqn:idahatplus}, we get
when taking only the $C^{1,\beta}$ norm,
(i.e., the $C^{2,\beta}$ norm of $\hat{a}$)
\begin{equation}\label{eqn:esta_withf_mixed}
\begin{aligned}
&\|\hat{a}^+_x\|_{C^{1,\beta}(\Omega)}\\
&\leq
\bigl(L\|\bigl(\tfrac{1}{g_x}\bigr)_x\|_{C^{1,\beta}(\Omega)}+ \|\tfrac{1}{g_x}\|_{C^{1,\beta}(\Omega)}\bigr) \bigl(\|D_t\hat{u}(T)\|_{C^{1,\beta}(\Omega)}+\|g\|_{C^{1,1}(\Omega))} \|\hat{f}^+\|_{C^{1,\beta}(J)}\bigr)
\end{aligned}
\end{equation}
in place of \eqref{eqn:esta_withf}, thus, together with \eqref{eqn:estf_mixed}, using again the abbreviations $\check{u}=D_t\hat{u}$, $\check{v}=D_t\hat{v}$,
\begin{equation}\label{eqn:esta_mixed}
\begin{aligned}
\|\hat{a}^+\|_{C^{2,\beta}(\Omega)}\leq
\bar{C}(g,h)\Bigl(\|\check{u}(T)\|_{C^{1,\beta}(\Omega)}+
\|\check{v}\|_{C^{2,\beta}(\Omega\times(0,T))}\Bigr)
\end{aligned}
\end{equation}
that is an estimate in norms that are by one order lower than those in \eqref{eqn:esta}.

\begin{proposition} \label{prop:M_mixed}
For $g\in C^{3,\beta}(\Omega)$, $h\in C^{1,1}(\Theta)$ satisfying \eqref{eqn:mu}, \eqref{eqn:rangecondition_mixed}, \eqref{eqn:mu_h}, $p\in J$,
there exists a constant $\bar{C}(g^u, g^v)$ depending only on
$\|g\|_{C^{3,\beta}(\Omega)}$, $\|h\|_{C^{1,1}(\Theta)}$, $\mu$, $\delta$, such that the operator $\mathbb{M}:X_a\times X_f\to C^{1,\beta}(\Omega)$ as defined in \eqref{eqn:MFb} with
\begin{equation*}
\begin{aligned}
&X_a=\{d\in C^{2,\beta}(\Omega)\, : \, d(L)=0\}\,,\\
&X_f=\{j\in C^{1,\beta}(J)\, : \, j(p)=0\}
\end{aligned}
\end{equation*}
is boundedly invertible with
\[
\|\mathbb{M}^{-1}\|_{C^{1,\beta}(\Omega)\to X_a\times X_f}\leq \bar{C}(g^u, g^v)\,.
\]
\end{proposition}

\begin{remark}
The obvious spatially higher dimensional version of \eqref{eqn:ahat_mixed}, \eqref{eqn:fhat_mixed} is
\begin{eqnarray}
\nabla\cdot(\hat{a}^+ \nabla g)(x)+\hat{f}^+(g(x))&=&D_t\hat{u}(x,T;a,f) \quad x\in\Omega \label{eqn:ahat_mixed1}\\
\hat{f}^+(h(t))&=&-a^L \triangle\hat{v}(\tilde{x},t;a,f) \quad t\in\Theta\,.\label{eqn:fhat_mixed1}
\end{eqnarray}
However, \eqref{eqn:ahat_mixed1}, which is a transport equation for $\hat{a}^+$, in general only yields an estimate of $\hat{a}^+$ in a weaker norm than the one needed here.
\end{remark}

\paragraph{step (b) smallness of $\mathbb{F}(a,f)-\mathbb{F}(a_{ex},f_{ex})$}:\\
For $D_t\hat{u}=\check{u}$ we can just use the estimate \eqref{eqn:estucheckr},
%\eqref{eqn:estucheck0C2beta}. However, the latter contains the $C^3$ norm of $\hat{a}$ and the $C^2$ norm of $\hat{f}$, so in order to avoid this, we assume that $u_0$ vanishes and $f(0)=0$. Thus, in place of \eqref{eqn:estucheck}, we have $\check{u}=\check{u}^r$ with $\check{u}^r$ satisfying \eqref{eqn:estucheckr}.
and again assume that $u_0$ and $f(0)$ vanish.

\begin{theorem} \label{th:contr_fiti_titr}
For $g\in C^{3,\beta}(\Omega)$, $h\in C^{2,\beta}(\theta)$ satisfying \eqref{eqn:mu}, \eqref{eqn:rangecondition_mixed}, \eqref{eqn:mu_h},  and if additionally $f_{ex}(0)=0$, $u_0=0$,
$\|u_{ex,t}\|_{C^{2,\beta}(\Theta)}$ is sufficiently small, there exists $\rho>0$ such that $\mathbb{T}$ is a self-mapping on $B_\rho^{X_a\times X_f}(a_{ex},f_{ex})$ and the convergence estimate
\[
\|\mathbb{T}(a,f)-\mathbb{T}(a_{ex},f_{ex})\|_{X_a\times X_f}\leq
q \|(a,f)-(a_{ex},f_{ex})\|_{X_a\times X_f}
\]
holds for some $q\in(0,1)$ and
\begin{equation}\label{eqn:XaXf_mixed}
\begin{aligned}
&X_a=\{d\in C^{2,\beta}(\Omega)\, : \, d(L)=0\}\,,\\
&X_f=\{j\in C^{1,\beta}(J)\, : \, j(0)=0\}
\end{aligned}
\end{equation}
\end{theorem}

\begin{remark}
As opposed to the previous section, here the strategy of proving contractivity for monotone $f$ from \cite{KaltenbacherRundell:2020b} would also work here since we can estimate the $C^{1,\beta}$ norm of $\hat{f}^+$. This would also allow to tackle nonhomogeneous initial data by using a maximal parabolic regularity estimate, see \cite[Section 3.3.2]{KaltenbacherRundell:2020b}.
\end{remark}

\begin{remark}
Since no such condition as \eqref{eqn:kappa}
(that is sufficiently different slopes of $g_u$, $g_v$) is needed here, the result from Theorem \oldref{th:contr_fiti_titr} remains valid with a single run, i.e., $v=u$.
\end{remark}

\subsection{Time trace observations of two states}\label{subsec:axfu-titr}

A well-known ``{\it folk theorem\/}'' suggests that trying to
reconstruct a function depending on a variable $x$
from data measured in an orthogonal direction to $x$ is
inevitably severely ill-conditioned.
This has been borne out in innumerable cases.
The recovery of $a(x)$ in the parabolic equation
$u_t - \nabla.(a\nabla u) = f(u)$ from time-valued data
is an example of the above.
It is known that recovery of spatially-dependent coefficients in an
elliptic operator from time data-valued measurements made on the boundary
$\partial\Omega$ is extremely ill-conditioned with an exponential dependence
of the unknown in terms of the data.
We give a synopsis of the reasons for this that have particular relevance
to our current theme; while recovery of the diffusion coefficient $a(x)$ from
final time data is only a mildly ill-conditioned problem, that from
time-trace data is extremely ill-conditioned.

We will set the problem in one space dimension taking
$\Omega=[0,1]$ and remove all non-essential terms.
Thus we consider the recovery of the uniformly positive  $a(x)$ from
$u_t - (au_x)_x = 0$ with (say) boundary conditions
$u_x(0,t) = 0$, $u(1,t) = 0$ with initial value $u_0(x)$
where our overposed data is the value of $h(t) = u(0,t)$.
Note that this is a single function recovery from a linear problem.
The standard approach here is due to Pierce \cite{Pierce:1979}.
The solution to the direct problem has the representation
$u(x,t) = \sum_{n=1}^\infty \langle u_0,\phi_n\rangle\,e^{-\lambda_nt}\phi_n(x)$
where $\{\lambda_n,\phi_n\}$ are the eigenvalues / eigenfunctions of
$\mathbb{L} = -(a u_x)_x$ under the conditions $u_x(0)=u(1)=0$
and we normalise these by
the standard Sturm-Liouville condition $\phi_n'(0)=1$.
For a given $a(x)$ this is transmutable into the {\it norming constants\/}
and more usual Hilbert space norm
$\|\phi\|_2$,
 \cite{RundellSacks:1992b,RundellSacks:1992a}.
We have to recover $a(x)$ from $h(t) = u(0,t)$, that is from
\begin{equation}\label{eqn:Dirichlet_series}
h(t) = \sum_{n=1}^\infty A_n e^{-\lambda_n t}
\qquad
A_n = \langle u_0,\phi_n\rangle.
\end{equation}
Note that $\lambda_n>0$ and that $h(t)$ is in fact analytic.
It is well-known that the pair $\{\lambda_n,A_n\}$ is uniquely recoverable
from $h(t)$ and this is easily seen by taking Laplace transform
(analytically extending $h(t)$ to all $t>0$ if need be).
Then \eqref{eqn:Dirichlet_series} becomes
\begin{equation}\label{eqn:rat_func_series}
\mathcal{L}(h(t))\to s := H(s) = \sum_{n=1}^\infty \frac{A_n}{s+\lambda_n}
\end{equation}
and shows that knowing the meromorphic function $H(s)$ provides the location
$\lambda_n$ and residues $A_n$ of its poles.
Of course, the extreme ill-conditioning of this step is now apparent;
we know the values of the analytic function $H(s)$ on the positive real line
and have to recover information on the negative real axis -- in fact
all the way to $-\infty$ as the eigenvalues have the asymptotic behaviour
$-\lambda_n \sim -c n^2$.
Recovery of just the eigenvalues isn't sufficient to recover $a(x)$
but we have further information, namely the sequence $\{A_n\}$.
The simplest way to use this is to select $u_0$ to be impulsive -- say
$u_0(x) = \delta(x)$ for then $A_n = \phi_n(0)$.
Thus with this we have both the spectrum and an endpoint condition
on the eigenfunctions and standard results from inverse Sturm-Liouville theory
show this is sufficient for the recovery of $a(x)$,
\cite{Borg:1946,Gelfand-Levitan:1951,RundellSacks:1992b}.

Actually, with this approach the computational cost of an effective
reconstruction of $a(x)$ is relatively low.
Determining the coefficients from a finite sum in \eqref{eqn:rat_func_series}
is just Pad\'e approximation of $H(s)$ and there are efficient methods
such as in \cite{BakerGrave-Morris:1996} to effect this recovery.
The reconstruction of $a(x)$ from this finite spectral can be accomplished
extremely quickly, \cite{RundellSacks:1992b}.

An entirely similar argument holds if instead our purpose was  to recover
the potential $q(x)$ in $\,\mathbb{L} = -u_{xx} + q(x)u = 0$.
The recovery of two spectral sequences proceeds identically and
this potential form for $\mathbb{L} $  is the canonical version
\cite{RundellSacks:1992a,CCPR:1997} of the inverse Sturm-Liouville problem
to which all others can be transformed by means of the Liouville transform.
In fact, this part of the process is only mildly ill-conditioned in terms
of mappings between spaces for we have
$\|q_1 - q_2\|_{L_2} \leq C\|\lambda_{n,1}-\lambda_{n,2}\|_{\ell^2}$.

However, the above paints a rather too rosy picture of the difficulties
involved.
The basic asymptotic form for the eigenvalues for a $q\in L^2$ is
\begin{equation}\label{eqn:sturm_liouville_eigen}
\lambda_n(q) = \bigl(n-{\textstyle{\frac{1}{2}}}\bigr)^2\pi^2 +
\int_0^1 q(s)\,ds + \epsilon_n, \qquad \epsilon_n \in \ell^2
\end{equation}
The dominant term in this expansion is identical for all $q\in L^2$.
The more regularity assumed of $q(x)$ the more rapid the decay of
the ``information sequence'' $\{\epsilon_n\}$,
\cite{PoschelTrubowitz:1987,CCPR:1997,Savchuk:Shalikov:2006}.
For example, if $q\in H^{m,2}$ then
$\epsilon_n \sim \sigma_n/n^{2[m/2]+1}$ where again $\sigma_n\ell^2$.
Thus the approximation to the eigenvalues  obtained from inverting
\eqref{eqn:rat_func_series} will inevitably contain errors and
from this we have to recover the sequence $\{\epsilon_n\}$ by first
subtracting off the term asymptotic to effectively $n^2\pi^2$.
For $n=10$ this is the order of a thousand and meantime the information
is in a possibly rapidly decreasing sequence of the difference.
In short, the inverse Sturm-Liouville problem is very mildly ill-conditioned from a space domain/range definition. but the strong masking by the
fixed leading term in the eigenvalue expansion which is independent of $q$
makes recovery of anything more than a few Fourier modes impractical from
other than exceedingly accurate spectral data.

From the existence of the Liouville transform one might conclude that
the picture is identical for the determination of the coefficient $a(x)$.
This is indeed so up to the point where we recover the spectral information.
However, the version of \eqref{eqn:sturm_liouville_eigen} for this case is
instead
\begin{equation}\label{eqn:sturm_liouville_eigen_a}
\lambda_n(q) = {\textstyle \frac{(n-\frac{1}{2})^2\pi^2}{L^2}} + \epsilon_n,
\qquad\mbox{where }\
L = \int_0^1 [a(s)]^{-1/2}\,ds.
\end{equation}
Thus unlike the potential case where the leading and dominant term is
independent of $q(x)$, the leading term now contains information about the
coefficient $a(x)$.
For a fixed error in the data, this can make a substantial difference.

All of the above has been predicated on two factors; that we are in one
space dimension and our underlying parabolic equation is linear.
With a nonlinear reaction term $f(u)$ the whole approach above fails
(and this is precisely our case of interest).
In higher space dimensions there is no sequel to the inverse Sturm-Liouville
theory and much more information is required over the one dimensional case.

There is an approach that can transcend these issues.
This is to look directly at the maps from unknown coefficient to the data:
$F_a[a] \to u(0,t)$ or $F_q[q] \to u(0,t)$.
We can get a sense of the practical invertibility by looking at the
Jacobian matrix $J$ arising from a finite dimensional version of
$F'$, where $F := [F_a\; F_q]$, taken in a direction
$\delta a$ or $\delta q$:
\begin{equation}\label{eqn:F'_map}
\begin{aligned}
 &\hat u_t - \nabla\cdot(a(x)\nabla \hat u) = f'(u)\hat u -
\nabla.(\delta a\nabla u) \\
 &\hat u_t - \triangle \hat u + q(x) \hat{u}= f'(u)\hat u - \delta q\,u \\
\end{aligned}
\end{equation}

An important question is whether this is feasible in respect to even the
theoretical invertibility of the map $F'$:
either as one of its two components or in combination.
We provide a sketch below to show that in the linear case $f(u) = c(x)u$
$F'$ is indeed locally invertible in $\mathbb{R}^d$ for $d>1$,
then later consider the issue of its
practical invertibility by estimating the singular values
of the matrix $J$ representing $F'$.
With this $f$ we can absorb the contribution into the zero'th order term of
$\mathbb{L}$ and use the same notation for the eigenvalues/eigenfunctions
of $\mathbb{L}$ on $\Omega$ with respect to homogeneous impedance boundary
conditions.
With this, the  solution to \eqref{eqn:F'_map} in the potential case is given by
\begin{equation}\label{eqn:F'_map_pot}
%\begin{aligned}
\hat u(x,t) = \sum_1^\infty \int_0^t e^{-\lambda_n(t-\tau)}
\langle \delta q\, u_0,\phi_n\rangle e^{-\lambda_1\tau}\phi_n(x).
%\end{aligned}
\end{equation}
Suppose now we have $\hat u(\tilde{x},t) = 0$ for $x=\tilde{x}\in\partial\Omega$
(if we only have this prescribed on $[0,T]$ then we simply extend
it by zero for $t>T$) and taking Laplace transforms gives
\begin{equation}\label{eqn:F'_map_zero}
\sum_1^\infty \frac{1}{p+\lambda_n}\frac{\phi_n(\tilde{x})}{p+\lambda_1}
\langle \delta q\; u_0,\phi_n\rangle = 0
\end{equation}
for all $p>0$.
Now we can select the point $\tilde{x}$ so that $\phi_n(\tilde{x}) \not=0$
since $\phi_n(x)$ satisfies homogeneous impedance conditions on
$\partial\Omega$ (or more exactly can select the origin of the eigenfunctions
to accomplish this for a prescribed $\tilde{x}$).
Then analyticity in \eqref{eqn:F'_map_zero} shows that
$\langle \delta q\, u_0,\phi_n\rangle = 0$ for all $n$ and completeness
implies that $\delta q \,u_0 =0$ and, by our assumption that $u_0$ was
nonzero, implies that $\delta q = 0$.
An entirely analogous argument holds for the case of $a(x)$.
It also trivially  extends to an initial value $u_0$ that is
a finite combinations of eigenfunctions and with a little more work
 extends to a general initial value $u_0$ that is not
identically zero in any subset of $\Omega$ of positive measure.
Finally, by considering the situation on the interval $[\epsilon,T]$
we have the new initial value $u(x,\epsilon)$ which is strictly positive
from the maximum principle provided only that $u_0(x)$ is non-negative
and non-trivial.
Thus we have proven,
\begin{theorem}\label{thm:q-a-injectivity}
Suppose the nonlinear term $f(u)$ is zero, the initial condition
$u_0(x)\not=0$ is nonnegative in $\Omega$, then the maps $F':q(x) \to u(\tilde{x},t)$
and $F':a(x) \to  u(\tilde{x},t)$ are injective.
\end{theorem}

We have carried out the process of computing the singular values of $F'$
taking the case of one spatial dimension and about a constant value for
$a$ or $q$, namely $a=1$ and $q=0$.
The initial condition was the first eigenfunction for
the associated elliptic operator, here $\sin(\frac{\pi}{2}x)$ and
the basis functions used for both $\delta a$ and $\delta q$ were
$\{\sin(n\pi x)\}_{n=1}^{20}$.
The resulting singular values  are shown in the Figure~\oldref{fig:sv_a-q}
plotted on a $\log_{10}$ scale.
%\magnification=\magstephalf
%\input fonts
%\input pictex
%\input colordvi
\newdimen\xfiglen \newdimen\yfiglen %
\newdimen\xfigdim \newdimen\yfigdim %
\newbox\figurelegendone
\newbox\figureone
\newbox\figuretwo

%\xfiglen = 2.5 true in
%\yfiglen = 2.2 true in
\xfiglen = 2. true in
\yfiglen = 1.76 true in
\xfigdim=0.05\xfiglen
\yfigdim=0.083\yfiglen
%%%%%%%%%%%%%%%%%%%
%
\setbox\figurelegendone=\hbox{
\beginpicture
  \setcoordinatesystem units <\xfigdim,0.8\yfigdim> %point at 0 -0.7
  \setplotarea x from 0 to 1.3, y from 0 to 3
\footnotesize
%\eightpoint
  \put {\Blue{$\bullet$}} [lb] at 0 1 \put {$dt = 0.001$} [lb] at 0.6 1
  \put {\Red{$\star$}} [lb] at 0 2 \put {$dt=0.01$} [lb] at 0.6 2
  \put {\Black{$\circ$}} [lb] at 0 3 \put {$dt=0.1$} [lb] at 0.6 3
\endpicture
}

\setbox\figureone=\vbox{\hsize=\xfiglen  % f_1   Time Trace
\beginpicture
%\sevenpoint
\footnotesize
  \setcoordinatesystem units <\xfigdim,\yfigdim>  %point at 0 -60
  \setplotarea x from 1 to 20, y from -10 to 2
\scriptsize
  \axis bottom shiftedto y=-10 ticks numbered from 4 to 20 by 4 unlabeled short quantity 20 /
  \axis left ticks short numbered from -10 to 2 by 2 /
\small
 \put {$F'[a].\delta a$} [lt] at 8 2
 \put {$\log($sv$(J))$} [lb] at 1.4 2
 \put {$n$} [lb] at 20.1 -9.9
\put {\copy\figurelegendone} [rt] at 19 1
\multiput {\Blue{$\bullet$}} at   % dt = 0.001
    1.0000    2.3208
    2.0000    1.1030
    3.0000    0.2424
    4.0000   -0.4804
    5.0000   -1.0377
    6.0000   -1.5091
    7.0000   -1.9165
    8.0000   -2.2799
    9.0000   -2.6385
   10.0000   -3.0171
   11.0000   -3.4557
   12.0000   -3.9049
   13.0000   -4.4222
   14.0000   -4.9321
   15.0000   -5.5317
   16.0000   -6.1066
   17.0000   -6.8111
   18.0000   -7.4632
   19.0000   -8.3338
   20.0000   -9.0879
/
\multiput {\Red{$\star$}} at   % dt = 0.01
    1.0000    1.8213
    2.0000    0.6035
    3.0000   -0.2546
    4.0000   -0.9714
    5.0000   -1.5166
    6.0000   -1.9627
    7.0000   -2.3249
    8.0000   -2.6345
    9.0000   -2.9759
   10.0000   -3.3475
   11.0000   -3.7803
   12.0000   -4.2192
   13.0000   -4.7276
   14.0000   -5.2265
   15.0000   -5.8167
   16.0000   -6.3805
   17.0000   -7.0753
   18.0000   -7.7162
   19.0000   -8.5766
   20.0000   -9.3194
/
\multiput {$\circ$} at  % dt = 0.1;
    1.0000    1.3257
    2.0000    0.1094
    3.0000   -0.7241
    4.0000   -1.3802
    5.0000   -1.8008
    6.0000   -2.0543
    7.0000   -2.2860
    8.0000   -2.5695
    9.0000   -2.9151
   10.0000   -3.2778
   11.0000   -3.6981
   12.0000   -4.1199
   13.0000   -4.6117
   14.0000   -5.0927
   15.0000   -5.6667
   16.0000   -6.2142
   17.0000   -6.8943
   18.0000   -7.5192
   19.0000   -8.3649
   20.0000   -9.1049
/
\endpicture
}   % end of time trace singular values
\setbox\figuretwo=\vbox{\hsize=\xfiglen  % f_1   Time Trace
\beginpicture
  \setcoordinatesystem units <\xfigdim,\yfigdim>  %point at 0 -60
\scriptsize
  \setplotarea x from 1 to 20, y from -10 to 2
  \axis bottom shiftedto y=-10 ticks numbered from 4 to 20 by 4 unlabeled short quantity 20 /
  \axis left ticks short numbered from -10 to 2 by 2 /
\footnotesize
\small
 \put {$F'[q].\delta q$} [lt] at 8 2
 \put {$\log($sv$(J))$} [lb] at 1.4 2
 \put {$n$} [lb] at 20.1 -9.9
\put {\copy\figurelegendone} [rt] at 19 1
\multiput {\Blue{$\bullet$}} at   % dt = 0.001
    1.0000    2.3030
    2.0000    0.5518
    3.0000   -0.7634
    4.0000   -1.6032
    5.0000   -2.2515
    6.0000   -2.7980
    7.0000   -3.2994
    8.0000   -3.7828
    9.0000   -4.2780
   10.0000   -4.7813
   11.0000   -5.3151
   12.0000   -5.8610
   13.0000   -6.4496
   14.0000   -7.0509
   15.0000   -7.7104
   16.0000   -8.3824
   17.0000   -9.1395
   18.0000   -9.9086
%   19.0000  -10.8266
%   20.0000  -11.7536
/
\multiput {\Red{$\star$}} at   % dt = 0.01
    1.0000    1.8034
    2.0000    0.0524
    3.0000   -1.2576
    4.0000   -2.0844
    5.0000   -2.7100
    6.0000   -3.2271
    7.0000   -3.7016
    8.0000   -4.1627
    9.0000   -4.6399
   10.0000   -5.1268
   11.0000   -5.6469
   12.0000   -6.1799
   13.0000   -6.7570
   14.0000   -7.3468
   15.0000   -7.9957
   16.0000   -8.6567
   17.0000   -9.4033
%   18.0000  -10.1616
%   19.0000  -11.0687
%   20.0000  -11.9874
/
\multiput {$\circ$} at  % dt = 0.1;
   1.0000    1.3083
    2.0000   -0.4409
    3.0000   -1.7027
    4.0000   -2.4214
    5.0000   -2.9079
    6.0000   -3.3164
    7.0000   -3.7288
    8.0000   -4.1544
    9.0000   -4.6070
   10.0000   -5.0725
   11.0000   -5.5731
   12.0000   -6.0865
   13.0000   -6.6450
   14.0000   -7.2163
   15.0000   -7.8481
   16.0000   -8.4929
   17.0000   -9.2248
   18.0000   -9.9675
%   19.0000  -10.8596
%   20.0000  -11.7726
/
\endpicture
}   % end of time trace singular values
%

%\centerline{\hss\copy\figureone\hss\hss\hss\copy\figuretwo\hss}
%%\ninepoint
%\small
%\smallskip
%\centerline{{\bf Singular values for $a(x)$ and $q(x)$
%recovery from time trace data}}
%
%\centerline{{\bf using $F:\delta c \to u(1,t)$ with $\{\delta c =\sin(n\pi x)\}$}}
%

\begin{figure}[ht]
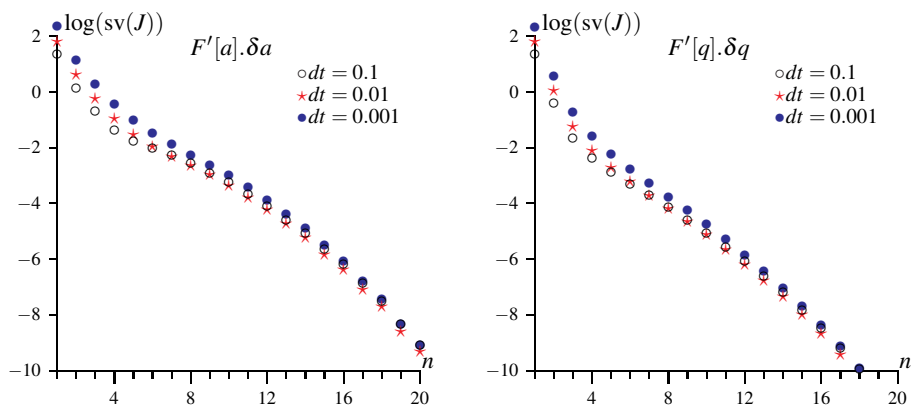

\hbox to \hsize{\hss\copy\figureone\hss\hss\copy\figuretwo\hss}
\caption{\small {\bf Logarithms of the singular values for $a(x)$ and $q(x)$
recovery from time trace data using $F:\delta c \to u(1,t)$ with $\{\delta c =\sin(n\pi x)\}$
}}
\label{fig:sv_a-q}
\end{figure}

The expected  exponential decay of these singular values is cleary evident.
From this data $h(t)$ -- which arose from a problem with the initial data
being the exact lowest eigenfunction and hence the best possible for
reconstructions, we see that only a few of the lowest singular values
are useable under anything but exceedingly high accuracy data.
These correspond to the lowest frequencies in the basis set.
While the same exponential decay is true for both $a(x)$ and $q(x)$, the rate
is greater for $q$, but more crucially, the first singular values in each
case differ considerably - by roughly a factor of ten.
Note also the effect that smaller time steps makes.
It is much less evident in the case shown  due to the solution $u$
to the base problem having only the lowest spectral mode.
In problems where higher modes are present and one has
to recover $\lambda_n$ from $e^{-\lambda_n t}$; the smaller the time step
$dt$ the better one can recover the eigenvalues and coefficients
 (especially the largest index ones sought)
from data error effecting this term.
For example, for $n=10$ and Dirichlet boundary conditions with $a=1$ and $q=0$
this means the factor $e^{-100\pi^2 dt}$ should lie above the noise level.
Of course,  if the diffusion coefficient is substantially smaller
then this changes the above estimate.

Our scaling of this problem where we have assumed the coefficients are of
order unity must be taken into account.
As noted in the introduction, in many physical settings
the diffusion coefficient could be several orders of magnitude smaller
and this would give a corresponding change in the spectrum $\{\lambda_n\}$.
In addition, for diffusion in fluids there is often a significant
dependence of the diffusion coefficient on temperature $u$ in addition
to any spatial dependence.
For example a standard model here is due to Arrhenius
\cite{Arrhenius:1889,Cussler:1997}
that supposes that $a = a_0\,e^{-E/R\,u}$ where $R$ is the
universal gas constant and $E$ the activation energy which is typically
such that $E >> R\,u$.
Of course, incorporating this into our model would dramatically increase
the complexity of any analysis.

It is entirely possible to extend these computation to higher
space dimensions and also in fact to include nonlinear reaction terms
although in the latter situation the above theory and hence the proof of
Theorem~\oldref{thm:q-a-injectivity} will no longer be valid.

This same effect, namely the superior reconstruction of $a(x)$ against
$q(x)$  from their simultaneous recovery from final time data,
 was noticed in \cite{KaltenbacherRundell:2020a}.
There is another ``{\it folk theorem\/}'' at play here.
The dependence of the solution of an elliptic equation on changes in its
coefficients is far greater from those appearing in higher order terms
(as in the $a(x)$ here) than in terms appearing in lower order coefficients.
One should expect that this translates into data terms arising from
projections of the solution onto a surface,
and in turn expect that one should be able to
more easily reconstruct coefficients appearing in the highest derivatives of
the operator as distinct from those appearing in the lowest order term.
from knowledge of projections of the solution onto a surface.

The reconstruction of both $a(x)$ and $f(u)$ requires two experiments
providing solutions and the corresponding data obtained from them.
The theoretical results of section~\ref{subsec:axfu-twofinal} shows this can
be sufficient under appropriate conditions.
In many applications one could attempt an oversampling by taking measurements
from $m>2$ data runs generating $m$ solutions providing overposed data.
Under the unrealistic situation of no data error this would simply lead
to redundancy in terms of uniqueness but might very well lead to
superior reconstructions due to the fact that exact optimality conditions
are difficult to determine and more measurements would likely lead to a
better case being obtained in one of the experiments that in turn would lead
to a larger number of singular values above a given threshold.

In the case of noise in the data one would also expect that oversampling
by utilizing $N$ solutions would lead to an effective lower noise rate
and in consequence superior reconstructions.
In the case of $N=2$ and 1\% random noise there is certainly a detectable
variation in the corresponding reconstructions and even a modest increase
in the number of measurements makes a significant reduction.

The extent to which this can be quantified depends on several factors.
One of these is the underlying source of the noise and its probability
density distribution.  If this is Gaussian then $N$ samples would reduce
the effective error by a factor of $\sqrt{N}$, but this holds only
for large $N$ and other distributions would give different answers.
However we are dealing with a far from linear process.
The reconstruction process is itself highly nonlinear in addition to the basic
forwards map being so due to the nonlinear term $f(u)$.
As a result of this, quantification of the overall uncertainty is far
from an easy question.

\section{Reconstructions}\label{sect:recons}

We will show the results of numerical experiments to recover both
$a$ and $f$  with two different types of data measurements.

The first of these is when we are only able to obtain ``census-type''
information and thus measure $g_u := u(x,T)$ and $g_v := v(x,T)$ for some
fixed time $T$ and as the result of two experiments obtained by
altering the initial, boundary or forcing term conditions.

The second is when we are able to measure multiple types of data from a
single experimental run: specifically both the final data and a
boundary measurement of either the solution $u$ or its normal derivative at a
fixed point $\tilde  x\in\partial\Omega$.

To procure data for the reconstructions a direct solver based on a
Crank-Nicolson scheme produced output values and data
values were produced from these by sampling at a relatively small
number $N_x$ and $N_t$ of points in both the spatial and temporal directions:
This sampled data was then both interpolated to a full working size to obtain
data values $g(x)$ and $h(t)$ commensurate to the grid being used by the solver
used in the inverse problem and such that the value of $g(x)$ was
smoothed by an $H^2$ filter while that of $h(t)$ by an $H^1$ filter.

All the reconstructed solutions we show are set in one space dimension.
This is in part to make the graphical results more transparent but also because
some of the algorithms have an $\mathbb{R}^1$ restriction from a
computational aspect and certainly from an analysis one as noted in the
previous section.
However, we will make note when extensions to higher space dimensions are feasible.

\subsection{Final time with two runs}\label{sect:fitifiti}

We shall pose the setting in one space variable as the exposition is simpler
as is the graphical representation of the reconstructions obtained.

The problem is to identify $a(x)$, $f(u)$ in
\begin{equation*}%\label{eqn:u-v_pair}
\begin{aligned}
u_t-(a u_x)_x &=f(u) + r_u(x,t,u)\\
v_t-(a v_x)_x &=f(v) + r_v(x,t,v)\\
\end{aligned}
\end{equation*}
where $r_u$ and $r_v$ are known driving terms and where $u(x,t)$ and $v(x,t)$
are subject to homogeneous impedance boundary conditions
$$
a\frac{\partial u}{\partial x} +\alpha u =0
\qquad
a\frac{\partial v}{\partial x} +\beta v =0
$$
and initial values
$$
u(x,0)=u_0(x) \qquad
v(x,0)=v_0(x) \qquad
$$
from observations $g_u(x)=u(x,T)$, $g_v=v(x,T)$.

Some restrictions must be imposed on this data.
An ideal situation will include the fact $\nabla g_u$ and $\nabla g_v$
will not vanish on the interior of $\Omega$.
This can be accomplished by varying the impedance coefficients
over the boundary $\partial\Omega$ and by imposing a larger positive flux into
the interior where $\alpha$, $\beta$ are smallest - for example Neumann
conditions.

Based on the fixed point operator defined in \eqref{eqn:fp_fiti-fiti}, we now
derive some reconstruction schemes based on different ways of resolving
\eqref{eqn:MFb} in an efficient way,
and provide computational results obtained with these schemes.

The first method is one of sequential iteration: recovering first
a new approximation to $a(x)$ using one direct solve
then followed by an update of $f(u)$ using a second direct solve,
but now using the new value of $a$.
More precisely,
let $f^0$, $a^0$ to some initial approximation, then for $k=0,1,2,\ldots$
\begin{itemize}
\item{}
compute $u_t(x,T;a^k,f^k)$ by solving \eqref{eqn:basic_pde_parabolic}
using the first set of initial/boundary conditions then differentiate
$u$ with respect to time at $t=T$ and for all $x\in\Omega$.
\item{}
Set $\phi(x):= u_t(x,T;a^k,f^k) - f^k\bigl(g_u(x)\bigr) - r_u(x,T,g_u)$
so that $\bigl((a(x) g_u'(x)\bigr)' = \phi(x)$.
Integrate this equation over $(0,x)$ using the boundary conditions
inherited by $g_u$ and $u(x,T;a^k,f_k)$  to obtain
$\Phi(x) := \int_0^x [u_t(s,T) - f^k(g_u(s)) - r_u(s,T)]\, ds$
which then gives the update
\begin{equation}\label{eqn:iter1}
a^{k+1}(x)=\frac{\Phi(x)}{g_u'(x)}
\end{equation}
\item{}
Use this new approximation $a^{k+1}$ for $a(x)$ together with the previous
value $f^k$
for $f$ to compute $v(x,T;a^{k+1},f^k)$ from the imposed second set of
initial/boundary conditions and thereby obtain the value of $v_t(x,T)$.
\item{}
We now update $f$ by
$$
f^{k+1}(g_v)= v_t(x,T)-\bigl(a^{k+1} g'_v(x)\bigr)'
$$
\end{itemize}
One drawback of this method is the need to differentiate the just-updated
value of $a(x)$ that can lead to numerical instabilities.
Here is a way to avoid this
and works well in a single space dimension.

We again let $f^0$, $a^0$ be some initial approximation, then for
$k=0,1,2,\ldots$
\begin{itemize}
\item{}
Compute both $u_t(T)$ and $v_t(T)$ by solving
\eqref{eqn:basic_pde_parabolic}, using the two sets of initial/boundary values
\begin{equation}\label{eqn:u-v_pair}
\begin{aligned}
u_t-(a u_x)_x &=f(u) + r_u(x,t,u)\qquad u(x,0)=u_0,\quad
a\frac{\partial u}{\partial x} +\alpha u = k_u(x,t)\\
v_t-(a v_x)_x &=f(v) + r_v(x,t,v)\qquad v(x,0) = v_0\quad
a\frac{\partial u}{\partial x} +\beta v = k_v(x,t)\\
\end{aligned}
\end{equation}
and hence obtain $u_t(x,T;a^k,f^k)$ and $v_t(x,;a^k,f^kT)$.
\item{}
Let $W(x) = g_u(x) g'_v(x) - g_v(x) g'_u(x)$ and
$W_p(x) = g_u(x) g''_v(x) - g_v(x) g''_u(x)$.
Multiply the first equation in \eqref{eqn:u-v_pair} by $g_u$,
the second by $g_v$ then subtract obtaining
$$
\bigl(a(x)g_u'(x)\bigr)'g_v(x) - \bigl(a(x)g_v'(x)\bigr)'g_u(x) = \psi(x)
$$
where
\begin{equation*}
\begin{aligned}
\psi(x) &= [u_t(x,T;a^k,f^k) - r_u(x,T,g_u) - f(g_u(x))]g_v(x) \\
&\quad -[v_t(x,T;a^k,f^k) - r_v(x,T,g_v) - f(g_v(x))]g_u(x)
\end{aligned}
\end{equation*}
Now integrating over $(0,x)$ we obtain
\begin{equation}\label{eqn:aW}
a(x)W(x) = \int_0^x \psi(s)\,ds
\end{equation}
from which an update $a^{k+1}(x)$ can be obtained by division by $W$.
It is obviously an advantage if $W$ doesn't vanish and this condition can
be obtained by controlling the values of $\alpha$, $\beta$ and
the input flux values $k_u$ and $k_v$.
(Even if $W$ has isolated zeros there are several approaches that can be
taken to still resolve \eqref{eqn:aW} in a stable manner.)
\item{}
To recover the next approximation for $f$ we multiply the first equation in
\eqref{eqn:u-v_pair} by $g'_v$, the second by $g'_u$ then subtract obtaining
$$
\tilde\phi = (u_t - r_u(x,T)g_v' - (v_t - r_v(x,T))g_u' - a(x)W_p(x)
$$
Then $f^{k+1}$ is determined by
\begin{equation}\label{eqn:f-update}
f^{k+1}\bigl(g_u(x)\bigr)g_v'(x) - f^{k+1}\bigl(g_v(x))\bigr)g'_u(x) =
\tilde\phi(x).
\end{equation}
The effect of this is to eliminate the derivative of the just computed $a(x)$.
\item{}
Solving \eqref{eqn:f-update} is not completely straightforward.
It is a delayed argument equation for $f$ if the maximum
value achieved for both $u$ and $v$ occurs at or near the same
point on $\partial\Omega$ but in fact the resolution of the entire scheme
is better if both $u$ and $v$ are ``more independent.''
\item{}
One approach is to use a basis representation for $f$ to  resolve this
situation.
Another, which works well if one of $g'_u$ or $g'_v$ is monotone,
is to use the approximation $\tilde f_1(u) = f^k(u)$ and then iterate
successively on \eqref{eqn:f-update} by
\begin{equation}\label{eqn:f_successive-update}
\tilde f_{j+1}\bigl(g_u(x)\bigr)g'_v(x) =
\tilde f_j\bigl(g_v(x))\bigr)g'_u(x) + {\tilde\phi}(x).
\end{equation}
This method worked well over a wide variety of examples of functions $f(u)$
and one such pairing of $a(x)$ and $f(u)$ recoveries is shown in
Figure~\oldref{fig:fiti_fiti} below.
\end{itemize}

%  \PiCTeX file for BB6
\input colordvi

\input pictex
\font\smallsymbol = cmmi8
\newdimen\xfiglen \newdimen\yfiglen
%\xfiglen=2.6 true in
%\yfiglen=1.6 true in
\xfiglen=2.08 true in
\yfiglen=1.28 true in
\newbox\figurelegendone
\newbox\figurelegendtwo
\newbox\figurelegendthree
\newbox\figurelegendfour
\newbox\figurelegendfive
\newbox\figurelegendsix
\newbox\figurelegendseven
\newbox\figurelegendeight
\newbox\figureone
\newbox\figureoneb
\newbox\figuretwo
\newbox\figuretwob
\newbox\figurethree 
\newbox\figurethreeb 
\newbox\figurefour
\newbox\figurefive
\newbox\figuresix
%%%%%%%%%%%%%%%%%%%
%
\setbox\figurelegendone=\hbox{
\beginpicture
  \setcoordinatesystem units <0.07\xfiglen,0.08\yfiglen> %point at 0 -0.7
  \setplotarea x from 0 to 2, y from 0 to 4
\footnotesize
\linethickness=0.6pt
\eightrm
   \setdashes <3pt>  \putrule from 0 3 to 1.5 3
   \setsolid  
   \Blue{\relax \putrule from 0 2 to 1.5 2 \relax}\relax
   \Red{\relax \putrule from 0 1 to 1.5 1 \relax}\relax
   \OliveGreen{\relax \putrule from 0 0 to 1.5 0 \relax}
  \setsolid
  \put {$a_{\hbox{act}}$}  [l] at 2 3
  \put {$a_{\hbox{iter}\,1}$}  [l] at 2 2
  \put {$a_{\hbox{iter}\,2}$}  [l] at 2 1
  \put {$a_{\mbox{iter}\,10}$} [l] at 2 0
\endpicture
}
\setbox\figurelegendtwo=\hbox{
\small
\beginpicture
  \setcoordinatesystem units <0.07\xfiglen,0.09\yfiglen> %point at 0 -0.7
  \setplotarea x from 0 to 2, y from 0 to 3
\small
\linethickness=0.6pt
\eightrm
   \setdashes <3pt>  \putrule from 0 3 to 1.5 3
   \setsolid  
   \Blue{\relax\putrule from 0 2 to 1.5 2 \relax}\relax
   \Red{\relax\putrule from 0 1 to 1.5 1 \relax}\relax
   \OliveGreen{\relax\putrule from 0 0 to 1.5 0 \relax}\relax
  \setsolid
  \put {$f_{\hbox{act}}$}  [l] at 2 3
  \put {$f_{\hbox{iter}\,1}$}  [l] at 2 2
  \put {$f_{\hbox{iter}\,10}$}  [l] at 2 1
  \put {$f_{\hbox{iter}\,50}$} [l] at 2 0
\endpicture
}
\setbox\figurelegendthree=\hbox{
\small
\beginpicture
  \setcoordinatesystem units <0.07\xfiglen,0.09\yfiglen> %point at 0 -0.7
  \setplotarea x from 0 to 2, y from 0 to 4
\footnotesize
\linethickness=0.6pt
\eightrm
   \setdashes <3pt>  \putrule from 0 4 to 1.5 4
   \setsolid  
   \Blue{\relax \putrule from 0 3 to 1.5 3 \relax}\relax
   \OliveGreen{\relax \putrule from 0 2 to 1.5 2 \relax}
   \Maroon{\relax \putrule from 0 1 to 1.5 1 \relax}\relax
   \Red{\relax \putrule from 0 0 to 1.5 0 \relax}\relax
  \setsolid
  \put {$a_{\hbox{act}}$}  [l] at 2 4
  \put {$a_{\hbox{iter}\,1}$}  [l] at 2 3
  \put {$a_{\hbox{iter}\,4}$}  [l] at 2 2
  \put {$a_{\hbox{iter}\,6}$}  [l] at 2 1
  \put {$a_{\mbox{iter}\,12}$} [l] at 2 0
\endpicture
}
\setbox\figurelegendfour=\hbox{
\small
\beginpicture
  \setcoordinatesystem units <0.07\xfiglen,0.09\yfiglen> %point at 0 -0.7
  \setplotarea x from 0 to 2, y from 0 to 3
\small
\linethickness=0.6pt
\eightrm
   \setdashes <3pt>  \putrule from 0 4 to 1.5 4
   \setsolid  
   \Blue{\relax\putrule from 0 3 to 1.5 3 \relax}\relax
   \OliveGreen{\relax\putrule from 0 2 to 1.5 2 \relax}\relax
   \Maroon{\relax\putrule from 0 1 to 1.5 1 \relax}\relax
   \Red{\relax\putrule from 0 0 to 1.5 0 \relax}\relax
  \setsolid
  \put {$f_{\hbox{act}}$}  [l] at 2 4
  \put {$f_{\hbox{iter}\,1}$}  [l] at 2 3
  \put {$f_{\hbox{iter}\,4}$}  [l] at 2 2
  \put {$f_{\hbox{iter}\,6}$} [l] at 2 1
  \put {$f_{\hbox{iter}\,12}$} [l] at 2 0
\endpicture
}
\setbox\figurelegendfive=\hbox{
\small
\beginpicture
  \setcoordinatesystem units <0.07\xfiglen,0.08\yfiglen> %point at 0 -0.7
  \setplotarea x from 0 to 2, y from 0 to 4
\footnotesize
\linethickness=0.6pt
\eightrm
   \setdashes <3pt>  \putrule from 0 3 to 1.5 3
   \setsolid  
   \Blue{\relax \putrule from 0 2 to 1.5 2 \relax}\relax
   \Red{\relax \putrule from 0 1 to 1.5 1 \relax}\relax
  % \OliveGreen{\relax \putrule from 0 0 to 1.5 0 \relax}
  \setsolid
  \put {$a_{\hbox{act}}$}  [l] at 2 3
  \put {$a_{\hbox{iter}\,1}$}  [l] at 2 2
  \put {$a_{\hbox{iter}\,4}$}  [l] at 2 1
\endpicture
}
\setbox\figurelegendsix=\hbox{
\small
\beginpicture
  \setcoordinatesystem units <0.07\xfiglen,0.09\yfiglen> %point at 0 -0.7
  \setplotarea x from 0 to 2, y from 0 to 3
\small
\linethickness=0.6pt
\eightrm
   \setdashes <3pt>  \putrule from 0 3 to 1.5 3
   \setsolid  
   \Blue{\relax\putrule from 0 2 to 1.5 2 \relax}\relax
   \Red{\relax\putrule from 0 1 to 1.5 1 \relax}\relax
  % \OliveGreen{\relax\putrule from 0 0 to 1.5 0 \relax}\relax
  \setsolid
  \put {$f_{\hbox{act}}$}  [l] at 2 3
  \put {$f_{\hbox{iter}\,1}$}  [l] at 2 2
  \put {$f_{\hbox{iter}\,4}$}  [l] at 2 1
\endpicture
}
\setbox\figurelegendseven=\hbox{
\small
\beginpicture
  \setcoordinatesystem units <0.07\xfiglen,0.07\yfiglen> %point at 0 -0.7
  \setplotarea x from 0 to 2, y from 2 to 4
\footnotesize
\linethickness=0.6pt
   \setdashes <3pt>  \putrule from 0 4 to 1.5 4
   \setsolid  
   \Brown{\relax \putrule from 0 3 to 1.5 3 \relax}\relax
   \Maroon{\relax \putrule from 0 2 to 1.5 2 \relax}\relax
   \Red{\relax \putrule from 0 1 to 1.5 1 \relax}\relax
  \setsolid
\scriptsize
  \put {$a_{\hbox{act}}$}  [l] at 2 4
  \put {$a_{\hbox{recon}}\quad \beta = 20$}  [l] at 2 3
  \put {$a_{\hbox{recon}}\quad \beta = 10$}  [l] at 2 2
  \put {$a_{\hbox{recon}}\quad \beta = \phantom{1}5$}  [l] at 2 1
\endpicture
}
\setbox\figurelegendeight=\hbox{
\small
\beginpicture
  \setcoordinatesystem units <0.07\xfiglen,0.08\yfiglen> %point at 0 -0.7
  \setplotarea x from 0 to 2, y from 1 to 6
\footnotesize
\linethickness=0.6pt
   \setdashes <3pt>  \putrule from 0 6 to 1.5 6
   \setsolid  
   \Blue{\relax \putrule from 0 1 to 1.5 1 \relax}\relax
   \OliveGreen{\relax \putrule from 0 2 to 1.5 2 \relax}\relax
   \Red{\relax \putrule from 0 3 to 1.5 3 \relax}\relax
   \Maroon{\relax \putrule from 0 4 to 1.5 4 \relax}\relax
   \Brown{\relax \putrule from 0 5 to 1.5 5 \relax}\relax
\scriptsize
  \setsolid
  \put {$f_{\hbox{act}}$}  [l] at 2 6
  \put {$f_{\hbox{recon}}\quad \beta = 20$}  [l] at 2 5
  \put {$f_{\hbox{recon}}\quad \beta = 10$}  [l] at 2 4
  \put {$f_{\hbox{recon}}\quad \beta = \phantom{1}5$}  [l] at 2 3
  \put {$f_{\hbox{recon}}\quad \beta = \phantom{1}2$}  [l] at 2 2
  \put {$f_{\hbox{recon}}\quad \beta = \phantom{1}1$}  [l] at 2 1
\endpicture
}
\setbox\figureone=\vbox{\hsize=\xfiglen
% $a(x) =   1 + 1*exp(-2*x).*sin(3*pi*x);$
% $f(u) = 2.5*u  + 0.6*sin(3*u)$
% Pointwise evaluation in fiti_fiti
\beginpicture
%\eightrm
  \setcoordinatesystem units <\xfiglen,0.9\yfiglen>  point at 0.0 0.5
  \setplotarea x from 0 to 1, y from 0.5 to 2
\scriptsize
  \axis bottom shiftedto y=0.5 ticks short numbered from 0 to 1 by 0.2 /
  \axis left ticks short numbered from 0.5 to 2 by 0.5 /
\footnotesize
\put {\copy\figurelegendone} [rb] at 0.97 0.55
\put {$a(x)$} [lt] at 0.02 2
\setdashes <3pt>
\Black{
\plot    % actual $a(x)
         0    1.0000
    0.0167    1.1513
    0.0333    1.2891
    0.0500    1.4108
    0.0667    1.5144
    0.0833    1.5986
    0.1000    1.6624
    0.1167    1.7056
    0.1333    1.7284
    0.1500    1.7317
    0.1667    1.7165
    0.1833    1.6845
    0.2000    1.6375
    0.2167    1.5777
    0.2333    1.5073
    0.2500    1.4289
    0.2667    1.3448
    0.2833    1.2576
    0.3000    1.1696
    0.3167    1.0830
    0.3333    1.0000
    0.3500    0.9223
    0.3667    0.8516
    0.3833    0.7891
    0.4000    0.7359
    0.4167    0.6927
    0.4333    0.6599
    0.4500    0.6377
    0.4667    0.6260
    0.4833    0.6243
    0.5000    0.6321
    0.5167    0.6486
    0.5333    0.6727
    0.5500    0.7034
    0.5667    0.7395
    0.5833    0.7798
    0.6000    0.8230
    0.6167    0.8677
    0.6333    0.9129
    0.6500    0.9574
    0.6667    1.0000
    0.6833    1.0399
    0.7000    1.0762
    0.7167    1.1083
    0.7333    1.1356
    0.7500    1.1578
    0.7667    1.1746
    0.7833    1.1860
    0.8000    1.1920
    0.8167    1.1929
    0.8333    1.1889
    0.8500    1.1804
    0.8667    1.1680
    0.8833    1.1523
    0.9000    1.1337
    0.9167    1.1131
    0.9333    1.0909
    0.9500    1.0679
    0.9667    1.0447
    0.9833    1.0219
    1.0000    1.0000
 /\relax}\relax
\setsolid
\Blue{\relax   % first iteration of $a(x)
\plot
         0    1.0000
    0.0167    1.1519
    0.0333    1.2907
    0.0500    1.4144
    0.0667    1.5209
    0.0833    1.6088
    0.1000    1.6772
    0.1167    1.7257
    0.1333    1.7544
    0.1500    1.7638
    0.1667    1.7549
    0.1833    1.7290
    0.2000    1.6879
    0.2167    1.6335
    0.2333    1.5679
    0.2500    1.4935
    0.2667    1.4125
    0.2833    1.3275
    0.3000    1.2408
    0.3167    1.1548
    0.3333    1.0715
    0.3500    0.9930
    0.3667    0.9210
    0.3833    0.8571
    0.4000    0.8024
    0.4167    0.7580
    0.4333    0.7244
    0.4500    0.7019
    0.4667    0.6906
    0.4833    0.6900
    0.5000    0.6997
    0.5167    0.7188
    0.5333    0.7463
    0.5500    0.7809
    0.5667    0.8213
    0.5833    0.8662
    0.6000    0.9140
    0.6167    0.9633
    0.6333    1.0127
    0.6500    1.0609
    0.6667    1.1066
    0.6833    1.1487
    0.7000    1.1864
    0.7167    1.2190
    0.7333    1.2458
    0.7500    1.2664
    0.7667    1.2809
    0.7833    1.2890
    0.8000    1.2911
    0.8167    1.2874
    0.8333    1.2784
    0.8500    1.2647
    0.8667    1.2469
    0.8833    1.2258
    0.9000    1.2021
    0.9167    1.1766
    0.9333    1.1500
    0.9500    1.1231
    0.9667    1.0966
    0.9833    1.0712
    1.0000    1.0473
 /\relax}\relax
\Red{\relax   % second iteration of $a(x)
\plot
         0    1.0000
    0.0167    1.1513
    0.0333    1.2883
    0.0500    1.4088
    0.0667    1.5109
    0.0833    1.5933
    0.1000    1.6551
    0.1167    1.6963
    0.1333    1.7173
    0.1500    1.7188
    0.1667    1.7021
    0.1833    1.6689
    0.2000    1.6210
    0.2167    1.5607
    0.2333    1.4903
    0.2500    1.4122
    0.2667    1.3288
    0.2833    1.2426
    0.3000    1.1559
    0.3167    1.0709
    0.3333    0.9895
    0.3500    0.9137
    0.3667    0.8450
    0.3833    0.7846
    0.4000    0.7335
    0.4167    0.6927
    0.4333    0.6624
    0.4500    0.6431
    0.4667    0.6345
    0.4833    0.6365
    0.5000    0.6482
    0.5167    0.6683
    0.5333    0.6956
    0.5500    0.7286
    0.5667    0.7659
    0.5833    0.8060
    0.6000    0.8477
    0.6167    0.8899
    0.6333    0.9317
    0.6500    0.9725
    0.6667    1.0114
    0.6833    1.0476
    0.7000    1.0804
    0.7167    1.1092
    0.7333    1.1334
    0.7500    1.1528
    0.7667    1.1672
    0.7833    1.1764
    0.8000    1.1808
    0.8167    1.1803
    0.8333    1.1754
    0.8500    1.1664
    0.8667    1.1538
    0.8833    1.1382
    0.9000    1.1201
    0.9167    1.1001
    0.9333    1.0788
    0.9500    1.0567
    0.9667    1.0344
    0.9833    1.0124
    1.0000    0.9910
 /\relax}\relax
\OliveGreen{\plot  % 10th iteration of $a(x)$
        0    1.0000
    0.0167    1.1514
    0.0333    1.2889
    0.0500    1.4102
    0.0667    1.5134
    0.0833    1.5972
    0.1000    1.6607
    0.1167    1.7037
    0.1333    1.7266
    0.1500    1.7300
    0.1667    1.7153
    0.1833    1.6839
    0.2000    1.6376
    0.2167    1.5788
    0.2333    1.5095
    0.2500    1.4322
    0.2667    1.3493
    0.2833    1.2632
    0.3000    1.1764
    0.3167    1.0909
    0.3333    1.0089
    0.3500    0.9322
    0.3667    0.8624
    0.3833    0.8008
    0.4000    0.7485
    0.4167    0.7062
    0.4333    0.6746
    0.4500    0.6537
    0.4667    0.6436
    0.4833    0.6440
    0.5000    0.6540
    0.5167    0.6727
    0.5333    0.6990
    0.5500    0.7316
    0.5667    0.7692
    0.5833    0.8103
    0.6000    0.8537
    0.6167    0.8982
    0.6333    0.9426
    0.6500    0.9860
    0.6667    1.0275
    0.6833    1.0661
    0.7000    1.1010
    0.7167    1.1316
    0.7333    1.1572
    0.7500    1.1777
    0.7667    1.1928
    0.7833    1.2024
    0.8000    1.2068
    0.8167    1.2060
    0.8333    1.2005
    0.8500    1.1908
    0.8667    1.1772
    0.8833    1.1605
    0.9000    1.1412
    0.9167    1.1201
    0.9333    1.0976
    0.9500    1.0744
    0.9667    1.0512
    0.9833    1.0283
    1.0000    1.0063
 /\relax}\relax
\endpicture
}
\setbox\figureoneb=\vbox{\hsize=\xfiglen
\beginpicture
  \setcoordinatesystem units <0.4\xfiglen,0.21\yfiglen>  point at 0 0
  \setplotarea x from 0 to 2.5, y from 0 to 6.4
\scriptsize
  \axis bottom shiftedto y=0 ticks short numbered from 0 to 2.5 by 0.5 /
  \axis left ticks short numbered from 0 to 6 by 1 /
\put {\copy\figurelegendtwo} [rb] at 2.4 0.3
\footnotesize
\put {$f(u)$} [lt] at 0.03 6.4
\setdashes <3pt>
\Black{\relax
\plot
         0         0
    0.0626    0.2686
    0.1252    0.5333
    0.1879    0.7902
    0.2505    1.0359
    0.3131    1.2671
    0.3757    1.4813
    0.4384    1.6764
    0.5010    1.8511
    0.5636    2.0047
    0.6262    2.1374
    0.6889    2.2499
    0.7515    2.3439
    0.8141    2.4215
    0.8767    2.4855
    0.9394    2.5391
    1.0020    2.5861
    1.0646    2.6302
    1.1272    2.6754
    1.1898    2.7256
    1.2525    2.7846
    1.3151    2.8558
    1.3777    2.9422
    1.4403    3.0462
    1.5030    3.1698
    1.5656    3.3140
    1.6282    3.4794
    1.6908    3.6656
    1.7535    3.8715
    1.8161    4.0954
    1.8787    4.3350
    1.9413    4.5874
    2.0040    4.8491
    2.0666    5.1164
    2.1292    5.3855
    2.1918    5.6524
    2.2544    5.9133
    2.3171    6.1643
    2.3797    6.4023
 /\relax}\relax
\setsolid
\Blue{\relax
\plot
         0         0
    0.0626    0.2925
    0.1252    0.5477
    0.1879    0.7675
    0.2505    0.9537
    0.3131    1.1080
    0.3757    1.2323
    0.4384    1.3284
    0.5010    1.3980
    0.5636    1.4431
    0.6262    1.4653
    0.6889    1.4666
    0.7515    1.4486
    0.8141    1.4132
    0.8767    1.3622
    0.9394    1.2974
    1.0020    1.2207
    1.0646    1.1337
    1.1272    1.0384
    1.1898    0.9400
    1.2525    0.8927
    1.3151    0.9799
    1.3777    1.2093
    1.4403    1.5357
    1.5030    1.9128
    1.5656    2.2945
    1.6282    2.6343
    1.6908    2.8860
    1.7535    3.0097
    1.8161    3.0540
    1.8787    3.0894
    1.9413    3.1195
    2.0040    3.1428
    2.0666    3.1582
    2.1292    3.1646
    2.1918    3.1610
    2.2544    3.1463
    2.3171    3.1195
    2.3797    3.0794
 /\relax}\relax
\Red{\relax % 10th iteration of $f$
\plot
        0         0
    0.0626    0.2832
    0.1252    0.5362
    0.1879    0.7606
    0.2505    0.9582
    0.3131    1.1306
    0.3757    1.2796
    0.4384    1.4068
    0.5010    1.5139
    0.5636    1.6027
    0.6262    1.6748
    0.6889    1.7319
    0.7515    1.7757
    0.8141    1.8079
    0.8767    1.8302
    0.9394    1.8442
    1.0020    1.8518
    1.0646    1.8546
    1.1272    1.8542
    1.1898    1.8541
    1.2525    1.8813
    1.3151    1.9773
    1.3777    2.1469
    1.4403    2.3691
    1.5030    2.6227
    1.5656    2.8864
    1.6282    3.1387
    1.6908    3.3585
    1.7535    3.5274
    1.8161    3.6697
    1.8787    3.8185
    1.9413    3.9722
    2.0040    4.1276
    2.0666    4.2828
    2.1292    4.4359
    2.1918    4.5853
    2.2544    4.7289
    2.3171    4.8651
    2.3797    4.9920
 /\relax}\relax
\OliveGreen{\relax
\plot
         0         0
    0.0626    0.3475
    0.1252    0.6580
    0.1879    0.9340
    0.2505    1.1780
    0.3131    1.3924
    0.3757    1.5796
    0.4384    1.7421
    0.5010    1.8824
    0.5636    2.0028
    0.6262    2.1058
    0.6889    2.1938
    0.7515    2.2694
    0.8141    2.3349
    0.8767    2.3928
    0.9394    2.4456
    1.0020    2.4956
    1.0646    2.5454
    1.1272    2.5973
    1.1898    2.6539
    1.2525    2.7180
    1.3151    2.7942
    1.3777    2.9043
    1.4403    3.0617
    1.5030    3.2521
    1.5656    3.4591
    1.6282    3.6665
    1.6908    3.8590
    1.7535    4.0494
    1.8161    4.2530
    1.8787    4.4680
    1.9413    4.6943
    2.0040    4.9301
    2.0666    5.1713
    2.1292    5.4139
    2.1918    5.6539
    2.2544    5.8870
    2.3171    6.1092
    2.3797    6.3164
 /\relax}\relax
\endpicture
}
\setbox\figuretwo=\vbox{\hsize=\xfiglen
% $a(x) =   1 + 1*exp(-2*x).*sin(3*pi*x);$
% $f(u) = 20*u.*exp(-4*u);
% Pointwise evaluation in fiti_titr
\beginpicture
%\eightrm
%\small
  \setcoordinatesystem units <\xfiglen,0.75\yfiglen>  point at 0.0 0.4
  \setplotarea x from 0 to 1, y from 0.4 to 2
\scriptsize
  \axis bottom shiftedto y=0.4 ticks short numbered from 0 to 1 by 0.2 /
  \axis left ticks short numbered from 0.5 to 2 by 0.5 /
\put {\copy\figurelegendthree} [rt] at 0.97 2
\footnotesize
\setquadratic
\put {$a(x)$} [l] at 0.02 2
\setdashes <3pt>
\Black{
\plot    % actual $a(x)$
         0    1.0000
    0.0250    1.1526
    0.0500    1.2939
    0.0750    1.4212
    0.1000    1.5319
    0.1250    1.6240
    0.1500    1.6963
    0.1750    1.7480
    0.2000    1.7787
    0.2250    1.7887
    0.2500    1.7788
    0.2750    1.7502
    0.3000    1.7046
    0.3250    1.6438
    0.3500    1.5701
    0.3750    1.4860
    0.4000    1.3940
    0.4250    1.2968
    0.4500    1.1970
    0.4750    1.0973
    0.5000    1.0000
    0.5250    0.9075
    0.5500    0.8217
    0.5750    0.7445
    0.6000    0.6774
    0.6250    0.6215
    0.6500    0.5777
    0.6750    0.5463
    0.7000    0.5277
    0.7250    0.5216
    0.7500    0.5276
    0.7750    0.5450
    0.8000    0.5727
    0.8250    0.6095
    0.8500    0.6542
    0.8750    0.7052
    0.9000    0.7610
    0.9250    0.8200
    0.9500    0.8805
    0.9750    0.9410
    1.0000    1.0000
/\relax}\relax
\setsolid
\Blue{\relax   % first iteration of $a(x)
\plot
         0    0.5042
    0.0250    0.5675
    0.0500    0.6352
    0.0750    0.7014
    0.1000    0.7660
    0.1250    0.8242
    0.1500    0.8725
    0.1750    0.9134
    0.2000    0.9485
    0.2250    0.9744
    0.2500    0.9890
    0.2750    0.9947
    0.3000    0.9936
    0.3250    0.9838
    0.3500    0.9641
    0.3750    0.9372
    0.4000    0.9054
    0.4250    0.8680
    0.4500    0.8250
    0.4750    0.7792
    0.5000    0.7326
    0.5250    0.6858
    0.5500    0.6397
    0.5750    0.5967
    0.6000    0.5588
    0.6250    0.5268
    0.6500    0.5019
    0.6750    0.4855
    0.7000    0.4789
    0.7250    0.4821
    0.7500    0.4952
    0.7750    0.5182
    0.8000    0.5507
    0.8250    0.5917
    0.8500    0.6396
    0.8750    0.6934
    0.9000    0.7521
    0.9250    0.8138
    0.9500    0.8763
    0.9750    0.9385
    1.0000    0.9999
/\relax}\relax
\setsolid
\OliveGreen{\relax   % fourth iteration of $a$
\plot
         0    0.8887
    0.0250    1.0030
    0.0500    1.1156
    0.0750    1.2173
    0.1000    1.3084
    0.1250    1.3810
    0.1500    1.4306
    0.1750    1.4621
    0.2000    1.4793
    0.2250    1.4780
    0.2500    1.4561
    0.2750    1.4192
    0.3000    1.3716
    0.3250    1.3120
    0.3500    1.2403
    0.3750    1.1618
    0.4000    1.0802
    0.4250    0.9959
    0.4500    0.9099
    0.4750    0.8261
    0.5000    0.7469
    0.5250    0.6730
    0.5500    0.6054
    0.5750    0.5460
    0.6000    0.4960
    0.6250    0.4557
    0.6500    0.4253
    0.6750    0.4054
    0.7000    0.3964
    0.7250    0.3982
    0.7500    0.4106
    0.7750    0.4335
    0.8000    0.4670
    0.8250    0.5104
    0.8500    0.5626
    0.8750    0.6230
    0.9000    0.6908
    0.9250    0.7643
    0.9500    0.8412
    0.9750    0.9200
    1.0000    0.9999
/\relax}\relax
\Maroon{\relax
\plot
         0    0.9807
    0.0250    1.1088
    0.0500    1.2365
    0.0750    1.3535
    0.1000    1.4600
    0.1250    1.5473
    0.1500    1.6098
    0.1750    1.6529
    0.2000    1.6806
    0.2250    1.6879
    0.2500    1.6722
    0.2750    1.6393
    0.3000    1.5940
    0.3250    1.5343
    0.3500    1.4600
    0.3750    1.3769
    0.4000    1.2891
    0.4250    1.1968
    0.4500    1.1012
    0.4750    1.0066
    0.5000    0.9161
    0.5250    0.8304
    0.5500    0.7510
    0.5750    0.6801
    0.6000    0.6196
    0.6250    0.5698
    0.6500    0.5312
    0.6750    0.5045
    0.7000    0.4904
    0.7250    0.4883
    0.7500    0.4976
    0.7750    0.5182
    0.8000    0.5492
    0.8250    0.5893
    0.8500    0.6368
    0.8750    0.6906
    0.9000    0.7497
    0.9250    0.8120
    0.9500    0.8752
    0.9750    0.9381
    1.0000    0.9999
/\relax}\relax
\Red{\relax
\plot
         0    1.0169
    0.0250    1.1507
    0.0500    1.2850
    0.0750    1.4090
    0.1000    1.5227
    0.1250    1.6168
    0.1500    1.6856
    0.1750    1.7345
    0.2000    1.7675
    0.2250    1.7794
    0.2500    1.7671
    0.2750    1.7366
    0.3000    1.6930
    0.3250    1.6338
    0.3500    1.5589
    0.3750    1.4740
    0.4000    1.3837
    0.4250    1.2879
    0.4500    1.1879
    0.4750    1.0882
    0.5000    0.9923
    0.5250    0.9009
    0.5500    0.8154
    0.5750    0.7386
    0.6000    0.6724
    0.6250    0.6172
    0.6500    0.5736
    0.6750    0.5426
    0.7000    0.5244
    0.7250    0.5188
    0.7500    0.5249
    0.7750    0.5422
    0.8000    0.5703
    0.8250    0.6075
    0.8500    0.6522
    0.8750    0.7033
    0.9000    0.7596
    0.9250    0.8191
    0.9500    0.8797
    0.9750    0.9402
    1.0000    0.9999
/\relax}\relax
\endpicture
}
\setbox\figuretwob=\vbox{\hsize=\xfiglen
% $a(x) =   1 + 1*exp(-2*x).*sin(3*pi*x);$
% $f(u) = 20*u.*exp(-4*u);
% Pointwise evaluation in fiti_titr
\beginpicture
%\eightrm
%\small
  \setcoordinatesystem units <0.7\xfiglen,0.48\yfiglen>  point at 0.0 -0.5
  \setplotarea x from 0 to 1.4, y from -0.5 to 2
\scriptsize
  \axis bottom shiftedto y=0 ticks short numbered from 0.25 to 1.25 by 0.25 /
  \axis left ticks short numbered from -0.5 to 2 by 0.5 /
\put {\copy\figurelegendfour} [rt] at 1.2 2
\footnotesize
\setquadratic
\put {$f(u)$} [l] at 0.02 2
\setdashes <3pt>
\Black{
\plot    % actual $f(u)$
         0         0
    0.0203    0.3736
    0.0405    0.6890
    0.0608    0.9531
    0.0810    1.1719
    0.1013    1.3509
    0.1215    1.4949
    0.1418    1.6083
    0.1621    1.6950
    0.1823    1.7585
    0.2026    1.8018
    0.2228    1.8277
    0.2431    1.8387
    0.2633    1.8369
    0.2836    1.8242
    0.3038    1.8024
    0.3241    1.7729
    0.3444    1.7371
    0.3646    1.6961
    0.3849    1.6510
    0.4051    1.6027
    0.4254    1.5518
    0.4456    1.4992
    0.4659    1.4453
    0.4862    1.3908
    0.5064    1.3360
    0.5267    1.2813
    0.5469    1.2270
    0.5672    1.1734
    0.5874    1.1207
    0.6077    1.0692
    0.6280    1.0188
    0.6482    0.9698
    0.6685    0.9223
    0.6887    0.8763
    0.7090    0.8318
    0.7292    0.7890
    0.7495    0.7478
    0.7697    0.7083
    0.7900    0.6703
    0.8103    0.6340
    0.8305    0.5993
    0.8508    0.5661
    0.8710    0.5345
    0.8913    0.5043
    0.9115    0.4757
    0.9318    0.4484
    0.9521    0.4225
    0.9723    0.3979
    0.9926    0.3746
    1.0128    0.3525
    1.0331    0.3315
    1.0533    0.3117
    1.0736    0.2930
    1.0939    0.2753
    1.1141    0.2586
    1.1344    0.2428
    1.1546    0.2279
    1.1749    0.2138
    1.1951    0.2006
    1.2154    0.1881
    1.2357    0.1764
    1.2559    0.1653
    1.2762    0.1549
    1.2964    0.1451
    1.3167    0.1359
    1.3369    0.1272
    1.3572    0.1191
    1.3774    0.1115
/\relax}\relax
\setsolid
\Blue{\relax   % first iteration of $f$
\plot
         0    0.0404
    0.0203    0.3089
    0.0405    0.5526
    0.0608    0.7545
    0.0810    0.9175
    0.1013    1.0471
    0.1215    1.1481
    0.1418    1.2240
    0.1621    1.2780
    0.1823    1.3125
    0.2026    1.3299
    0.2228    1.3323
    0.2431    1.3216
    0.2633    1.2995
    0.2836    1.2676
    0.3038    1.2272
    0.3241    1.1796
    0.3444    1.1258
    0.3646    1.0669
    0.3849    1.0036
    0.4051    0.9366
    0.4254    0.8668
    0.4456    0.7945
    0.4659    0.7203
    0.4862    0.6445
    0.5064    0.5677
    0.5267    0.4900
    0.5469    0.4117
    0.5672    0.3331
    0.5874    0.2543
    0.6077    0.1756
    0.6280    0.0970
    0.6482    0.0187
    0.6685   -0.0593
    0.6887   -0.1367
    0.7090   -0.2137
    0.7292   -0.2900
    0.7495   -0.3656
    0.7697   -0.4405
    0.7900   -0.5146
    0.8103   -0.5880
%    0.8305   -0.6605
%    0.8508   -0.7321
%    0.8710   -0.8029
%    0.8913   -0.8728
%    0.9115   -0.9417
%    0.9318   -1.0097
%    0.9521   -1.0768
%    0.9723   -1.1430
%    0.9926   -1.2082
%    1.0128   -1.2725
%    1.0331   -1.3359
%    1.0533   -1.3983
%    1.0736   -1.4598
%    1.0939   -1.5204
%    1.1141   -1.5801
%    1.1344   -1.6389
%    1.1546   -1.6968
%    1.1749   -1.7538
%    1.1951   -1.8100
%    1.2154   -1.8654
%    1.2357   -1.9200
%    1.2559   -1.9738
%    1.2762   -2.0268
%    1.2964   -2.0791
%    1.3167   -2.1308
%    1.3369   -2.1819
%    1.3572   -2.2327
%    1.3774   -2.2833
/\relax}\relax
\setsolid
\OliveGreen{\relax   % fourth iteration of $f$
\plot
         0    0.0531
    0.0203    0.4073
    0.0405    0.7322
    0.0608    1.0042
    0.0810    1.2274
    0.1013    1.4094
    0.1215    1.5566
    0.1418    1.6738
    0.1621    1.7650
    0.1823    1.8335
    0.2026    1.8820
    0.2228    1.9134
    0.2431    1.9298
    0.2633    1.9335
    0.2836    1.9262
    0.3038    1.9098
    0.3241    1.8857
    0.3444    1.8553
    0.3646    1.8199
    0.3849    1.7804
    0.4051    1.7378
    0.4254    1.6930
    0.4456    1.6465
    0.4659    1.5990
    0.4862    1.5511
    0.5064    1.5030
    0.5267    1.4553
    0.5469    1.4081
    0.5672    1.3619
    0.5874    1.3167
    0.6077    1.2727
    0.6280    1.2302
    0.6482    1.1892
    0.6685    1.1499
    0.6887    1.1122
    0.7090    1.0764
    0.7292    1.0425
    0.7495    1.0105
    0.7697    0.9804
    0.7900    0.9523
    0.8103    0.9261
    0.8305    0.9019
    0.8508    0.8796
    0.8710    0.8593
    0.8913    0.8408
    0.9115    0.8241
    0.9318    0.8093
    0.9521    0.7964
    0.9723    0.7855
    0.9926    0.7768
    1.0128    0.7703
    1.0331    0.7663
    1.0533    0.7650
    1.0736    0.7666
    1.0939    0.7709
    1.1141    0.7778
    1.1344    0.7869
    1.1546    0.7976
    1.1749    0.8092
    1.1951    0.8211
    1.2154    0.8330
    1.2357    0.8446
    1.2559    0.8561
    1.2762    0.8675
    1.2964    0.8791
    1.3167    0.8909
    1.3369    0.9031
    1.3572    0.9156
    1.3774    0.9283
/\relax}\relax
\Maroon{\relax
\plot
         0    0.0491
    0.0203    0.3785
    0.0405    0.6842
    0.0608    0.9415
    0.0810    1.1529
    0.1013    1.3251
    0.1215    1.4640
    0.1418    1.5740
    0.1621    1.6588
    0.1823    1.7214
    0.2026    1.7646
    0.2228    1.7907
    0.2431    1.8021
    0.2633    1.8007
    0.2836    1.7884
    0.3038    1.7668
    0.3241    1.7374
    0.3444    1.7016
    0.3646    1.6605
    0.3849    1.6152
    0.4051    1.5667
    0.4254    1.5157
    0.4456    1.4630
    0.4659    1.4092
    0.4862    1.3548
    0.5064    1.3003
    0.5267    1.2461
    0.5469    1.1925
    0.5672    1.1397
    0.5874    1.0881
    0.6077    1.0379
    0.6280    0.9891
    0.6482    0.9419
    0.6685    0.8965
    0.6887    0.8529
    0.7090    0.8113
    0.7292    0.7716
    0.7495    0.7340
    0.7697    0.6985
    0.7900    0.6653
    0.8103    0.6343
    0.8305    0.6056
    0.8508    0.5793
    0.8710    0.5554
    0.8913    0.5339
    0.9115    0.5147
    0.9318    0.4978
    0.9521    0.4830
    0.9723    0.4702
    0.9926    0.4592
    1.0128    0.4498
    1.0331    0.4416
    1.0533    0.4344
    1.0736    0.4281
    1.0939    0.4223
    1.1141    0.4169
    1.1344    0.4118
    1.1546    0.4070
    1.1749    0.4024
    1.1951    0.3982
    1.2154    0.3942
    1.2357    0.3904
    1.2559    0.3870
    1.2762    0.3837
    1.2964    0.3806
    1.3167    0.3776
    1.3369    0.3748
    1.3572    0.3721
    1.3774    0.3695
/\relax}\relax
\Red{\relax
\plot
         0    0.0489
    0.0203    0.3779
    0.0405    0.6845
    0.0608    0.9432
    0.0810    1.1562
    0.1013    1.3301
    0.1215    1.4708
    0.1418    1.5827
    0.1621    1.6693
    0.1823    1.7338
    0.2026    1.7787
    0.2228    1.8065
    0.2431    1.8194
    0.2633    1.8193
    0.2836    1.8082
    0.3038    1.7876
    0.3241    1.7590
    0.3444    1.7238
    0.3646    1.6831
    0.3849    1.6380
    0.4051    1.5893
    0.4254    1.5381
    0.4456    1.4848
    0.4659    1.4303
    0.4862    1.3750
    0.5064    1.3193
    0.5267    1.2637
    0.5469    1.2085
    0.5672    1.1540
    0.5874    1.1004
    0.6077    1.0479
    0.6280    0.9966
    0.6482    0.9466
    0.6685    0.8982
    0.6887    0.8512
    0.7090    0.8059
    0.7292    0.7621
    0.7495    0.7200
    0.7697    0.6795
    0.7900    0.6407
    0.8103    0.6035
    0.8305    0.5679
    0.8508    0.5339
    0.8710    0.5015
    0.8913    0.4706
    0.9115    0.4411
    0.9318    0.4131
    0.9521    0.3865
    0.9723    0.3613
    0.9926    0.3374
    1.0128    0.3147
    1.0331    0.2932
    1.0533    0.2728
    1.0736    0.2536
    1.0939    0.2354
    1.1141    0.2182
    1.1344    0.2020
    1.1546    0.1867
    1.1749    0.1723
    1.1951    0.1587
    1.2154    0.1459
    1.2357    0.1339
    1.2559    0.1226
    1.2762    0.1120
    1.2964    0.1021
    1.3167    0.0927
    1.3369    0.0839
    1.3572    0.0753
    1.3774    0.0670
/\relax}\relax
\endpicture
}
\setbox\figurethree=\vbox{\hsize=\xfiglen
% $a(x) =   1 + 1*exp(-2*x).*sin(3*pi*x);$
% $f(u) = 20*u.*exp(-4*u);
% Pointwise evaluation in fiti_titr
\beginpicture
%\eightrm
\small
  \setcoordinatesystem units <\xfiglen,0.75\yfiglen>  point at 0.0 0.4
\scriptsize
  \setplotarea x from 0 to 1, y from 0.4 to 2
  \axis bottom shiftedto y=0.4 ticks short numbered from 0 to 1 by 0.2 /
  \axis left ticks short numbered from 0.5 to 2 by 0.5 /
\put {\copy\figurelegendfive} [rt] at 0.97 2
\footnotesize
\setquadratic
\put {$a(x)$} [l] at 0.02 2
\setdashes <3pt>
\Black{
\plot    % actual $a(x)$
         0    1.0000
    0.0250    1.1526
    0.0500    1.2939
    0.0750    1.4212
    0.1000    1.5319
    0.1250    1.6240
    0.1500    1.6963
    0.1750    1.7480
    0.2000    1.7787
    0.2250    1.7887
    0.2500    1.7788
    0.2750    1.7502
    0.3000    1.7046
    0.3250    1.6438
    0.3500    1.5701
    0.3750    1.4860
    0.4000    1.3940
    0.4250    1.2968
    0.4500    1.1970
    0.4750    1.0973
    0.5000    1.0000
    0.5250    0.9075
    0.5500    0.8217
    0.5750    0.7445
    0.6000    0.6774
    0.6250    0.6215
    0.6500    0.5777
    0.6750    0.5463
    0.7000    0.5277
    0.7250    0.5216
    0.7500    0.5276
    0.7750    0.5450
    0.8000    0.5727
    0.8250    0.6095
    0.8500    0.6542
    0.8750    0.7052
    0.9000    0.7610
    0.9250    0.8200
    0.9500    0.8805
    0.9750    0.9410
    1.0000    1.0000
/\relax}\relax
\setsolid
\Blue{\relax   % first iteration of $a(x)
\plot
         0    0.5014
    0.0250    0.6294
    0.0500    0.7795
    0.0750    0.9358
    0.1000    1.0726
    0.1250    1.1685
    0.1500    1.2197
    0.1750    1.2381
    0.2000    1.2393
    0.2250    1.2338
    0.2500    1.2252
    0.2750    1.2129
    0.3000    1.1961
    0.3250    1.1753
    0.3500    1.1501
    0.3750    1.1173
    0.4000    1.0722
    0.4250    1.0134
    0.4500    0.9442
    0.4750    0.8707
    0.5000    0.7986
    0.5250    0.7312
    0.5500    0.6704
    0.5750    0.6176
    0.6000    0.5743
    0.6250    0.5413
    0.6500    0.5185
    0.6750    0.5051
    0.7000    0.4997
    0.7250    0.5008
    0.7500    0.5072
    0.7750    0.5187
    0.8000    0.5360
    0.8250    0.5608
    0.8500    0.5945
    0.8750    0.6381
    0.9000    0.6914
    0.9250    0.7526
    0.9500    0.8206
    0.9750    0.8989
    1.0000    1.0000
/\relax}\relax
\setsolid
\Red{\relax   % fourth iteration of $a$
\plot
         0    0.7861
    0.0250    0.9870
    0.0500    1.2145
    0.0750    1.4433
    0.1000    1.6344
    0.1250    1.7568
    0.1500    1.8072
    0.1750    1.8055
    0.2000    1.7765
    0.2250    1.7364
    0.2500    1.6910
    0.2750    1.6404
    0.3000    1.5843
    0.3250    1.5242
    0.3500    1.4603
    0.3750    1.3888
    0.4000    1.3050
    0.4250    1.2076
    0.4500    1.1014
    0.4750    0.9941
    0.5000    0.8926
    0.5250    0.8005
    0.5500    0.7199
    0.5750    0.6518
    0.6000    0.5968
    0.6250    0.5546
    0.6500    0.5240
    0.6750    0.5036
    0.7000    0.4920
    0.7250    0.4878
    0.7500    0.4903
    0.7750    0.4992
    0.8000    0.5148
    0.8250    0.5385
    0.8500    0.5716
    0.8750    0.6151
    0.9000    0.6690
    0.9250    0.7321
    0.9500    0.8040
    0.9750    0.8887
    1.0000    1.0000
/\relax}\relax
\endpicture
}
\setbox\figurethreeb=\vbox{\hsize=\xfiglen
% $a(x) =   1 + 1*exp(-2*x).*sin(3*pi*x);$
% $f(u) = 20*u.*exp(-4*u);
% Pointwise evaluation in fiti_titr from sparse data at 1% noise
\beginpicture
%\eightrm
%\small
  \setcoordinatesystem units <0.7\xfiglen,0.48\yfiglen>  point at 0.0 -0.5
\scriptsize
  \setplotarea x from 0 to 1.4, y from -0.5 to 2
  \axis bottom shiftedto y=0 ticks short numbered from 0.25 to 1.25 by 0.25 /
  \axis left ticks short numbered from -0.5 to 2 by 0.5 /
\footnotesize
\put {\copy\figurelegendsix} [rt] at 1.2 2
\setquadratic
\put {$f(u)$} [l] at 0.02 2
\setdashes <3pt>
\Black{
\plot    % actual $f(u)$
    0.0000    0.0000
    0.0202    0.3726
    0.0404    0.6874
    0.0606    0.9510
    0.0808    1.1696
    0.1010    1.3486
    0.1212    1.4927
    0.1414    1.6063
    0.1616    1.6933
    0.1818    1.7571
    0.2020    1.8008
    0.2222    1.8271
    0.2424    1.8385
    0.2626    1.8371
    0.2828    1.8249
    0.3030    1.8035
    0.3232    1.7744
    0.3434    1.7390
    0.3636    1.6984
    0.3838    1.6536
    0.4040    1.6055
    0.4242    1.5549
    0.4444    1.5026
    0.4646    1.4489
    0.4848    1.3946
    0.5050    1.3399
    0.5252    1.2854
    0.5454    1.2312
    0.5656    1.1777
    0.5858    1.1251
    0.6060    1.0736
    0.6261    1.0232
    0.6463    0.9743
    0.6665    0.9267
    0.6867    0.8807
    0.7069    0.8362
    0.7271    0.7934
    0.7473    0.7521
    0.7675    0.7125
    0.7877    0.6745
    0.8079    0.6381
    0.8281    0.6033
    0.8483    0.5700
    0.8685    0.5383
    0.8887    0.5081
    0.9089    0.4793
    0.9291    0.4519
    0.9493    0.4259
    0.9695    0.4012
    0.9897    0.3778
    1.0099    0.3556
    1.0301    0.3345
    1.0503    0.3146
    1.0705    0.2958
    1.0907    0.2780
    1.1109    0.2611
    1.1311    0.2452
    1.1513    0.2302
    1.1715    0.2161
    1.1917    0.2028
    1.2119    0.1902
    1.2321    0.1784
    1.2523    0.1672
    1.2725    0.1567
    1.2927    0.1469
    1.3129    0.1376
    1.3331    0.1288
    1.3533    0.1206
    1.3735    0.1129
    1.3937    0.1057
    1.4139    0.0989
    1.4341    0.0925
    1.4543    0.0866
    1.4745    0.0810
    1.4947    0.0757
    1.5149    0.0708
%    1.5351    0.0661
    1.5553    0.0618
/\relax}\relax
\setsolid
\Blue{\relax   % first iteration of $f$
\plot
   0.0000    0.0000
    0.0202    0.2331
    0.0404    0.4612
    0.0606    0.6797
    0.0808    0.8841
    0.1010    1.0705
    0.1212    1.2356
    0.1414    1.3769
    0.1616    1.4931
    0.1818    1.5835
    0.2020    1.6487
    0.2222    1.6903
    0.2424    1.7107
    0.2626    1.7132
    0.2828    1.7014
    0.3030    1.6791
    0.3232    1.6502
    0.3434    1.6180
    0.3636    1.5850
    0.3838    1.5531
    0.4040    1.5230
    0.4242    1.4942
    0.4444    1.4658
    0.4646    1.4356
    0.4848    1.4015
    0.5050    1.3611
    0.5252    1.3124
    0.5454    1.2540
    0.5656    1.1854
    0.5858    1.1071
    0.6060    1.0206
    0.6261    0.9284
    0.6463    0.8335
    0.6665    0.7392
    0.6867    0.6487
    0.7069    0.5649
    0.7271    0.4897
    0.7473    0.4238
    0.7675    0.3670
    0.7877    0.3176
    0.8079    0.2731
    0.8281    0.2306
    0.8483    0.1866
    0.8685    0.1382
    0.8887    0.0831
    0.9089    0.0201
    0.9291   -0.0507
    0.9493   -0.1278
    0.9695   -0.2083
    0.9897   -0.2885
    1.0099   -0.3642
    1.0301   -0.4313
    1.0503   -0.4863
    1.0705   -0.5271
    1.0907   -0.5531
    1.1109   -0.5656
    1.1311   -0.5679
    1.1513   -0.5648
    1.1715   -0.5623
    1.1917   -0.5669
    1.2119   -0.5845
 %   1.2321   -0.6201
 %   1.2523   -0.6765
 %   1.2725   -0.7542
 %   1.2927   -0.8504
 %   1.3129   -0.9596
 %   1.3331   -1.0734
 %   1.3533   -1.1814
 %   1.3735   -1.2719
 %   1.3937   -1.3330
 %   1.4139   -1.3539
 %   1.4341   -1.3259
 %   1.4543   -1.2433
 %   1.4745   -1.1044
 %   1.4947   -0.9116
 %   1.5149   -0.6715
 %   1.5351   -0.3946
 %   1.5553   -0.0943
/\relax}\relax
\Red{\relax   % fourth iteration of $f$
\plot
    0.0000    0.0000
    0.0202    0.2649
    0.0404    0.5213
    0.0606    0.7613
    0.0808    0.9784
    0.1010    1.1679
    0.1212    1.3273
    0.1414    1.4561
    0.1616    1.5558
    0.1818    1.6298
    0.2020    1.6824
    0.2222    1.7186
    0.2424    1.7433
    0.2626    1.7608
    0.2828    1.7744
    0.3030    1.7864
    0.3232    1.7972
    0.3434    1.8064
    0.3636    1.8126
    0.3838    1.8135
    0.4040    1.8070
    0.4242    1.7909
    0.4444    1.7641
    0.4646    1.7260
    0.4848    1.6772
    0.5050    1.6193
    0.5252    1.5546
    0.5454    1.4860
    0.5656    1.4164
    0.5858    1.3488
    0.6060    1.2851
    0.6261    1.2268
    0.6463    1.1744
    0.6665    1.1273
    0.6867    1.0842
    0.7069    1.0435
    0.7271    1.0030
    0.7473    0.9609
    0.7675    0.9161
    0.7877    0.8680
    0.8079    0.8169
    0.8281    0.7640
    0.8483    0.7113
    0.8685    0.6614
    0.8887    0.6168
    0.9089    0.5798
    0.9291    0.5523
    0.9493    0.5352
    0.9695    0.5282
    0.9897    0.5301
    1.0099    0.5387
    1.0301    0.5512
    1.0503    0.5641
    1.0705    0.5744
    1.0907    0.5795
    1.1109    0.5776
    1.1311    0.5681
    1.1513    0.5517
    1.1715    0.5303
    1.1917    0.5066
    1.2119    0.4842
    1.2321    0.4664
    1.2523    0.4562
    1.2725    0.4558
    1.2927    0.4660
    1.3129    0.4859
    1.3331    0.5133
    1.3533    0.5441
    1.3735    0.5736
    1.3937    0.5961
    1.4139    0.6062
    1.4341    0.5990
    1.4543    0.5708
    1.4745    0.5199
    1.4947    0.4464
    1.5149    0.3524
%    1.5351    0.2423
    1.5553    0.1219
/\relax}\relax
\endpicture
}
\setbox\figurefive=\vbox{\hsize=\xfiglen
% $a(x) =   1 + 1*exp(-2*x).*sin(2*pi*x);$
% $f(u) = beta*u.*exp(-4*u);$
% Pointwise evaluation in fiti_titr
\beginpicture
%\eightrm
%\small
  \setcoordinatesystem units <\xfiglen,0.8\yfiglen>  point at 0.0 0.5
  \setplotarea x from 0 to 1, y from 0.5 to 2
\scriptsize
  \axis bottom shiftedto y=0.5 ticks short numbered from 0 to 1 by 0.2 /
  \axis left ticks short numbered from 0.5 to 2 by 0.5 /
\put {\copy\figurelegendseven} [rt] at 0.95 1.98
\footnotesize
\put {$a(x)$} [lt] at 0.02 2
\put {$x$} [lb] at 1 0.535
\setsolid
\Brown{\relax   % beta = 20
\plot
         0    0.8421
    0.0250    0.9740
    0.0500    1.1042
    0.0750    1.2282
    0.1000    1.3422
    0.1250    1.4430
    0.1500    1.5278
    0.1750    1.5944
    0.2000    1.6415
    0.2250    1.6683
    0.2500    1.6744
    0.2750    1.6606
    0.3000    1.6279
    0.3250    1.5782
    0.3500    1.5136
    0.3750    1.4368
    0.4000    1.3509
    0.4250    1.2588
    0.4500    1.1638
    0.4750    1.0685
    0.5000    0.9757
    0.5250    0.8876
    0.5500    0.8062
    0.5750    0.7330
    0.6000    0.6691
    0.6250    0.6155
    0.6500    0.5730
    0.6750    0.5422
    0.7000    0.5236
    0.7250    0.5173
    0.7500    0.5230
    0.7750    0.5399
    0.8000    0.5673
    0.8250    0.6039
    0.8500    0.6483
    0.8750    0.6993
    0.9000    0.7553
    0.9250    0.8147
    0.9500    0.8762
    0.9750    0.9384
    1.0000    1.0000
 /\relax}\relax
\Maroon{\relax   % beta = 10
\plot
         0    0.9444
    0.0250    1.0918
    0.0500    1.2340
    0.0750    1.3655
    0.1000    1.4822
    0.1250    1.5809
    0.1500    1.6594
    0.1750    1.7163
    0.2000    1.7512
    0.2250    1.7643
    0.2500    1.7565
    0.2750    1.7295
    0.3000    1.6849
    0.3250    1.6250
    0.3500    1.5522
    0.3750    1.4691
    0.4000    1.3784
    0.4250    1.2826
    0.4500    1.1844
    0.4750    1.0861
    0.5000    0.9902
    0.5250    0.8988
    0.5500    0.8140
    0.5750    0.7375
    0.6000    0.6709
    0.6250    0.6154
    0.6500    0.5719
    0.6750    0.5409
    0.7000    0.5225
    0.7250    0.5166
    0.7500    0.5228
    0.7750    0.5402
    0.8000    0.5680
    0.8250    0.6049
    0.8500    0.6497
    0.8750    0.7009
    0.9000    0.7571
    0.9250    0.8165
    0.9500    0.8778
    0.9750    0.9394
    1.0000    1.0000
 /\relax}\relax
\Red{\plot  % beta = 5
         0    0.9973
    0.0250    1.1501
    0.0500    1.2927
    0.0750    1.4215
    0.1000    1.5336
    0.1250    1.6267
    0.1500    1.6993
    0.1750    1.7505
    0.2000    1.7804
    0.2250    1.7893
    0.2500    1.7782
    0.2750    1.7484
    0.3000    1.7016
    0.3250    1.6400
    0.3500    1.5656
    0.3750    1.4811
    0.4000    1.3890
    0.4250    1.2918
    0.4500    1.1921
    0.4750    1.0926
    0.5000    0.9956
    0.5250    0.9034
    0.5500    0.8181
    0.5750    0.7412
    0.6000    0.6745
    0.6250    0.6189
    0.6500    0.5753
    0.6750    0.5442
    0.7000    0.5258
    0.7250    0.5199
    0.7500    0.5261
    0.7750    0.5435
    0.8000    0.5713
    0.8250    0.6082
    0.8500    0.6530
    0.8750    0.7041
    0.9000    0.7600
    0.9250    0.8191
    0.9500    0.8798
    0.9750    0.9406
    1.0000    1.0000
 /\relax}\relax
\setdashes <3pt>
\Black{\relax
\plot    % actual $a(x)
         0    1.0000
    0.0250    1.1526
    0.0500    1.2939
    0.0750    1.4212
    0.1000    1.5319
    0.1250    1.6240
    0.1500    1.6963
    0.1750    1.7480
    0.2000    1.7787
    0.2250    1.7887
    0.2500    1.7788
    0.2750    1.7502
    0.3000    1.7046
    0.3250    1.6438
    0.3500    1.5701
    0.3750    1.4860
    0.4000    1.3940
    0.4250    1.2968
    0.4500    1.1970
    0.4750    1.0973
    0.5000    1.0000
    0.5250    0.9075
    0.5500    0.8217
    0.5750    0.7445
    0.6000    0.6774
    0.6250    0.6215
    0.6500    0.5777
    0.6750    0.5463
    0.7000    0.5277
    0.7250    0.5216
    0.7500    0.5276
    0.7750    0.5450
    0.8000    0.5727
    0.8250    0.6095
    0.8500    0.6542
    0.8750    0.7052
    0.9000    0.7610
    0.9250    0.8200
    0.9500    0.8805
    0.9750    0.9410
    1.0000    1.0000
 /\relax}\relax
\endpicture
}
\setbox\figuresix=\vbox{\hsize=\xfiglen
\beginpicture
%\small
  \setcoordinatesystem units <0.57\xfiglen,0.61\yfiglen>  point at 0 0
  \setplotarea x from 0 to 1.62, y from 0 to 2
\scriptsize
  \axis bottom shiftedto y=0 ticks short numbered from 0 to 1.5 by 0.5 /
  \axis left ticks short numbered from 0 to 2 by 0.5 /
\put {\copy\figurelegendeight} [rt] at 1.6 2.0
\footnotesize
\put {$f(u)$} [l] at 0.02 2
\put {$u$} [lt] at 1.61 -0.03
\setdashes <3pt>
\Black{\relax % actual $f$
\plot
    0.0000    0.0000
    0.0405    0.6887
    0.0810    1.1715
    0.1215    1.4945
    0.1620    1.6947
    0.2025    1.8016
    0.2430    1.8387
    0.2834    1.8243
    0.3239    1.7732
    0.3644    1.6965
    0.4049    1.6032
    0.4454    1.4998
    0.4859    1.3915
    0.5264    1.2820
    0.5669    1.1742
    0.6074    1.0699
    0.6479    0.9706
    0.6884    0.8771
    0.7289    0.7898
    0.7694    0.7090
    0.8098    0.6347
    0.8503    0.5668
    0.8908    0.5050
    0.9313    0.4490
    0.9718    0.3985
    1.0123    0.3530
    1.0528    0.3122
    1.0933    0.2758
    1.1338    0.2432
    1.1743    0.2142
    1.2148    0.1885
    1.2553    0.1656
    1.2958    0.1454
    1.3362    0.1275
    1.3767    0.1117
    1.4172    0.0978
    1.4577    0.0856
    1.4982    0.0748
    1.5387    0.0653
    1.5792    0.0570
    1.6197    0.0497
 /\relax}\relax
\setsolid
\Blue{\relax %beta =1
\plot
    0.0000    0.0004
    0.0404    0.6068
    0.0808    1.0452
    0.1212    1.3480
    0.1616    1.5575
    0.2020    1.6984
    0.2424    1.7823
    0.2828    1.8154
    0.3233    1.8031
    0.3637    1.7521
    0.4041    1.6708
    0.4445    1.5680
    0.4849    1.4524
    0.5253    1.3315
    0.5657    1.2112
    0.6061    1.0954
    0.6465    0.9868
    0.6869    0.8863
    0.7273    0.7942
    0.7677    0.7101
    0.8081    0.6333
    0.8485    0.5634
    0.8889    0.4999
    0.9293    0.4424
    0.9698    0.3905
    1.0102    0.3441
    1.0506    0.3027
    1.0910    0.2661
    1.1314    0.2338
    1.1718    0.2056
    1.2122    0.1809
    1.2526    0.1595
    1.2930    0.1409
    1.3334    0.1249
    1.3738    0.1112
    1.4142    0.0995
    1.4546    0.0894
    1.4950    0.0807
    1.5354    0.0728
    1.5758    0.0652
 /\relax}\relax
\OliveGreen{\relax   % beta = 2
\plot
   0.0000    0.0005
    0.0404    0.5381
    0.0807    0.9693
    0.1211    1.2859
    0.1614    1.5094
    0.2018    1.6609
    0.2421    1.7540
    0.2825    1.7964
    0.3229    1.7935
    0.3632    1.7508
    0.4036    1.6756
    0.4439    1.5762
    0.4843    1.4616
    0.5246    1.3400
    0.5650    1.2181
    0.6054    1.1006
    0.6457    0.9904
    0.6861    0.8888
    0.7264    0.7959
    0.7668    0.7112
    0.8071    0.6341
    0.8475    0.5638
    0.8879    0.4999
    0.9282    0.4418
    0.9686    0.3894
    1.0089    0.3425
    1.0493    0.3007
    1.0896    0.2637
    1.1300    0.2312
    1.1704    0.2028
    1.2107    0.1781
    1.2511    0.1566
    1.2914    0.1381
    1.3318    0.1223
    1.3721    0.1087
    1.4125    0.0973
    1.4529    0.0876
    1.4932    0.0793
    1.5336    0.0720
    1.5739    0.0651
 /\relax}\relax
\Red{\relax %      % beta = 5
\plot
    0.0000    0.0005
    0.0403    0.3649
    0.0807    0.7090
    0.1210    1.0093
    0.1613    1.2542
    0.2017    1.4409
    0.2420    1.5716
    0.2823    1.6504
    0.3226    1.6822
    0.3630    1.6722
    0.4033    1.6260
    0.4436    1.5501
    0.4840    1.4520
    0.5243    1.3399
    0.5646    1.2220
    0.6050    1.1054
    0.6453    0.9952
    0.6856    0.8938
    0.7259    0.8020
    0.7663    0.7190
    0.8066    0.6437
    0.8469    0.5749
    0.8873    0.5116
    0.9276    0.4534
    0.9679    0.4000
    1.0083    0.3516
    1.0486    0.3080
    1.0889    0.2691
    1.1292    0.2348
    1.1696    0.2047
    1.2099    0.1785
    1.2502    0.1559
    1.2906    0.1364
    1.3309    0.1198
    1.3712    0.1057
    1.4116    0.0940
    1.4519    0.0842
    1.4922    0.0762
    1.5325    0.0694
    1.5729    0.0634
 /\relax}\relax
\Maroon{\relax   % beta = 10
\plot
    0.0000    0.0003
    0.0405    0.1856
    0.0810    0.3758
    0.1215    0.5659
    0.1620    0.7504
    0.2025    0.9234
    0.2430    1.0789
    0.2834    1.2105
    0.3239    1.3122
    0.3644    1.3785
    0.4049    1.4060
    0.4454    1.3943
    0.4859    1.3476
    0.5264    1.2734
    0.5669    1.1802
    0.6074    1.0765
    0.6479    0.9703
    0.6884    0.8691
    0.7289    0.7782
    0.7694    0.6997
    0.8098    0.6328
    0.8503    0.5747
    0.8908    0.5227
    0.9313    0.4743
    0.9718    0.4282
    1.0123    0.3839
    1.0528    0.3416
    1.0933    0.3020
    1.1338    0.2653
    1.1743    0.2320
    1.2148    0.2020
    1.2553    0.1754
    1.2958    0.1519
    1.3362    0.1316
    1.3767    0.1142
    1.4172    0.0994
    1.4577    0.0871
    1.4982    0.0770
    1.5387    0.0687
    1.5792    0.0617
    1.6197    0.0553
 /\relax}\relax
\Brown{\relax %  beta = 20
\plot
    0.0000   -0.0009
    0.0404    0.0011
    0.0807    0.0067
    0.1211    0.0188
    0.1615    0.0402
    0.2018    0.0727
    0.2422    0.1179
    0.2826    0.1760
    0.3229    0.2469
    0.3633    0.3288
    0.4037    0.4190
    0.4441    0.5133
    0.4844    0.6060
    0.5248    0.6901
    0.5652    0.7573
    0.6055    0.7989
    0.6459    0.8080
    0.6863    0.7825
    0.7266    0.7275
    0.7670    0.6515
    0.8074    0.5638
    0.8477    0.4733
    0.8881    0.3891
    0.9285    0.3202
    0.9688    0.2736
    1.0092    0.2502
    1.0496    0.2467
    1.0899    0.2561
    1.1303    0.2697
    1.1707    0.2798
    1.2110    0.2815
    1.2514    0.2736
    1.2918    0.2573
    1.3322    0.2349
    1.3725    0.2091
    1.4129    0.1823
    1.4533    0.1559
    1.4936    0.1311
    1.5340    0.1086
    1.5744    0.0884
    1.6147    0.0704
    1.6551    0.0538
 /\relax}\relax
\endpicture
}

\begin{figure}[ht]
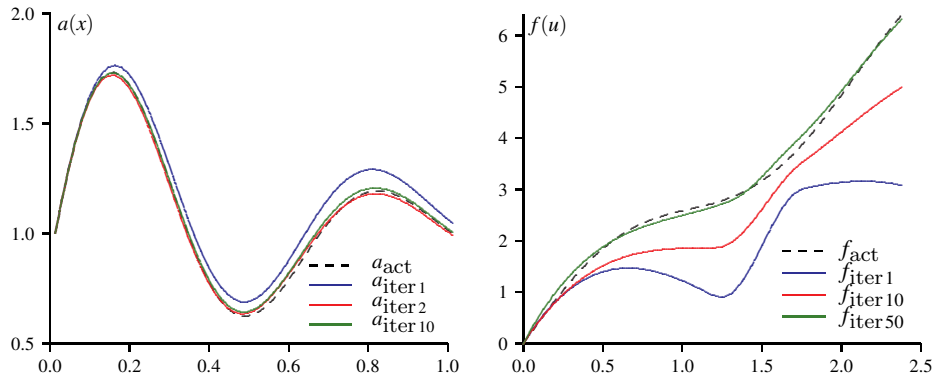

\hbox to \hsize{\hss\copy\figureone\hss\hss\copy\figureoneb\hss}
\caption{\small {\bf Recovery of $a(x)$ and $f(u)$ from final data}}
\label{fig:fiti_fiti}
\end{figure}

The initial approximations were $a^0(x)=1$ and $f^0(u)=0$.
An important point to observe is the speed of convergence of the
scheme updating $a(x)$.  The second iteration was already extremely good
and the difference between the third and tenth iteration shown would
not be distinguishable at the figure resolution.
On the other hand, the iterations for $f$ proceeded much more slowly.

One factor that should be noted here is that we used both data sets to
recover $a$ initially, and then similarly for $f$.
Much the same picture would have resulted from reversing this order.
The fact is that the diffusion coefficient $a(x)$ dominates the
equation in terms of its ability to modify solutions $u(x,t:a;f)$ in
comparison to that for $f(u)$.
This situation was true whether pointwise or basis reconstructions schemes
were used for representing the unknowns.

\begin{remark}\label{remark:less data}
In higher space dimensions it is possible from a reconstruction perspective
to measure less data. For example, in $\mathbb{R}^d$ with $d=2,\,3$ one
data run for $g_u(x) = u(x,T)$ measures this quantity for all $x\in\Omega$, but
the second run measures only $g_v(x) = v(s(x),T)$ for some
one-dimensional curve $C$ parametrized by $s(x)$ that connects points
$x_1,\,x_2$ on $\partial\Omega$.
The strategy is to use the first of these to recover $a(x)$ then since $f(u)$
depends only the single variable $u$ we seek to recover $f(s(x))$ from this
second measurement.
The points $x_1$ and $x_2$ cannot be chosen arbitrarily but must in fact
correspond to the maximum and minimum values of $v(\cdot,T)$ on
$\partial\Omega$ in order to satisfy the range condition necessary
for recovering $f$.
We note here that we have  no theoretical analysis for this case;
neither a uniquness result nor a convergence theorem.
\end{remark}

\subsection{Final time and time-trace data}\label{sect:fitititr}

As noted previously, an alternative data measurement set is to measure both
boundary values as well as final time information, namely
\begin{equation}\label{eqn:fitititr_data}
u(x,T) = g_u(x) \qquad u(\tilde x,t) = h_u(t)
\end{equation}
where we have assumed that $\tilde x \in\partial\Omega$ is such that that
the maximum range of $u(x,t)$ occurs at that point.
It will also be convenient to take the impedance value to be zero
at $\tilde x$ and arrange that $g_u(x)$ is monotonic in $\Omega$.

\begin{remark}\label{remark:energy}
We have chosen a particular choice of a time trace measurements.
Other possibilities include a time trace of $u(x^\star,t)$ for
some interior point $x^\star\in\Omega$.
The difficulty here is in ensuring that the range condition is satisfied
and this is likely to be prohibited by the maximum principle.
An alternative is a measurement of the energy
$E(t) = \int_\Omega |\nabla u|^2\,dx$ or some similar functional.
Again, a determining factor would be the range condition.
\end{remark}

Again, starting from the fixed point iteration defined by
\eqref{eqn:fp_fiti-titr}, we derive a numerical reconstruction algorithm.
There are several ways to proceed and we start with pointwise updates
of both $a(x)$ and $f(u)$ in sequence at each full iteration step.

We again let $f^0$, $a^0$ be some initial approximation, then for
$k=0,1,2,\ldots$
\begin{itemize}
\item{}
Compute $u(x,t;a^k,f^k)$.
Update $a(x)$ using the final time information $g_u(x)$ by setting
\begin{equation}\label{eqn:a_successive-update_2}
\begin{aligned}
\phi(x) &= u_t(x,T;a^k,f^k) - r_u(x,T,gu) - f^k\bigl(g_u(x)\bigr)
\qquad \Phi(x) = \int_0^x \phi(s)\,ds \\
a^{k+1}(x) &= \Phi(x)/g_u'(x) \\
\end{aligned}
\end{equation}
\item{}
Now update $f(u)$ by projecting onto the boundary at $x=\tilde x$ using the
data $h_u(t)$ and computing $\triangle u(\tilde x,t) = u_{xx}(\tilde x,t)$
so that
\begin{equation}
\psi(t) = h_u(t) - r_u\bigl(\tilde x,t,h_u(t)\bigr)
- a^{k+1}(\tilde x) u_{xx}(\tilde x,t).
\end{equation}
Here the advantage of the imposed boundary condition $u_x(\tilde x,t) = 0$
is evident as there is no need to involve the derivative of $a$.
\item{}
Now if $h_u(t)$ is monotone then $f = f^{k+1}(u)$ is recovered from
\begin{equation}
f^{k+1}\bigl(h_u(t)\bigr) = \psi(t).
\end{equation}
\end{itemize}

In Figure~\oldref{fig:fiti_titr_no_noise} we show a pair of reconstructions of
$a(x)$ and $f(u)$ using the above scheme.
The spatial interval was taken as $0,L$ with $L=1$ and the final time
of measurement was $T=0.5$.
Thus data consisted of measurements of  $g(x_i) = u(x_i,T)$ and
$h(t_j)=u(L,t_j)$ where the numbers of sampled points were
$N_x=20$ and $N_t=25$.

In this reconstruction the initial values were once again
taken as $a^0 = 1$ and $f^0=0$.
In most applications a better initial approximation would likely be available.

As can be seen from this figure, the convergence of the scheme was rather
slow initially  and in fact the first approximation to the diffusion
coefficient $a(x)$ and the reaction $f(u)$ were both uniformly lower than
the actual values.
One might expect that a uniform too small value of $a$ would result
in an $f$ that was larger then the actual and this did occur on the second
iteration for $f$.  The reason it did not occur at the first iterative step
is at each step the updates for $a$ and $f$ are independent:
the former only using $g(x)$ and the latter only using $h(t)$.

This in fact indicates an almost limiting distance apart of
the initial approximations and the actual values of $a$ and $f$.
If an iterate of $a$ becomes too close to a zero value the scheme becomes
very unstable.
Also, in the presence of noise, one would in general have to take
a much closer initial approximation.

The reason for the final small amount of mis-fit of the reconstructed
and the actual functions is the necessity to use a smoothing scheme on the
data to restore the correct mapping properties.

\begin{figure}[ht]
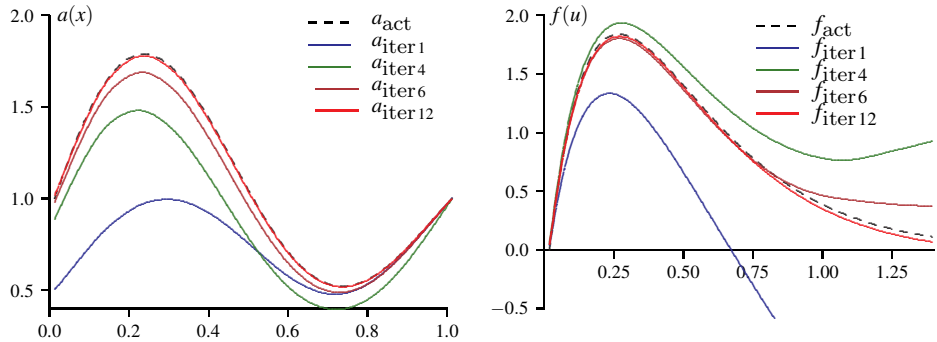

\hbox to \hsize{\hss\copy\figuretwo\hss\hss\copy\figuretwob\hss}
\caption{\small {\bf Recovery of $\,a(x)\,$ and $\,f(u)\,$ from space-and-time data}}
\label{fig:fiti_titr_no_noise}
\end{figure}

\vspace{.2in}
%\begin{figure}[ht]
%\hbox to %\hsize{\hss\copy\figurethree\hss\hss\copy\figurethreeb\hss}
%\caption{\small {\bf Recovery of $\,a(x),$ and $\,f(u)\,$ from %space-and-time data: 1\% noise}}
%\label{fig:fiti_titr_percent_noise}
%\end{figure}

\begin{figure}[ht]
\hbox to \hsize{\hss\copy\figurethree\hss}
%\caption{\small {\bf Recovery of $\,a(x),$ and $\,f(u)\,$ from space-and-time data: 1\% noise}}
\label{fig:fiti_titr_percent_noise}
\end{figure}

\begin{figure}[ht]
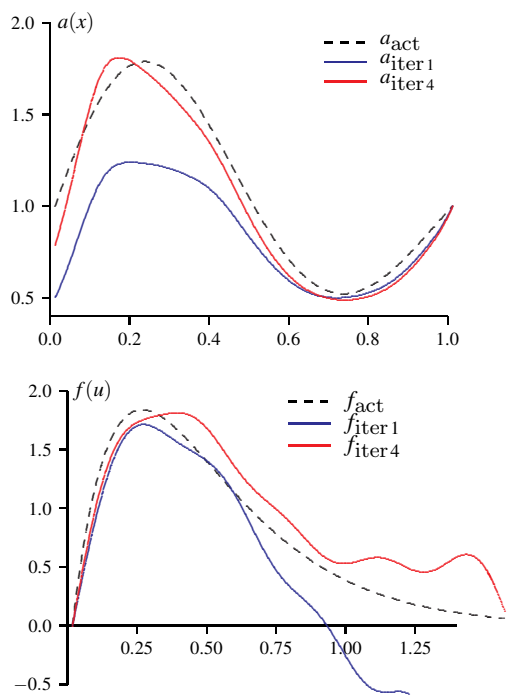

\hbox to \hsize{\hss\copy\figurethreeb\hss}
\caption{\small {\bf Recovery of $\,a(x),$ and $\,f(u)\,$ from space-and-time data: 1\% noise}}
\label{fig:fiti_titr_percent_noise}
\end{figure}

Notice here the considerable degradation of the reconstructions under
even modest amounts of noise.
This is not an artifact of the method but inherent in the problem.
Reconstructions using a pair of final values as in the last section
are more forgiving under data noise but still show the limitations
of reconstructions in these inverse problems.

\subsection{Non-trivial initial values}\label{sec:u0zero}
In section~\oldref{sec:conv} we required that the initial value $u_0(x) = u(x,0) = 0$.
This was to remove a possible singularity arising from the convolution
in equation~\eqref{eqn:checku0}.
This issue has been noted in previous work attempting to recover
a reaction term $f(u)$ from time-valued data on the boundary $\partial\Omega$,
\cite{PilantRundell:1986,PilantRundell:1987,KaltenbacherRundell:2020b}
even when this term is the only unknown in the model.
The question thus arises if this is a fundamental requirement or
merely a delicate technical issue to be surmounted.

Suppose that $u$ satisfies
$$
u_t - (au_x)_x = f(u) + r(x,t),\qquad u(x,0) = u_0(x),\quad
{\mathbb B} u = \phi
$$
where ${\mathbb B}$ is the boundary operator on $\partial\Omega$,
$\Omega=(0,1)$.
For continuity we should  impose the condition ${\mathbb B} u_0 = \phi$
at $t=0$ but, in particular we suppose $u_0$ is positive and has support
contained in the strict interior of $(0,1)$ to remove all boundary effects.

Now let $v(x,t) = u(x,t) - u_0(x)$ so that $v$ satisfies
$$
v_t - (av_x)_x = \tilde f(v) + s(x,t),\qquad v(x,0) = 0 ,\quad {\mathbb B} v = 0
$$
where $\tilde f(v) = f(u) = f(v+u_0)$ and $s(x,t) = r(x,t) - (a u_0')'$.

The overposed data $u(1,t)=h(t)$ is converted into $v=\tilde h(t)$ but
relating $\tilde h$ to $h$ follows from the fact that $u_0$ vanishes
at $x=0$ and $x=1$ gives that  $\tilde h(t) = h(t)$.

Thus in principle, the problem with nontrivial $u_0$ can be converted into
one with $u_0=0$ by the above transformation:
we get a new, known, right hand side driving term
(assuming $a(x)$ is known) and reconstruct as before to get
$\tilde f$.  Then from this, knowing $u_0$, recover $f$.
Does this then amount to resolution of the issue?
Not quite.  The previous convolution situation is not resolved by this
as the argument of $f$ in $\tilde f$ is now shifted.

However, the transformation does point to a definite issue for nontrivial $u_0$.
We have to ensure that the standard range condition is satisfied:
$h(t)$ must encompass all values that the domain of $f$ requires.
The given data $h(t)$ is the same for both $u$ and $v$ but
$\tilde f$ requires a different range than $f$ (since $u_0>0$)
and thus appears to have missing information.

In Figure~\oldref{fig:fiti_titr_u0} below we show final reconstructions
(taken as the $10^{\hbox{\sevenrm th}}$ iteration which certainly corresponded
to effective numerical convergence) for both $a(x)$ and $f(u)$.
This was under no data-noise conditions and sampling at a large number of
points.
A nontrivial value of $u_0$ was taken; namely $u_0(x) = \beta\,x^2(1-x^2)$.
for the values $\,\beta = 1,\,2,\,5,\,10,\,20$.
This function and its first derivative is zero at both endpoints
so ${\mathbb B}u_0 = 0$. The value $\beta=20$ corresponds to a maximum value of
$1.25$ for $u_0$; $\beta=1$ gives a height of $0.0625$.
The rightmost graphic shows the reconstructed $f(u)$ for these choices.
In the case $\beta =0$ the reconstruction would have been indistinguishable
from the actual $f$ as the data was noise-free and sampled at a large number
of points.
Notice how rapidly the reconstructed $f$ deteriorates  from the actual
with increasing $\beta$ and hence size of $u_0$, as the
range condition violation,
which predominantly affected the smaller values of $u$, became stronger.
For sufficiently large $u_0$ this finally affects the entire reconstruction.

\begin{figure}[ht]
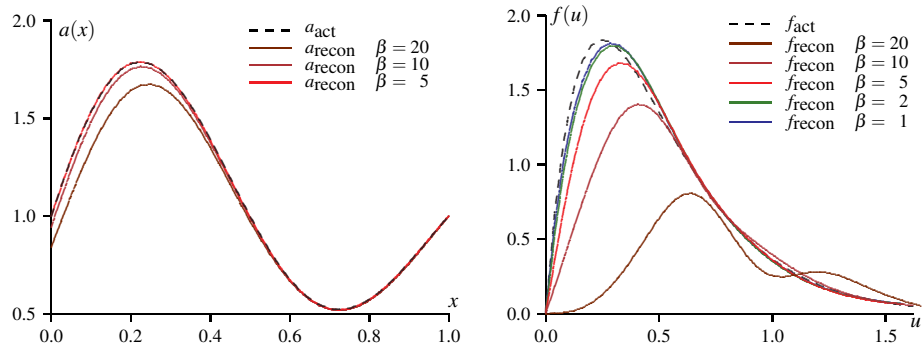

\hbox to \hsize{\hss\copy\figurefive\hss\hss\copy\figuresix\hss}
\caption{\small {\bf Recovery of $\,a(x)\,$ and $\,f(u)\,$ from space-and-time data when $u_0\not=0$}}
\label{fig:fiti_titr_u0}
\end{figure}

The coefficient $a(x)$ which is reconstructed solely from the final time
values $g(x)=u(x,T)$ is relatively immune to this effect -
until the reconstructed $f$ becomes sufficiently poor that this has an
effect on subsequent iterations of the combined scheme.
Once again, the smaller values of the data $g(x)$ used are most sensitive
to a nontrivial $u_0$ and since our $g(x)$ function is monotonically
increasing away from $x=0$, the values of $a(x)$ nearer this point are
most affected.
The leftmost graphic shows only values for $\beta=5$, $\beta=10$ and $\beta=20$ since
for smaller $\beta$ the reconstructed $a$ is indistinguishable from the actual
$a(x)$ at the resolution of the figure.
However, this does also illustrate the point that the diffusion term
$\nabla.(a\nabla u) $ plays a dominant role in the behaviour of the equation.

The experiment was repeated but with $u_0$ taken to be a narrow Gaussian
of unit fixed height and center $x_0$.
A similar effect to the above was observed:
a greater disparity between the reconstructions and the actual $f$
occurred as $x_0$ approached the boundary $x=1$ where the time-trace data
$h(t) = u(1,t)$ was measured.
This is in keeping with fact that the initial value is perturbing the
range condition and this is more pronounced the closer the support of the
perturbation is to the measurement boundary.

In short, nontrivial initial values used either directly or through the
lifting device can be used to extract information on both $a$ and $f$,
but this has a definite limitation which becomes more acute as $u_0(x)$
increases in magnitude and so more significantly affects the range condition.

\section*{Acknowledgments}

\noindent
The work of the first author was supported by the Austrian Science Fund {\sc fwf}
under the grant P30054.
%
%\noindent
The work of the second author was supported in part by the
National Science Foundation through award {\sc dms}-1620138.

Moreover, the authors would like to thank both reviewers for their careful
reading of the manuscript and their valuable comments and suggestions that
have led to an improved version of the paper.

%%%%%%%%%%%%%%%%%%%%%%%%%%%%%%%%%%%%%%%%%%%%

\medskip
% The data information below will be filled by AIMS editorial staff
Received March 2020; revised May 2020.
\medskip

{\it E-mail address: }barbara.kaltenbacher@aau.at\\

\indent {\it E-mail address: }rundell@math.tamu.edu\\


\begin{thebibliography}{99}

\bibitem{AronsonWeinberger:1978} (MR511740) [10.1016/0001-8708(78)90130-5]
\newblock D. G. Aronson and H. F. Weinberger,
\newblock \doititle{Multidimensional nonlinear diffusion arising in population genetics},
\newblock \emph{Advances in Mathematics}, \textbf{30} (1978), 33--76.

\bibitem{Arrhenius:1889}
\newblock R.~A. Arrhenius,
\newblock \"uber die dissociationsw\"arme und den einflu\ss der temperatur auf den dissociationsgrad der elektrolyte,
\newblock \emph{Z. Phys. Chem.}, \textbf{4} (1889), 96--116.

\bibitem{BakerGrave-Morris:1996} (MR1383091) [10.1017/CBO9780511530074]
\newblock G. A. Baker and P.~Graves-Morris,
\newblock \emph{Pad\'e Approximants},
\newblock Cambridge University Press, Cambridge, second edition, 1996.

\bibitem{Borg:1946} (MR15185) [10.1007/BF02421600]
\newblock G.~Borg,
\newblock \doititle{{E}ine {u}mkehrung der {S}turm-{L}iouville {e}igenwertaufgabe},
\newblock \emph{Acta Mathematica}, \textbf{76} (1946), 1--96.

\bibitem{CahnHillard:1958} [10.1002/9781118788295.ch4]
\newblock J.~W. Cahn and J.~E. Hilliard,
\newblock \doititle{Free energy of a nonuniform system. I. Interfacial free energy},
\newblock \emph{The Journal of Chemical Physics}, \textbf{28} (1958), 258--267.

\bibitem{CCPR:1997} (MR1445771) [10.1137/1.9780898719710]
\newblock K.~Chadan, D.~Colton, L.~P\"{a}iv\"{a}rinta and W.~Rundell,
\newblock \emph{An Introduction to Inverse Scattering and Inverse Spectral Problems},
\newblock SIAM Monographs on Mathematical Modeling and Computation, SIAM, 1997.

\bibitem{Cussler:1997}
\newblock E.~L. Cussler,
\newblock \emph{Diffusion: Mass Transfer in Fluid Systems},
\newblock Cambridge University Press., 1997.

\bibitem{ElliottSungmu:1986} (MR855754) [10.1007/BF00251803]
\newblock C. M. Elliott and Z.~Songmu,
\newblock \doititle{On the cahn-hilliard equation},
\newblock \emph{Arch. Rational Mech. Anal}, \textbf{96} (1986), 339--357.

\bibitem{Friedman:1964} (MR0181836)
\newblock A.~Friedman,
\newblock \emph{Partial Differential Equations of Parabolic Type},
\newblock Prentice-Hall, 1964.

\bibitem{Friedman:1969} (MR0445088)
\newblock A.~Friedman,
\newblock \emph{Partial Differential Equations},
\newblock Holt, Rinehart and Winston, New York, 1969.

\bibitem{Gelfand-Levitan:1951} (MR0073805) [10.1090/trans2/001/11]
\newblock I.~M. Gel'fand and B.~M. Levitan,
\newblock \doititle{On the determination of a differential equation from its spectral function},
\newblock \emph{Amer. Math. Soc. Transl.}, \textbf{1} (1951), 253--291.

\bibitem{Grindrod:1996} (MR1423804)
\newblock P. Grindrod,
\newblock \emph{The Theory and Applications of Reaction-Diffusion Equations: Patterns and Waves},
\newblock 2$^{nd}$ edition, Oxford Applied Mathematics and Computing Science Series, The Clarendon Press, Oxford University Press, New York, 1996.

\bibitem{Isakov:1991} (MR1085828) [10.1002/cpa.3160440203]
\newblock V. Isakov,
\newblock \doititle{Inverse parabolic problems with the final overdetermination},
\newblock \emph{Comm. Pure Appl. Math.}, \textbf{44} (1991), 185--209.

\bibitem{Isakov:2006} (MR2193218)
\newblock V. Isakov,
\newblock {Inverse problems for partial differential equations},
\newblock 2$^{nd}$ edition, \emph{Applied Mathematical Sciences}, 127, Springer, New York, 2006.

\bibitem{KaltenbacherRundell:2019b} (MR3975371) [10.1088/1361-6420/ab109e]
\newblock B. Kaltenbacher and W. Rundell,
\newblock \doititle{On an inverse potential problem for a fractional reaction-diffusion equation},
\newblock \emph{Inverse Problems}, \textbf{35} (2019), 065004.

\bibitem{KaltenbacherRundell:2019c} (MR4019539) [10.1088/1361-6420/ab2aab]
\newblock B. Kaltenbacher and W. Rundell,
\newblock \doititle{On the identification of a nonlinear term in a reaction-diffusion equation},
\newblock \emph{Inverse Problems}, \textbf{35} (2019), 115007.

\bibitem{KaltenbacherRundell:2020a} (MR4002155) [10.1016/j.jmaa.2019.123475]
\newblock B. Kaltenbacher and W. Rundell,
\newblock \doititle{Recovery of multiple coefficients in a reaction-diffusion equation},
\newblock \emph{J. Math. Anal. Appl.}, \textbf{481} (2019), 123475.

\bibitem{KaltenbacherRundell:2020b}
\newblock B. Kaltenbacher and W. Rundell,
\newblock The inverse problem of reconstructing reaction-diffusion systems,
\newblock \emph{Inverse Problems}, \textbf{36} (2020), 065011.

\bibitem{Kuttler:17} (MR3839508)
\newblock C. Kuttler,
\newblock Reaction-diffusion equations and their application on bacterial communication,
\newblock in \emph{Disease Modelling and Public Health}, Handbook of Statistics, Elsevier Science, 2017.

\bibitem{Levine:1990} (MR1056055) [10.1137/1032046]
\newblock H.~A. Levine,
\newblock \doititle{The role of critical exponents in blowup theorems},
\newblock \emph{SIAM Review}, \textbf{32} (1990), 262--288.

\bibitem{Lunardi:1995} (MR3012216)
\newblock A.~Lunardi,
\newblock \emph{Analytic Semigroups and Optimal Regularity in Parabolic Problems},
\newblock Modern Birkh{\"a}user Classics, Springer Basel, 1995.

\bibitem{Britanicca:2020}
\newblock E.~A. Mason,
\newblock \emph{Gas, State of Matter},
\newblock Encyclopedia Britannica, 2020.

\bibitem{Murray:2002} (MR1908418)
\newblock J.~D. Murray,
\newblock {Mathematical Biology {I}},
\newblock 3$^{rd}$ edition, \emph{Interdisciplinary Applied Mathematics}, 17, Springer-Verlag, New York, 2002.

\bibitem{Pierce:1979} (MR534419) [10.1137/0317035]
\newblock A. Pierce,
\newblock \doititle{Unique identification of eigenvalues and coefficients in a parabolic problem},
\newblock \emph{SIAM J. Control Optim.}, \textbf{17} (1979), 494--499.

\bibitem{PilantRundell:1987} (MR1108124) [10.1002/num.1690030404]
\newblock M. Pilant and W. Rundell,
\newblock \doititle{Iteration schemes for unknown coefficient problems in parabolic equations},
\newblock \emph{Numer. Methods for P.D.E.}, \textbf{3} (1987), 313--325, 1987.

\bibitem{PilantRundell:1986} (MR829324) [10.1080/03605308608820430]
\newblock M.~S. Pilant and W. Rundell,
\newblock \doititle{An inverse problem for a nonlinear parabolic equation},
\newblock \emph{Comm. Partial Differential Equations}, \textbf{11} (1986), 445--457.

\bibitem{PoschelTrubowitz:1987} (MR894477)
\newblock J. P\"{o}schel and E. Trubowitz,
\newblock {Inverse spectral theory},
\newblock in \emph{Pure and Applied Mathematics}, 130, Academic Press, Inc., Boston, MA, 1987.

\bibitem{RundellSacks:1992b} (MR1166492) [10.1088/0266-5611/8/3/007]
\newblock W. Rundell and P.~E. Sacks,
\newblock \doititle{The reconstruction of {S}turm-{L}iouville operators},
\newblock \emph{Inverse Problems}, \textbf{8}(1992), 457--482.

\bibitem{RundellSacks:1992a} (MR1106979) [10.1090/S0025-5718-1992-1106979-0]
\newblock W. Rundell and P.~E. Sacks,
\newblock \doititle{Reconstruction techniques for classical inverse {S}turm-{L}iouville problems},
\newblock \emph{Math. Comp.}, \textbf{58}(1992), 161--183.

\bibitem{Savchuk:Shalikov:2006} (MR2311614) [10.1007/s11006-006-0204-6]
\newblock A.~M. Savchuk and A.~A. Shkalikov,
\newblock \doititle{On the eigenvalues of the {S}turm-{L}iouville operator with potentials in {S}obolev spaces},
\newblock \emph{Mat. Zametki}, \textbf{80} (2006), 864--884.

\bibitem{ZhangZhi:2017} (MR3671483) [10.1093/imamat/hxx004]
\newblock Z. Zhang and Z. Zhou,
\newblock \doititle{Recovering the potential term in a fractional diffusion equation},
\newblock \emph{IMA J. Appl. Math.}, \textbf{82} (2017), 579--600.



\end{thebibliography}
\end{document}